\documentclass{amsart}
\usepackage[utf8]{inputenc}

\usepackage{colonequals}
\usepackage{stmaryrd}
\usepackage{csquotes}
\usepackage{float}
\usepackage{tikz-cd}
\usepackage{amsmath}%
\usepackage{amsfonts}%
\usepackage{amsthm}
\usepackage{amssymb}%
\usepackage{graphicx}
\usepackage{hyperref}
\usepackage[all,cmtip]{xy}
\usepackage{geometry}
\allowdisplaybreaks
\theoremstyle{plain}
\newtheorem{theorem}{Theorem}[section]
\newtheorem{algorithm}[theorem]{Algorithm}
\newtheorem{lemma}[theorem]{Lemma}
\newtheorem{corollary}[theorem]{Corollary}

\newtheorem{proposition}[theorem]{Proposition}


\theoremstyle{definition}
\newtheorem{example}[theorem]{Example}
\newtheorem{definition}[theorem]{Definition}

\theoremstyle{remark}
\newtheorem{remark}[theorem]{Remark}
\makeatletter
\newcommand{\bigplus}{%
  \DOTSB\mathop{\mathpalette\mattos@bigplus\relax}\slimits@
}
\newcommand\mattos@bigplus[2]{%
  \vcenter{\hbox{%
    \sbox\z@{$#1\sum$}%
    \resizebox{!}{0.9\dimexpr\ht\z@+\dp\z@}{\raisebox{\depth}{$\m@th#1+$}}%
  }}%
  \vphantom{\sum}%
}
\makeatother

\renewcommand{\mod}{\textrm{ mod }}

\newcommand{\N}{\mathbb{N}}

\newcommand{\Z}{\mathbb{Z}}
\newcommand{\Q}{\mathbb{Q}}
\newcommand{\R}{\mathbb{R}}
\newcommand{\A}{\mathbb{A}}

\newcommand{\CC}{\mathbb{C}}

\newcommand{\om}{\omega}

\renewcommand{\P}{\mathbb{P}}
\newcommand{\F}{\mathbb{F}}

\newcommand{\al}{\alpha}

\renewcommand{\O}{\mathcal{O}}

\newcommand{\dd}{\mathrm{d}}
\newcommand{\ep}{\epsilon}

\newcommand{\T}{\mathbb{T}}

\begin{document}
\title{Computing models for quotients of modular curves}
\author{Josha Box}
\maketitle
\begin{abstract}
We describe an algorithm for computing a $\Q$-rational model for the quotient of a modular curve by an automorphism group, under mild assumptions on the curve and the automorphisms, by determining $q$-expansions for a basis of the corresponding space of cusp forms. We also give a moduli interpretation for general morphisms between modular curves.
\end{abstract}
\section{Introduction}
Consider a positive integer $N$ and a subgroup $G\subset \mathrm{GL}_2(\Z/N\Z)$. To the group $G$ we can associate the modular curve $X_G$, which parametrises pairs $(E,\phi)$ up to isomorphism, where $E$ is an elliptic curve and $\phi$ is a ``$G$-level structure'' on $E$ (see Definition \ref{modcurvedef}). We present in this paper an algorithm (Algorithm \ref{algorithm1}) for computing a model for $X_G/\Q$ in the case where $\mathrm{det}(G)=(\Z/N\Z)^{\times}$, $-I\in G$ and $G$ is normalised by $J:=\begin{pmatrix} 1&0\\0&-1\end{pmatrix}$. This algorithm determines $q$-expansions of a basis for the corresponding space of cusp forms, from which the equations can be deduced via Galbraith's techniques \cite{galbraith} when the genus is at least $ 2$. 

Moreover, we can explicitly describe (auto)morphisms of modular curves. For finite groups $\mathcal{A}$ of such automorphisms, we can also determine $X_G/\mathcal{A}$ directly, without computing $X_G$ first. These morphisms include, but, more importantly, are not limited to, Atkin--Lehner involutions. This opens the way for the explicit computation of trees of arbitrary modular curves and their quotients. We have applied this in Section \ref{examplesection} to find models for three level 35 modular curves, as well as the $j$-map on one of them; this has contributed in \cite{box2} to a proof that all elliptic curves over quartic fields not containing $\sqrt{5}$ are modular. 

 The main step to understanding general morphisms between modular curves it to describe their  moduli interpretation. We do this is Section \ref{sec2}, generalising a result of Bruin and Najman \cite[Section 3]{bruin} for $X_0(N)$. 

%We have applied this algorithm to compute canonical models for three modular curves of level 35, presented in Section \ref{examplesection}. These curves have genera 5, 6 and 8, and the congruence subgroups defining them all have level 35. Finding canonical models for these curves was the initial motivation for embarking upon this project, and our results have already been used in \cite{box2} to show that all elliptic curves over quartic fields not containing $\sqrt{5}$ are modular. Along the way it became apparent that a general approach was required due to the complicated nature of these congruence subgroups. 

In Section \ref{sec3}, we develop the algorithm for computing $q$-expansions of a basis of cusp forms with respect to $G$, thus extending previous results dating back to Tingley \cite{tingley}, who in 1975 computed cusp forms on $\Gamma_0(N)$ for $N$ prime. Tingley's results were improved to all $N$ and optimised by Cremona \cite{cremona}, after which Stein \cite{stein} generalised this approach further to the spaces $S_k(\Gamma_0(N),\ep)$, where $\ep$ is a mod $N$ Dirichlet character. The same approach, using modular symbols, does not simply carry over to general congruence subgroups. In Section \ref{qexps}, we describe the scaling issue that occurs, which we solve in subsequent sections using twist operators, an idea due originally to John Cremona, c.f. \cite{annotatedsagecode}. \\

Despite the lack of a general algorithm, models for several more complicated modular curves have been found previously. We mention some of these, as well as their strong implications. Baran \cite{baran}  found models for the curves  $X_{\mathrm{ns}^+}(20)$ and $X_{\mathrm{ns}^+}(21)$, as well as for  the isomorphic curves $X_{\mathrm{ns}^+}(13)$ and $X_{s^+}(13)$ \cite{baran2}. The determination of the integral points of these curves gave new solutions to the class number one problem, while the rational points on the level 13 curves shed light onto Serre's uniformity problem over $\Q$ (see also \cite{serre}).

Derickx, Najman and Siksek \cite{derickx} used a planar model for  $X(\mathrm{b}5,\mathrm{ns}7)$ (defined in Section \ref{examplesection}), to prove that all elliptic curves over cubic fields are modular. This planar model was derived from Le Hung's equations \cite{lehung} for the curve as a fibred product $X_0(5)\times_{X(1)}X_{\mathrm{ns}^+}(7)$. Furthermore, Banwait and Cremona \cite{banwait} determined a model for the exceptional modular curve $X_{S_4}(13)$ by instead computing pseudo-eigenvalues of Atkin--Lehner operators. This allowed them to study the failure of the local-to-global principle for the existence of $\ell$-isogenies of elliptic curves over number fields. Simultaneously, Cremona and Banwait \cite{annotatedsagecode} found a model for the same curve $X_{S_4}(13)$, as well as Baran's curves $X_{\mathrm{ns}^+}(13)$ and $X_{s^+}(13)$ and equations describing the $j$-maps, using his method of modular symbols. This is not published, but available online as a \texttt{Sage} worksheet with annotations by Banwait and Cremona \cite{annotatedsagecode}. 

Given the desire for a more general algorithm for computing models of modular curves, it  may not come as a surprise that, during the author's work on this project, three independent results of similar nature were published -- at least in preprint. Brunault and Neururer \cite{brunault} used Eisenstein series to find an algorithm for computing the spaces of modular forms $M_k(\Gamma,\CC)$ of arbitrary weight and congruence subgroup $\Gamma\subset \mathrm{SL}_2(\Z)$. Zywina \cite{zywina}, on the other hand, generalised the work of Banwait and Cremona \cite{banwait}, using numerical approximation of pseudo--eigenvalues of Atkin--Lehner operators to determine $q$-expansions and models for modular curves. Finally, Assaf \cite{assaf} recently generalised the `classical' strategy of Cremona by defining and successfully utilising modular symbols and Hecke operators on general congruence subgroups to compute Fourier coefficients, at least at primes not dividing the level. Currently, as far as the author is aware, Assaf's algorithm is unable to determine the Fourier coefficients at primes dividing the level for congruence subgroups such as those in Section \ref{examplesection}, which may complicate provable determination of equations satisfied by those modular forms. Zywina can determine all Fourier coefficients, and his method can in fact be used to find a model for $X(\mathrm{b}5,\mathrm{e}7)$ (defined in Section \ref{examplesection}), but not currently for its quotients. 

Our approach instead generalises Cremona's work  \cite{annotatedsagecode} on $X_{S_4}(13)$. We can compute any Fourier coefficient for a basis of cusp forms for any congruence subgroup, without the need for numerical approximation. We have chosen this approach because it is a natural extension of the current methods for determining $q$-expansions of cusp forms on $\Gamma_0(N)$.  This enables us to use the current packages for cusp forms in \texttt{Sage}, making the algorithm relatively easy to implement. Another forte of our approach is that we can directly compute quotients of modular curves by automorphisms. As far as the author is aware, there is currently no other algorithm available that can compute models for the modular curves in Section \ref{examplesection}. 

The \texttt{Sage} and \texttt{Magma} code used for the computations in Section \ref{examplesection} is publicly available at
\[
\texttt{\href{https://github.com/joshabox/modularcurvemodels}{https://github.com/joshabox/modularcurvemodels} .}
\]
Given the existing comprehensive \texttt{Magma} implementations of Assaf \cite{assaf} and Zywina \cite{zywina}, we have not implemented a general version of our algorithm, although parts of our implementation do work more generally. We note that it should certainly be possible to implement the algorithm; in particular, the examples computed in Section \ref{examplesection} do not appear to be in any subcategory of ``easier cases''. The pragmatic reader in search of a model for their modular curve is advised to try Zywina's code first. 

\subsection{Acknowledgements}
When you are in Warwick -- or indeed anywhere in the world -- and need to do computations with modular forms, few would wonder who to ask for advice. The author is extremely grateful to John Cremona for multiple inspiring conversations, and for sharing his unpublished work on level 13 modular curves. In particular, Cremona's idea of using twist operators has been the cornerstone for this project. 

The author also thanks David Loeffler, Jeroen Sijsling and Samir Siksek for their kind advice, and John Cremona and Samir Siksek for valuable feedback on earlier versions of this article. 

We also thank the anonymous referees for their valuable comments.

\section{Morphisms between modular curves}
\label{sec2}
\subsection{Modular Curves not of the standard type}
In the literature, modular curves tend to be described as being determined by a level $N\in \Z_{>0}$ and a subgroup $G\subset \mathrm{GL}_2(\Z/N\Z)$. This is a convenient point of view, since such modular curves have an interpretation as  moduli spaces of elliptic curves with additional structure. 
%For example, if $K$ is a field, then the $K$-rational points of such modular curves $X_G$ correspond 1-1 with $\mathrm{Gal}(\overline{K}/K)$-invariant pairs $(E,\phi)$ of elliptic curves $E$ with ``level structure'' $\phi$, where $\phi$ is a $G$-equivalence class of isomorphisms $\Z/N\Z^2\to E[N]$. 

However, some modular curves do not fit in this framework. The curve $X_0(N)$ -- associated to the group $B_0(N)\subset \mathrm{GL}_2(\Z/N\Z)$ of upper-triangular matrices -- parametrises pairs $(E,C)$ where $E$ is an elliptic curves and $C\subset E$ is a cyclic subgroup of order $N$. This curve admits a well-known involution, called the Atkin--Lehner involution $w_N$, mapping such a pair $(E,C)$ to $(E/C,E[N]/C)$. The quotient curve $X_0(N)/w_N$ does not parametrise elliptic curves with additional structure, but rather certain pairs of elliptic curves with extra structure, and therefore the standard theory of ``moduli problems'' does not apply.

Nonetheless, $X_0(N)/w_N$ does have a moduli interpretation, it is defined over $\Q$, and it is a modular curve in the adelic sense: $(X_0(N)/w_N)_{\CC}=\mathrm{GL}_2^+(\Q)\backslash (\mathrm{GL}_2(\A_f)\times \mathcal{H})/U$, where $\A_f$ denotes the finite ad\`eles, $\mathcal{H}$ is the complex upper half-plane and $U$ is the compact open subgroup of $\mathrm{GL}_2(\A_f)$ generated by $w_N$ and the inverse image of $G$ in $\mathrm{GL}_2(\widehat{\Z})$. 
%Denote by $\A_f$ the finite ad\`eles over $\Q$, and by $\widehat{\Z}\subset \A_f$ the inverse limit of all finite quotients $\Z/N\Z$. The curve $X_0(N)/w_N$ is indeed a modular curve in the adelic sense. Let $U\subset \mathrm{GL}_2(\A_f)$ be the subgroup generated by the inverse image of $B_0(N)$ under $\mathrm{GL}_2(\widehat{\Z})\to \mathrm{GL}_2(\Z/N\Z)$ and the Atkin--Lehner operator $W_N:=\begin{pmatrix} 0 & -1 \\ N & 0\end{pmatrix} \in \mathrm{GL}_2(\A_f)$. This is a compact open subgroup giving rise to the complex modular curve
%\[
%\mathrm{GL}_2^+(\Q)\setminus (\mathrm{GL}_2(\A_f)\times \mathcal{H})/U,
%\]
%where $\mathcal{H}$ denotes the complex upper half-plane. This curve equals $(X_0(N)/w_N)_{\CC}$. 

While Atkin--Lehner involutions may be well understood, more modular curves can arise in this way. Firstly, when $h^2\mid N$ for a non-trivial divisor $h$ of 24, the normaliser of $\Gamma_0(N)$ in $\mathrm{PGL}_2(\Q)$ is generated by more than just the Atkin--Lehner involutions (see Lemma \ref{normaliserlemma}), giving rise to extra automorphisms on $X_0(N)$ (not all defined over $\Q$, however). When $N\in \{40,48\}$, two such automorphisms were explicitly determined by Bruin and Najman \cite{bruin}. When $9\mid N$, one normalising matrix is $\begin{pmatrix} 1 & 1/3 \\ 0 & 1\end{pmatrix}$, giving rise to an automorphism $\al_3$ of order 3, defined over $\Q(\zeta_3)$, such that the group $\langle\al_3\rangle$ generated by $\al_3$ is $\Q$-rational. In particular, this yields a ``new'' morphism of curves over $\Q$, $X_0(N)\to X_0(N)/\langle \al_3\rangle$. 

More types of examples occur on modular curves of mixed level by composing automorphisms. Denote by $G(\mathrm{s}3^+)$ and $G(\mathrm{ns}3^+)$ the normalisers in $\mathrm{GL}_2(\F_3)$ of split and non-split Cartan subgroups respectively. Then $G(\mathrm{s}3^+)\subset G(\mathrm{ns}3^+)$ with index 2. Any matrix in $G(\mathrm{ns}3^+)\setminus  G(\mathrm{s}3^+)$ determines an involution $\phi_3$ on $X_{G(\mathrm{s}3^+)}$. On the level 15 modular curve $X(\mathrm{b}5,\mathrm{s}3^+)$, determined by the intersection of the inverse images of $B_0(5)$ and $G(\mathrm{s}3^+)$ in $\mathrm{GL}_2(\Z/15\Z)$, we then obtain an Atkin--Lehner involution $w_5$ as well as a lift $\psi_3$ of $\phi_3$. These involutions commute and give rise to another involution $\psi_3w_5$, and another modular curve $X(\mathrm{b}5,\mathrm{s}3^+)/\psi_3w_5$.  %Section \ref{examplesection} we define a modular curve $X(\mathrm{b}5,\mathrm{e}7)$ of level 35, which admits the Atkin--Lehner involution $w_5$, as well as an involution $\phi_7$  onsider $X_0(4N)$ for some positive integer $N$. The congruence subgroup $\Gamma_0(2N)$ contains $\Gamma_0(4N)$ with index 2. Let $A$ be a representative for the non-trivial coset of $\Gamma_0(2N)/\Gamma_0(4N)$. Since index 2 subgroups are normal, we obtain from $A$ an involution $X_0(4N)\to X_0(4N)$. We also have the Atkin--Lehner involution $w_{4N}$ on this curve, and we can consider their product $w_{4N}A$. Now the quotient curve $X_0(4N)/\langle w_{4N}A\rangle$ is another modular curve defined over $\Q$, but it is not the quotient of a modular curve $X_G$ by an Atkin--Lehner involution. 
In Section \ref{examplesection} we study a similar example, which the author stumbled upon ``in nature'' (see \cite{box2}).  In order to understand such quotient curves, we first study the moduli interpretation of the automorphisms determined by such matrices. 

\subsection{Modular curves and their moduli interpretation} We use their moduli interpretation to define modular curves over more general base schemes. While we shall not need the description of modular curves as schemes over $\Z[1/N]$ or $\Z[1/N,\zeta_N]$ as defined below, this approach does help us decide the \emph{field} over which modular curves and the Fourier coefficients of their cusp forms are defined. It moreover allows us to prove which morphisms are defined over this field. We give an overview of standard results from Deligne and Rapoport \cite{deligne} and Katz and Mazur \cite{katzmazur}, which we attempt to describe as concretely as possible. 

Let $N\in \Z_{\geq 1}$ be an integer, and choose a primitive $N$th root of unity $\zeta_N:=e^{2\pi i/N}\in \CC$. To define modular curves via their moduli interpretation, we need to consider arbitrary base schemes. Let $S$ be a scheme over $\Z[1/N]$. An \emph{elliptic curve over $S$} is a pair $(E\to S,O)$, where $E\to S$ is a proper smooth map, all of whose fibres are geometrically connected curves of genus 1, and $O$ is a section of $E\to S$. Then $E/S$ obtains the structure of a commutative group scheme. On $E/S$, there is the Weil pairing
\[
e_N:\;\;E[N](S)\times E[N](S) \to \mathbf{\mu}_N(S),\;\; (P,Q)\mapsto e_N(P,Q),
\]
where $\mu_N=\mathrm{Spec}(\Z[X]/(X^N-1))$ is the multiplicative group scheme of $N$th roots of unity. 
To such an elliptic curve $E/S$, we can associate its \emph{$\Gamma(N)$-structures}, defined as the maps
\[
\phi:\; (\Z/N\Z)_S^2\to E[N](S),
\]
such that $E[N]=\sum_{(a,b)\in \Z/N\Z^2}\phi(a,b)$ as effective Cartier divisors. (When $S=\mathrm{Spec}(K)$ for a field $K$ of characteristic coprime to $N$, this means that $\phi(0,1)$ and $\phi(1,0)$ form a basis.) Now suppose that $g\in \mathrm{GL}_2(\Z/N\Z)$. Then $g$ acts on $(\Z/N\Z)_S^2$ by right-multiplication of row vectors, and this is compatible with the Weil pairing in the sense that
\begin{align}
\label{weilpairingdet}
e_N(\phi(a,b),\phi(c,d)) =e_N(\phi(1,0),\phi(0,1))^{\mathrm{deg}(g)} \text{ for each } g=\begin{pmatrix} a & b \\ c & d\end{pmatrix} \in \mathrm{GL}_2(\Z/N\Z).
\end{align}
We consider the functor
\[
\mathcal{F}_N:\;\mathrm{\underline{Sch}}_{\Z[\zeta_N]}\to \mathrm{\underline{Set}},\;\; S\mapsto \{(E/S,\phi)\},
\]
 mapping a scheme $S$ to the isomorphism classes of pairs $(E/S,\phi)$, where $E$ is an elliptic curve over $S$ and $\phi$ is a $\Gamma(N)$-structure on $E/S$. The Weil pairing defines a map of functors $e_N:\; \mathcal{F}_N\to \mu_N$, and we define the subfunctor $
 \mathcal{F}_N^{\mathrm{can}}:\;\mathrm{\underline{Sch}}_{\Z[\zeta_N]}\to \mathrm{\underline{Set}}
 $, mapping $S$ to the set of pairs $(E/S,\phi)\in \mathcal{F}_N(S)$ such that $e_N(\phi(1,0),\phi(0,1))=\zeta_N$. This rigidifies the moduli problem. Now $\mathcal{F}_N^{\mathrm{can}}$ admits a coarse moduli space $Y(N)/\Z[\zeta_N]$, whose compactification $X(N)$ is smooth over $\Z[\zeta_N,1/N]$, as shown e.g. in \cite[Chapter 9]{katzmazur}.
 
 We now consider any subgroup $G\subset \mathrm{GL}_2(\Z/N\Z)$. Its group of determinants $\mathrm{det}(G)$ acts on $\Z[\zeta_N]$ by automorphisms via $\zeta_N\mapsto \zeta_N^a$ for $a\in \mathrm{det}(G)$. We obtain a fixed subring $\Z[\zeta_N]^{\mathrm{det}(G)}\subset \Z[\zeta_N]$. For schemes $S/\Z[\zeta_N]$, the right-action of $ G$ on $(\Z/N\Z)_S^2$ by right-multiplication gives rise to a left-action on $\Gamma(N)$-structures. For $g\in G$ and a $\Gamma(N)$-structure $\phi$, we denote this by $g\cdot \phi$, so that $(g\cdot \phi)(a)=\phi(a\cdot g)$.   Denote the $G$-equivalence class of the $\Gamma(N)$-structure $\phi$ by $[\phi]_G$. 
 
 Given a $\Z[\zeta_N]^{\mathrm{det}(G)}$-scheme $S$ and an elliptic curve $E/S$, we can consider schemes $T/S$ and their base-change $T':=T\times_{\Z[\zeta_N]^{\mathrm{det}(G)}}\Z[\zeta_N]$.
 \begin{definition}\label{modcurvedef}We define the functor
 \[
 \mathcal{F}_G:\;\mathrm{\underline{Sch}}_{\Z[\zeta_N]^{\mathrm{det}(G)}}\to \mathrm{\underline{Set}},\;\; S\mapsto \{(E/S,[\phi]_G\},
 \]
 mapping a scheme $S$ to the set of isomorphism classes of pairs $(E/S,[\phi]_G)$, where $[\phi]_G$ is a $G$-equivalence class of $\Gamma(N)$-structures on $E_{T'}/T'$ for some $T/S$, such that $[\phi]_G$ is ``defined over $S$''. We define $\mathcal{F}_G^{\mathrm{can}}$ as the subfunctor of those pairs $(E/S,[\phi]_G)$ where $e_N(\phi(1,0),\phi(0,1))$ and $\zeta_N$ have the same image in $\mu_N/\mathrm{det}(G)(S)$, or, more concretely, where 
 \[
 e_N(\phi(1,0),\phi(0,1))=\zeta_N^a \text{ for some } a\in \mathrm{det}(G).
 \]
% Equivalently, if $\Z[\zeta_N]^{\mathrm{det}(G)}=\Z[\zeta_N^K]$ for $K\mid N$, then we demand that $e_N(\psi(1,0),\psi(0,1))^{K}=\zeta_N^K$ for any choice of $\psi\in [\phi]_G$. 
As shown in \cite[Chapter 9]{katzmazur}, $\mathcal{F}_G^{\mathrm{can}}$ admits a coarse moduli scheme $Y_G/\Z[1/N,\zeta_N]^{\mathrm{det}(G)}$, whose compactification $X_G$ is smooth. We call $X_G$ the \emph{modular curve associated to $G$}.
\end{definition}
Finally, we mention what it means for $[\phi]_G$ to be ``defined over $S$''. Given an elliptic curve $E/S$, we consider the  functor
\[
(\mathcal{F}_N)_{E/S}:\;\; \underline{\mathrm{Sch}}_{S}\longrightarrow \underline{\mathrm{Set}},\;\; T\mapsto \{\Gamma(N)\text{-structures on } E_T/T\}.
\]
This functor is represented by an $S$-scheme $\mathcal{M}_{E/S}$, meaning that we have bijections $\mathcal{M}_{E/S}(T)\longleftrightarrow (\mathcal{F}_N)_{E/S}(T)$, functorially in $T$. Now $G$ acts on $\mathcal{M}_{E/S}$, and we say that $[\phi]_G$ is \emph{defined over $S$}, when the image of $\phi$ in the $T'$-points $(\mathcal{M}_{E/S}/G)(T')$ of the quotient scheme is in fact in $(\mathcal{M}_{E/S}/G)(S)$. 

\subsection{Notation for modular curves}
We define the \emph{congruence subgroup} $\Gamma_G$ associated to $G\subset \mathrm{GL}_2(\Z/N\Z)$ to be the inverse image under $\mathrm{SL}_2(\Z)\to \mathrm{SL}_2(\Z/N\Z)$ of $G\cap \mathrm{SL}_2(\Z/N\Z)$. For any ring $R$, we denote by $\mathrm{P}\Gamma$ the image of $\Gamma\subset \mathrm{GL}_2(R)$ in $\mathrm{PGL}_2(R)$. Recall that $(Y_G)_{\CC}\simeq \Gamma_G\backslash\mathcal{H}$, where $\mathcal{H}$ is the upper half-plane and $\Gamma_G$ acts by fractional linear transformations. When $N=K\cdot M$, we denote by $G_K$ the image of $G$ in $\mathrm{GL}_2(\Z/K\Z)$. 

By $\mathcal{N}_{\Gamma_G}$ we denote the normaliser of $\mathrm{P}\Gamma_G$ in $\mathrm{PGL}_2^+(\Q)$, where the superscript $+$ means ``with positive determinant'', and by $\mathcal{N}_G\subset \mathcal{N}_{\Gamma_G}$ the subgroup of those $\gamma\in \mathcal{N}_{\Gamma_G}$ satisfying condition (\ref{morphismcondition}), to be defined in Proposition \ref{normalisermoduliprop}. 
We define the following subgroups of $\mathrm{GL}_2(\Z/N\Z)$:
\[
B_0(N):=\begin{pmatrix} * & * \\ 0 & *\end{pmatrix},\;\;B_1(N):=\begin{pmatrix} * & * \\ 0 & 1 \end{pmatrix},\;\; \widetilde{B}_1(N):=\begin{pmatrix} 1 & * \\ 0 & 1 \end{pmatrix},\;\; G(N):=\begin{pmatrix} *&0 \\ 0 & 1 \end{pmatrix},\;\; \widetilde{G}(N):=\left\{I\right\}.
\]
 We denote their congruence subgroups by $\Gamma_0(N)$, $\Gamma_1(N)$, $\widetilde{\Gamma}_1(N)$, $\Gamma(N)$ and $\widetilde{\Gamma}(N)$ respectively. Note that $\widetilde{\Gamma}_1(N)=\Gamma_1(N)$ and $\widetilde{\Gamma}(N)=\Gamma(N)$. The corresponding modular curves are denoted by $X_0(N)$, $X_1(N)$, $\widetilde{X}_1(N)$, $X(N)$ and $\widetilde{X}(N)$. 

We also write $X(\Gamma_G)$ instead of $X_G$ and $\widetilde{X}(\Gamma_G)$ instead of $X_{\widetilde{G}}$ when $G$ is clear from context. For positive integers $K,M$, we define $X(\Gamma_0(M)\cap \Gamma_1(K))$, resp. $\widetilde{X}(\Gamma_0(M)\cap \Gamma_1(K))$, for the curve associated to the intersection of the inverse images of $B_0(M)$ and $B_1(K)$, resp. $B_0(M)$ and $\widetilde{B}_1(K)$, in $\mathrm{GL}_2(\Z/\mathrm{lcm}(K,M)\Z)$. Similarly, define $X(\Gamma_0(M)\cap \Gamma(K))$ and $\widetilde{X}(\Gamma_0(M)\cap \Gamma(K))$. 

The curves $\widetilde{X}(\Gamma_0(M)\cap \Gamma(K))$ and $\widetilde{X}(\Gamma_0(M)\cap \Gamma_1(K))$ will be the protagonists of our story.

\subsection{Morphisms between modular curves and their moduli interpretation}
We first mention two kinds of trivial morphisms.
\begin{itemize}
    \item[(M1)]When $G\subset H\subset \mathrm{GL}_2(\Z/N\Z)$, denote by $\mathcal{F}_H^{\mathrm{can}}|_{\mathrm{\underline{Sch}}_{\Z[\zeta_N]^{\mathrm{det}(G)}}}$ the restriction of $\mathcal{F}_H^{\mathrm{can}}$ to $\underline{\mathrm{Sch}}_{\Z[\zeta_N]^{\mathrm{det}(G)}}$. We obtain a forgetful map of functors $\mathcal{F}_G^{\mathrm{can}}\to \mathcal{F}_H^{\mathrm{can}}|_{\mathrm{\underline{Sch}}_{\Z[\zeta_N]^{\mathrm{det}(G)}}}$ yielding a quotient morphism $X_G\to X_H\times_{\Z[\zeta_N]^{\mathrm{det}(H)}}\Z[\zeta_N]^{\mathrm{det}(G)}$. 
\item[(M2)]Consider any group $G\subset \mathrm{GL}_2(\Z/N\Z)$ and an integer $M$. Define $\pi: \; \mathrm{GL}_2(\Z/MN\Z)\to \mathrm{GL}_2(\Z/N\Z)$. The multiplication-by-$M$ map on $NM$-torsion of elliptic curves and the to-the-$M$th-power map $\mu_{NM}\to \mu_N$ commute with respect to the Weil pairing by (\ref{weilpairingdet}), and define an isomorphism of functors $\mathcal{F}_{\pi^{-1}(G)}^{\mathrm{can}}\to \mathcal{F}_G^{\mathrm{can}}$. We conclude that $X_{\pi^{-1}(G)}=X_G$. 
\end{itemize}
By (M2), any morphism of modular curves can be viewed as a morphism between modular curves of the same level. 
\begin{example}
\label{tildeisoex}
When $G\subset \mathrm{GL}_2(\Z/N\Z)$ and $\Delta\subset \mathrm{det}(G)$ is a subgroup, we can consider the subgroup $H\subset G$ of elements $g\in G$ with $\mathrm{det}(g)\in \Delta$. Then $G$ and $H$ give rise to the same congruence subgroups, hence $(X_G)_{\CC}=(X_H)_{\CC}$. We obtain a morphism of curves $(X_H)_{\Q(\zeta_N)^{\mathrm{det}(H)}}\to (X_G)_{\Q(\zeta_N)^{\mathrm{det}(H)}}$ of the form (M1), which must be an isomorphism.  For example, $\widetilde{X}(\Gamma_0(M)\cap \Gamma_1(K))_{\Q(\zeta_K)}=X(\Gamma_0(M)\cap \Gamma_1(K))_{\Q(\zeta_K)}$ and $\widetilde{X}(\Gamma_0(M)\cap \Gamma(K))_{\Q(\zeta_K)}=X(\Gamma_0(M)\cap \Gamma(K))_{\Q(\zeta_K)}$. 
\end{example}
Suppose that $\gamma\in \mathrm{GL}_2^+(\Q)$ satisfies $\gamma\Gamma_G\gamma^{-1}\subset\Gamma_H$ for some $G,H\subset \mathrm{GL}_2(\Z/N\Z)$. Then $\gamma$ defines a morphism $(X_G)_{\CC}\to (X_{H})_{\CC}$ through its action as a fractional linear transformation on $\mathcal{H}$. This is a morphism defined a priori over $\CC$. We investigate when this morphism is in fact defined over $\Z[\zeta_N]^{\mathrm{det}(G)}$. The following proposition generalises, and was inspired by, \cite[Section 3]{bruin}.

\begin{proposition}
\label{normalisermoduliprop}
Suppose that $G,H\subset \mathrm{GL}_2(\Z/N\Z)$ and consider $\gamma\in \mathrm{GL}_2^+(\Q)$. Assume (after scaling) that $\gamma$ has integral coefficients, with $\delta\colonequals \mathrm{det}(\gamma)$. Denote by $\pi: \; \mathrm{GL}_2(\Z/N\delta \Z)\to \mathrm{GL}_2(\Z/N\Z)$ the reduction map. If 
\begin{align}
\label{morphismcondition}
\gamma\pi^{-1}(G) \subset \pi^{-1}(H)\gamma
\end{align}
then $\mathrm{det}(G)\subset \mathrm{det}(H)$ and
 $\gamma$ determines a morphism $\theta_{\gamma}: X_G\to X_H\times_{\Z[\zeta_N]^{\mathrm{det}(H)}} \Z[\zeta_N]^{\mathrm{det}(G)}$.

Moreover, $\gamma\Gamma_G\gamma^{-1} \subset \Gamma_H$ and $(\theta_{\gamma})_{\CC}:\; (X_G)_{\CC}\to (X_H)_{\CC}$ corresponds to the action of $\gamma$ on $\Gamma_G\backslash \mathcal{H}$ followed by projection onto $\Gamma_H\backslash \mathcal{H}$ .
\end{proposition}
%\begin{remark} \label{scalingremark}
%We note that (\ref{morphismcondition}) is indeed invariant under scaling $\gamma$ by integer values. 
%\end{remark}
\begin{proof}
We shall replace $G$ by $\pi^{-1}(G)$ and construct a morphism $\mathcal{F}^{\mathrm{can}}_{\pi^{-1}(G)}\to \mathcal{F}^{\mathrm{can}}_H|_{\Z[\zeta_N]^{\mathrm{det}(G)}}$. We first unravel formalities to make concrete what this means. The functor $\mathcal{F}^{\mathrm{can}}_G$ is the functor associated to the ``moduli problem'' $([\Gamma(N)]/G)^{\mathrm{can}}$, as defined in \cite[9.4]{katzmazur}. There are two constructions at work here: the quotient by $G$ (see \cite[Chapter 7]{katzmazur}), and the $\mathrm{can}$-construction (see \cite[Chapter 9]{katzmazur}). It follows from (M2), \cite[Corollary 9.1.9]{katzmazur} and \cite[Theorem 7.1.3]{katzmazur} that it suffices to define a $(\pi^{-1}G,H)$-equivariant map $\al\colon \mathcal{F}_{N\delta}\to \mathcal{F}_N$ given by
\[
\al:\; (E/S,\phi: (\Z/N\delta\Z)_S\to E[N\delta])\mapsto(E_{\gamma}/S,\phi_{\gamma}: (\Z/N\Z)_S\to E_{\gamma}[N])
\]
that gives rise to a commutative diagram
\[
\begin{tikzcd}
\mathcal{F}_{N\delta} \arrow[d, "e_{N\delta}"'] \arrow[rr, "\al"'] &  & \mathcal{F}_N \arrow[d, "e_N"] \\
\mu_{N\delta} \arrow[rr]                                           &  & \mu_{N}                       
\end{tikzcd}
\]
 where the bottom map is defined by $\zeta_{N\delta}\mapsto \zeta_N=\zeta_{N\delta}^{\delta}$. By $(\pi^{-1}G,H)$-equivariant we mean that  $E_{\gamma}/S$ depends only on $[\phi]_{\pi^{-1}(G)}$, and moreover for each $(E,\phi)\in \mathcal{F}_{N\delta}$ and each $g\in \pi^{-1}(G)$ there exists $h\in H$ such that $(g\cdot \phi)_{\gamma}=h\cdot \phi_{\gamma}$ and $e_N(\al(E/S,\phi))^{\mathrm{det}(h)}=e_{N\delta}(E/S,\phi)^{\delta\mathrm{det}(g)}$.

 So we consider $E/S$ and a $\Gamma(N\delta)$-structure $\phi:\; (\Z/N\delta \Z)_S^2\to E[N\delta]$. %such that $\psi(0,\delta)=\phi'(0,1)$ and $\psi(\delta,0)=\phi'(1,0)$ for some $\phi' \in [\phi]_G$.
 %We note that many such choices of $\psi$ exist, lifting $\phi\circ g$ for some $g\in G$, but if $\psi'$ is another such choice, then there is $g\in \pi_N^{-1}(G)$ such that  $\psi'=\psi\circ g$.
 %
 For $M\in \Z_{\geq 1}$, denote by $\mathrm{Im}_M(\gamma)$ the image of the action of $\gamma$ on $(\Z/M\Z)_S^2$ by right-multiplication. We then define $C_{\gamma}\colonequals N\phi(\mathrm{Im}_{N\delta}(\gamma))$. Due to (\ref{morphismcondition}), the subgroup $C_{\gamma}$ depends only on $[\phi]_{\pi^{-1}(G)}$. The size of $C_{\gamma}(S)$ is $\delta$. We define $E_{\gamma}\colonequals E/C_{\gamma}$, and note that $E_{\gamma}[N]=[N]^{-1}C_{\gamma}/C_{\gamma}$, where $[N]$ denotes multiplication-by-$N$ on $E$.  Next, we define the map
\begin{equation}
\label{torsionmap}
\begin{tikzcd}
\phi_{\gamma}:\;(\Z/N\Z)_S^2 \arrow[rr, "\overline{\gamma}"] &  & \mathrm{Im}_{N\delta}(\gamma)/N\mathrm{Im}_{N\delta}(\gamma) \arrow[rr, "\overline{\phi}"] &  & {[N]^{-1}C_{\gamma}/C_{\gamma}=E_{\gamma}[N]}.
\end{tikzcd}
\end{equation}
Suppose $\gamma=\begin{pmatrix} a & b \\ c & d\end{pmatrix}$. The map denoted by $\overline{\gamma}$ is given by $(1,0)\mapsto (a,b)$ and $(0,1)\mapsto (c,d)$, while the map denoted by $\overline{\phi}$ is determined by $\phi$. 

Now consider $g\in \pi^{-1}(G)$. By (\ref{morphismcondition}), we find $h\in \pi^{-1}(H)$ such that $ h\gamma =\gamma g$. For a matrix $\eta$, denote by $m_{\eta}$ the right-multiplication map by $\eta$. Then 
\[
\overline{\phi}\circ m_g \circ m_{\overline{\gamma}} = \overline{\phi}\circ m_{\overline{\gamma}g}=\overline{\phi}\circ m_{\overline{\gamma}}\circ m_{\pi(h)} 
\]
so that indeed $(g\cdot \phi)_{\gamma}=\pi(h)\cdot\phi_{\gamma}$, as desired.  The standard properties of the Weil pairing under isogenies show that
\[
e_N(\phi_{\gamma}(1,0),\phi_{\gamma}(0,1))=e_{N\delta}(\phi(1,0),\phi(0,1))^{\delta},
\]
leaving us to show only that $\mathrm{det}(g)\equiv\mathrm{det}(h) \mod N$. For this, we note that from the matrix $\gamma$ of determinant $\delta$, we have constructed an \emph{invertible} linear map $\overline{\gamma}:\; (\Z/N\Z)^2_S\to \mathrm{Im}_{N\delta}(\gamma)/N\mathrm{Im}_{N\delta}(\gamma)$, and $m_{\overline{\gamma}}\circ m_{\pi(h)} = m_{g}\circ m_{\overline{\gamma}}$ thus implies $\mathrm{det}(h)\equiv \mathrm{det}(g)\mod N$. It follows that $\gamma$ indeed determines a morphism of functors $\mathcal{F}_{\pi^{-1}(G)}^{\mathrm{can}}\to \mathcal{F}_H^{\mathrm{can}}|_{\Z[\zeta_N]^{\mathrm{det}(G)}}$, and thus of their compactified coarse moduli spaces. 

Next, we consider $S=\mathrm{Spec}(\CC)$. We start by showing that $\gamma\Gamma_G\gamma^{-1}\subset \Gamma_H$. By (\ref{morphismcondition}), for each $s\in \Gamma_G$ there exists $s'\in \Gamma_H$ such that $s'\gamma\equiv \gamma s \mod N\delta$. Then $s'\gamma-\gamma s$ has all entries divisible by $\mathrm{det}(\gamma)$, and consequently $\gamma s \gamma^{-1}\in \mathrm{SL}_2(\Z)$.  Hence  $u=\gamma s\gamma^{-1}(s')^{-1}$ is in $\mathrm{SL}_2(\Z)$ and reduces to the identity mod $N$, so $u\in \Gamma_H$. We conclude that $\gamma s \gamma^{-1} = us' \in \Gamma_H$, as desired. 

 Recall that we have a bijection $\Gamma_G\backslash \mathcal{H}\to \mathcal{F}_G(\mathrm{Spec}(\CC))$ given by 
\[
\Gamma_G\tau\mapsto (E_{\tau},[\phi_{\tau}]_G),\;\; E_{\tau}=\frac{\CC}{\tau\Z\oplus  \Z},\;\; \phi_{\tau}(1,0)=\frac{\tau}{N},\; \phi_{\tau}(0,1)=\frac{1}{N},
\]
and similarly for $H$.
We check that the action of $\gamma$ just defined corresponds under this bijection to the action of $\gamma$ on $\mathcal{H}$ as a fractional linear transformation. So we consider $\tau\in \mathcal{H}$ and $(E_{\tau},[\phi_{\tau}]_G)\in \mathcal{F}_G(\mathrm{Spec}(\CC))$. First, we note that
\[
C_{\gamma}=\left\langle \frac{a\tau}{\delta}+\frac{b}{\delta},\frac{c\tau}{\delta}+\frac{d}{\delta}\right\rangle.
\]
As $\mathrm{det}(\gamma)=\delta$, we have $\tau\Z\oplus  \Z \subset \frac{1}\delta((a\tau+b)\Z\oplus (c\tau+d)\Z)$, and we obtain an isomorphism
\begin{align}
\label{ciso}
\beta:\; E_{\tau}/C_{\gamma}=\CC/\left(\left(\frac{a\tau}{\delta}+\frac{b}{\delta}\right)\Z\oplus\left(\frac{c\tau}{\delta}+\frac{d}{\delta}\right)\Z \right)\longrightarrow E_{\gamma(\tau)},
\end{align}
defined by $z\mapsto \delta\cdot z /(c\tau+d)$. Finally,  $(1,0)$ transforms under $\beta\circ (\phi_{\tau})_{\gamma}$ as follows:
\[
\beta\circ (\phi_{\tau})_{\gamma}:\; (1,0)\mapsto (a,b)\mapsto \frac{a\tau}{N\delta}+\frac{b}{N\delta}\mapsto \frac{\gamma(\tau)}{N}
\]
and similarly it maps $(0,1)$ to $1/N$, so that indeed $(E_{\tau},[\phi_{\tau}]_G)$ is mapped to $(E_{\gamma(\tau)},[\phi_{\gamma(\tau)}]_H)$.
\end{proof}
\begin{remark}
\label{mapremark}
Suppose that $\gamma\in \mathrm{GL}_2(\Z/N\Z)$ normalises $G\subset \mathrm{GL}_2(\Z/N\Z)$. Then the prescription $(E/S,[\phi]_G)\mapsto (E/S,[\gamma\cdot \phi]_G)$  determines a morphism of functors $\mathcal{F}_G\to \mathcal{F}_G$, and if moreover $\mathrm{det}(\gamma)\in \mathrm{det}(G)$, this specialises to a map $\mathcal{F}_G^{\mathrm{can}}\to \mathcal{F}_G^{\mathrm{can}}$. We obtain an automorphism of $X_G$. However, this does not correspond to the action of $\gamma$ on $\mathcal{H}$ by fractional linear transformations.

Instead, we can choose $g\in G$ such that $\mathrm{det}(g)=\mathrm{det}(\gamma)$. Any lift of $g^{-1}\gamma$ to $\mathrm{SL}_2(\Z)$ then normalises $\Gamma_G$ and satisfies the conditions of Proposition \ref{normalisermoduliprop}. It determines the same morphism $X_G\to X_G$ as $\gamma$ did.  
\end{remark}

\begin{example}
\label{alexample}
We consider the curve $X_0(N)/\Z[1/N]$. It is well-known that $W_N=\begin{pmatrix} 0 & -1\\N & 0\end{pmatrix}$ normalises $\Gamma_0(N)$ and $\Gamma_1(N)$. Here we have $\delta=N$ and $\mathrm{det}(G)=(\Z/N\Z)^{\times}$. We verify that
\[
W_N\pi^{-1}(B_0(N))=\left \{\begin{pmatrix} Na & b\\ Nc& Nd\end{pmatrix} \in M_2(\Z/N^2\Z)\mid b,c\in (\Z/N^2\Z)^{\times}\right\} = \pi^{-1}(B_0(N))W_N,
\]
so that $W_N$ indeed defines a morphism of $\Z[1/N]$-schemes $X_0(N)\to X_0(N)$. We similarly find that $W_N$ defines a morphism of $\Z[1/N,\zeta_N]$-schemes $\widetilde{X}_1(N)\to \widetilde{X}_1(N)$. 
\end{example}
\begin{example}
\label{ex27}
Remark \ref{mapremark} actually helps us find examples. This situation occurs when we have two groups $G_1\subset G_2\subset \mathrm{GL}_2(\Z/N\Z)$ such that $G_1$ is normal in $G_2$ and $\mathrm{det}(G_1)=\mathrm{det}(G_2)$. Any lift $\gamma\in \Gamma_{G_2}\setminus \Gamma_{G_1}$ then determines an automorphism $\theta_{\gamma}$ on $X_{G_1}$.

A concrete example of this is $B_1(N)\subset B_0(N)$.

Other examples occur when $[\Gamma_{G_1}:\Gamma_{G_2}]=2$ and $\mathrm{det}(G_1)=\mathrm{det}(G_2)$. Then also $[G_1:G_2]=2$, so that $G_2$ is normal in $G_1$. Then $\theta_{\gamma}$ (for any $\gamma\in \Gamma_{G_2}\setminus \Gamma_{G_1}$) is the involution on $X_{G_1}$ such that $X_{G_1}/\theta_{\gamma} = X_{G_2}$. 

Examples of this are $B_0(4)\subset \pi^{-1}(B_0(2))$, where $\pi: \mathrm{GL}_2(\Z/4\Z)\to \mathrm{GL}_2(\Z/2\Z)$,  $G(\mathrm{s}3^+)\subset G(\mathrm{ns}3^+)$ (the normalisers of split and non-split Cartan subgroups in $\mathrm{GL}_2(\F_3)$ respectively), and the inclusion $G(\mathrm{e7})\subset G(\mathrm{ns}7^+)$ described in Section \ref{examplesection}.
\end{example}

Next, we lift automorphisms at level $M$ to higher levels $KM$ when $K$ is coprime to $M$. This is important for understanding Atkin--Lehner operators at mixed level. 
\begin{lemma}
\label{coprimelemma}
Suppose that $N=KM$ with $\mathrm{gcd}(K,M)=1$, and we have $G_K\subset \mathrm{GL}_2(\Z/K\Z)$ and $G_M,H_M\subset \mathrm{GL}_2(\Z/M\Z)$. Consider $\gamma\in \mathrm{GL}_2(\Q)^+$ with integral coefficients and determinant $\delta\in \Z$ coprime to $K$, satisfying (\ref{morphismcondition}) of Proposition \ref{normalisermoduliprop} for $G_M$ and $H_M$.  Consider also $\eta\in \mathrm{GL}_2(\Z/K\Z)$ of determinant $\delta \mod K$ normalising $G_K$. 

Define $G=\pi_K^{-1}(G_K)\cap \pi_M^{-1}(G_M)$ and $H=\pi_K^{-1}(G_K)\cap \pi_M^{-1}(H_M)$, where $\pi_K:\mathrm{GL}_2(\Z/N\Z)\to \mathrm{GL}_2(\Z/K\Z)$ and $\pi_M: \mathrm{GL}_2(\Z/N\Z)\to \mathrm{GL}_2(\Z/M\Z)$.

Then $\gamma$ and $\eta$ determine a morphism $X_G\to X_H$ such that the diagram 
\[
\begin{tikzcd}
X_G \arrow[rr] \arrow[d] &  & X_H \arrow[d] \\
X_{G_M} \arrow[rr]       &  & X_{H_M}      
\end{tikzcd}
\]
commutes. This morphism depends only on $\gamma$ and $\eta (G_K\cap \mathrm{SL}_2(\Z/K\Z))\subset \mathrm{GL}_2(\Z/K\Z)$. 
\end{lemma}
\begin{proof}
Recall from the proof of Proposition \ref{normalisermoduliprop} that $\theta_{\gamma}: X_{G_M}\to X_{H_M}$ is only determined by its image $\overline{\gamma}$ in $\mathrm{GL}_2(\Z/\delta M\Z)$.  As $\mathrm{gcd}(K,M\delta)=1$, we can thus find $\al\in M_2(\Z)$ of $\mathrm{det}(\al)=\delta$ lifting both $\eta$ and $\overline{\gamma}$ (and the image of $\al$ in $M_2(\Z/\delta MK\Z)$ is uniquely determined by $\overline{\gamma}$ and $\eta$). Denote by $\pi\colon \mathrm{GL}_2(\Z/\delta M\Z)\to \mathrm{GL}_2(\Z/M\Z)$ the natural map. Then $\al$ satisfies $\al \pi^{-1}(G_M)\subset \pi^{-1}(H_M)\al$ and $\al G_K=G_K\al$. As $\mathrm{gcd}(M\delta,K)=1$, we conclude that also $\al G\subset H\al$, as desired.  The commutativity of the diagram follows by construction. Finally, each $\widetilde{\beta}\in G_K\cap \mathrm{SL}_2(\Z/K\Z)$ can be lifted to $\beta\in \Gamma_{G_K}$ that is the identity mod $\delta M$, and therefore acts trivially on $X_G$. The morphism determined by $\gamma$ and $\eta\widetilde{\beta}$ is $\theta_{\al}\circ \theta_{\beta}=\theta_{\al}$.
\end{proof}
Note that we do not assume in the lemma that $\eta\in G_K$. However, if there exists $\eta\in G_K$ of determinant $\delta\mod K$, then the map $X_G\to X_H$ determined by $\gamma$ and $\eta$ is independent of the choice of $\eta\in G_K$, and we call it \emph{the} lift of $X_{G_M}\to X_{H_M}$ to $X_G\to X_H$. When such $\eta\in G_K$ does not exist, the obtained map really depends on the choice of $\eta$. This distinction becomes apparent when considering the Atkin--Lehner morphisms on $X_0(KM)$ and $\widetilde{X}_1(KM)$ determined by $W_M$. 
\begin{definition}\label{aldef}
Consider again $N=KM$ with $\mathrm{gcd}(M,K)=1$, and a group $G_K\subset \mathrm{GL}_2(\Z/K\Z)$ which is normalised by $\eta\in \mathrm{GL}_2(\Z/K\Z)$ with $\mathrm{det}(\eta)=M\mod K$. Consider $G:=\pi_K^{-1}(G_K)\cap \pi_M^{-1}(B_0(M))\subset \mathrm{GL}_2(\Z/KM\Z)$ or $G:=\pi_K^{-1}(G_K)\cap \pi_M^{-1}(\widetilde{B}_1(M))$, where $\pi_K$ and $\pi_M$ are defined as in Lemma \ref{coprimelemma}. By Lemma \ref{coprimelemma},  $W_M$ and $\eta$ define an automorphism on $X_G$, which we call an \emph{Atkin--Lehner morphism at $M$}. When $\eta\in G_K$, we call it \emph{the} Atkin--Lehner \emph{involution} at $M$ and denote it by $w_M$. 
\end{definition}

\begin{example}\label{atkinlehnerex}
 For coprime integers $K$, $M$, define  $W_M(x,y,z,w):=\begin{pmatrix} Mx & y \\ MKz & Mw \end{pmatrix}$, where $x,y,z,w\in \Z$ satisfy $\mathrm{det}(W_M(x,y,z,w))=M$. Its mod $K$ reduction is in $B_0(K)$, and moreover 
 \[
 W_M(x,y,z,w)=\begin{pmatrix} -y & x \\ -Mw & zK\end{pmatrix}\begin{pmatrix} 0 & -1 \\ M & 0\end{pmatrix},
 \]
 so its reduction mod $M^2$ is $\gamma W_M$, where $\gamma\in \mathrm{GL}_2(\Z/M^2\Z)$ reduces into $B_0(M)$ mod $M$. So indeed $W_M(x,y,z,w)$ defines the Atkin--Lehner involution $w_M$ on $X_0(KM)$ (which is independent of the choice of $x,y,z,w\in \Z$), and an Atkin--Lehner morphism $w_M(x,y,z,w)$ on  $\widetilde{X}_1(KM)$ because $B_0(K)$ normalises $\widetilde{B}_1(K)$ (dependent on the choice of $x,y,z,w\in \Z$). 
 \end{example}

 For $K\in \Z$, we define
 \[
 \gamma_K:=\begin{pmatrix} K & 0 \\ 0 & 1\end{pmatrix}.
 \]
 The action of $\gamma_K$ on $\mathcal{H}$ often leads to interesting morphisms between modular curves. 
\begin{example}
 Let $p$ be a prime. The \emph{ split Cartan subgroup} of $\mathrm{GL}_2(\F_p)$ is the group $G(\mathrm{s}p)$ of diagonal matrices. We interpret this as a group of level $p^2$ by considering its inverse image in $\mathrm{GL}_2(\Z/p^2\Z)$. We then apply Proposition \ref{normalisermoduliprop} with $\gamma_p$ to find a morphism 
 \[
 \phi:\;X_0(p^2)\longrightarrow X(\mathrm{s}p)
 \]
 defined over $\Q$, where $X(\mathrm{s}p):=X_{G(\mathrm{s}p)}$. By considering the congruence subgroups, we see that $\phi$ must be an isomorphism. Or, alternatively, the inverse is defined by $p\cdot \gamma_p^{-1}$. 
 
 On $X_0(p^2)$ we have the involution $w_{p^2}$ defined by $W_{p^2}$, see Example \ref{alexample}. On $X(\mathrm{s}p)$ this corresponds under $\phi$ to the involution defined by the matrix $i:=\frac{1}{p}\gamma_p W_{p^2}\gamma_p^{-1}=\begin{pmatrix} 0 & 1 \\ -1 & 0\end{pmatrix}$. This matrix is in fact invertible mod $p$. Denote by $G(\mathrm{s}p^+)\subset\mathrm{GL}_2(\F_p)$ the group generated by $i$ and $G(\mathrm{s}p)$. This is the normaliser of $G(\mathrm{s}p)$, and it defines a modular curve $X(\mathrm{s}p^+)$, a degree 2 quotient of $X(\mathrm{s}p)$. We conclude that $\phi$ descends to an isomorphism $X(\mathrm{s}p^+)\simeq X_0(p^2)/w_{p^2}$ over $\Q$, a fact also observed in \cite[p. 555]{conrad}.
\end{example}
\begin{example}
\label{isoex}
Consider again a prime $p$, and positive integers $b, a$. Define $c:=\mathrm{max}(a,b)$. As in the previous example, $\gamma_{p^a}$ defines an isomorphism
\[
\widetilde{X}(\Gamma_0(p^b)\cap \Gamma(p^a))\simeq \widetilde{X}(\Gamma_0(p^{c+a})\cap \Gamma_1(p^a)).
\]
%The intersection of the inverse images of $B_0(p^b)$ and $\widetilde{G}(p^a)$ in $\mathrm{GL}_2(\Z/p^c\Z)$ defines a group $\widetilde{G}$.  As in the previous example, $\gamma_{p^a}$ defines by Proposition \ref{normalisermoduliprop} an isomorphism
%\[
%X_{\widetilde{G}}\simeq X_{B_0(p^{c+a})\cap B_1(p^a)}
%\]
%of curves over $\Q(\zeta_{p^a})$. 
For any $K,M\in \Z_{\geq 1}$ with $L:=\mathrm{gcd}(K,M)$, we thus deduce from Lemma \ref{coprimelemma} that $\gamma_K$ defines an isomorphism
\[
\widetilde{X}( \Gamma_0(M)\cap \Gamma(K))\simeq \widetilde{X}(\Gamma_0(MK^2/L)\cap \Gamma_1(K))
\]
of curves over $\Q(\zeta_K)$.
%Now let $G$ be the intersection of the inverse images of $B_0(p^b)$ and$G(p^a)$. Then by Proposition \ref{normalisermoduliprop}, the identity matrix defines a morphism 
%\[
%X_{G}\to X_{\widetilde{G}}\times_{\Z[1/N]}\Z[1/N,\zeta_{p^a}],
%\]
%which must be an isomorphism of curves over $\Q(\zeta_{p^a})$. We conclude that $\gamma_{p^a}$ defines an isomorphism $(X_{G})_{\Q(\zeta_{p^a})}\to (X_{B_0(p^{c+a})\cap B_1(p^a)})_{\Q(\zeta_{p^a})}$. 
%
%
\end{example}

%\begin{lemma}
%Let $\Gamma\subset \mathrm{SL}_2(\Z)$ be a congruence subgroup, and $\mathcal{N}$ its normaliser in $\mathrm{GL}_2^+(\Q)$. The images $\mathrm{P}\Gamma$ and $\mathrm{P}\mathcal{N}$ of $\Gamma$ and $\mathcal{N}$ in $P\mathrm{GL}_2(\Q)$  have finite index  $[\mathrm{P}\Gamma:\mathrm{P}\mathcal{N}]$ is finite. 
%\end{lemma}
%Atkin--Lehner involutions act on modular curves over $\Q$ because they act on moduli. But on the full moduli, over schemes $S$, they are not involutions, because matrices of the form $\lambda \mathrm{I}$ for $\lambda\in \Q^{\times}$ do not act non-trivially over schemes. So there may not be a nice scheme whose $S$-points correspond to moduli for the quotient. However, over $\Q$ the Atkin--Lehner involutions define a finite automorphism group, and we can construct the quotient. Would like to show that elements of $\mathcal{N}(\Gamma_G)$ act on the scheme $X_G$ over $\Z[1/N,\mu_n^{\mathrm{det}(G)}]$. For this, we need a moduli argument. 

\section{Operators on spaces of cusp forms}
\label{sec3}
\subsection{The regular 1-forms on a modular curve}
\label{sec31}
Let $G$ be a subgroup of $\mathrm{GL}_2(\Z/N\Z)$. From now on, we shall only be concerned with curves over fields, and denote by $X_G$ the modular curve associated to $G$, base changed to $\Q(\zeta_N)^{\mathrm{det}(G)}$. 

 For $\gamma\in \mathrm{SL}_2(\Z/N\Z)$, we see that any lift of $\gamma$ to $\mathrm{SL}_2(\Z)$ acts on $X(N)$ as the automorphism $\theta_{\gamma}$, by Proposition \ref{normalisermoduliprop}. Next, consider $a\in (\Z/N\Z)^{\times}$ and the matrix $\gamma_a:=\begin{pmatrix} a & 0 \\ 0 & 1\end{pmatrix}$.  As in Remark \ref{mapremark}, the prescription $(E/S,\phi)\mapsto (E/S,\phi\circ \gamma_a)$ determines a morphism of functors $\mathcal{F}_N\to \mathcal{F}_N$. However, unless $a=1$, we have $\mathrm{det}(\gamma_a)\notin \mathrm{det}(\{1\})$ so that $\gamma_a$ does not determine a morphism $\mathcal{F}_N^{\mathrm{can}}\to \mathcal{F}_N^{\mathrm{can}}$. Define the functor $\mathcal{F}_N^{\mathrm{can},a}$ by instead mapping a scheme $S$ to those pairs $(E/S,\phi)$, where $e_N(\phi)=\zeta_N^a$. This similarly has a coarse moduli space $\widetilde{X}(N)^a/\Q(\zeta_N)$, which is the base change of $\widetilde{X}(N)$ by Galois conjugation $\sigma_a:\Q(\zeta_N)\to \Q(\zeta_N),\;\zeta_N\mapsto\zeta_N^a$. Then $\gamma_a$ does determine a morphism $\widetilde{X}(N)\to \widetilde{X}(N)^a$ \emph{of curves over $\Q(\zeta_N)$}. Composing this map with base change by $\sigma_a^{-1}$, we obtain a map of \emph{schemes} $\theta_{a}:\;X(N)\to X(N)$, whose corresponding map on function fields is merely a morphism of $\Q(\zeta_N)^{\langle \sigma_a\rangle}$-algebras. 

%We give a short introduction to modular curves akin to the more detailed description in \cite[Section 6]{zywina}. %Denote by $\mathcal{H}$ the complex upper half-plane, and consider an integer $N\geq 1$. The group $\mathrm{SL}_2(\Z)$ -- containing the congruence subgroup $\Gamma(N)$ -- acts on $\mathcal{H}$ by fractional linear transformations, and the quotient $\Gamma(N)\setminus \mathcal{H}$ can be extended to a compact Riemann surface $X(N)$. 
Each function $f\in \Q(\zeta_N)(\widetilde{X}(N))$ has a Laurent series expansion around the infinity cusp (which is a $\Q(\zeta_N)$-rational point). Denote by $q_N(\tau)=e^{2\pi i \tau/N}$ a uniformiser at this cusp, and write the expansion of $f$ in its completed local ring as $f=\sum_{n\geq -m} a_n(f) q_N^n$, where $m\in \Z_{>0}$ and each $a_n(f)\in \Q(\zeta_N)$.  As also shown by  Shimura \cite[Proposition 6.9]{shimura},  the maps just described yield a right action $\circ$ of $\mathrm{GL}_2(\Z/N\Z)$ on $\Q(\zeta_N)(\widetilde{X}(N))$ such that for each $f=\sum_{n\geq -m} a_n q_N^n\in \Q(\zeta_N)(X(N))$: 
\begin{itemize}
    \item[(i)] for each $\gamma\in \mathrm{SL}_2(\Z/N\Z)$ we have $f\circ \gamma = \theta_{\widetilde{\gamma}}^*(f)$, where $\widetilde{\gamma}$ is a lift of $\gamma$ to $\mathrm{SL}_2(\Z)$, and
    \item[(ii)] $f\circ \gamma_a=\theta_a^*(f)=\sum_{n\geq -m} \sigma_a(a_n)q_N^n$. 
\end{itemize}

From now on, we suppose that $G\subset \mathrm{GL}_2(\Z/N\Z)$ satisfies 
\[
\mathrm{det}(G)=(\Z/N\Z)^{\times} \text{ and }-I\in G.
\]
The first condition means we consider only curves defined over $\Q$. Then, again by Shimura's work \cite{shimura}, the fixed field $\Q(\zeta_N)(\widetilde{X}(N))^{G}$ defines an irreducible projective curve over $\Q$, which is simply $X_G$. 

We denote by $S_k(\Gamma,K)$ the space of weight $k$ cusp forms with respect to $\Gamma$ whose Fourier coefficients all lie in $K$. 

The action of $\mathrm{GL}_2(\Z/N\Z)$ on $\widetilde{X}(N)$ gives rise to an action on its sheaf of regular 1-forms, which in turn corresponds to modular forms. 

\begin{proposition}
\label{canonicalmodelprop}
The map $f\mapsto f(q_N) (\mathrm{d}q_N)/q_N$ defines an isomorphism 
\[
S_{2}(\Gamma(N),\Q(\zeta_N))^G \simeq H^0(X_G,\Omega).
\]
Here
 $\Omega$ is the sheaf of regular 1-forms on $X_G$.
\end{proposition}
\begin{proof}
By definition of $S_2(\Gamma(N),\Q(\zeta_N))$, the map $S_2(\Gamma(N),\Q(\zeta_N))\to H^0(X(N),\Omega_{X(N)})$ is well-defined, an isomorphism and $G$-equivariant. We then take $G$-invariants, and note that $H^0(X_G,\Omega)=H^0(X(N),\Omega)^G$ because $\Q(X_G)=\Q(\zeta_N)(\widetilde{X}(N))^G$. For more details, see \cite[Lemma 6.5]{zywina}.  %It was proved in detail in \cite[Lemma 6.5]{zywina}.
\end{proof}

%\begin{remark}We stress the importance of using the the full subgroup $G\subset \mathrm{GL}_2(\Z/N\Z)$, as opposed to working with just $G\cap \mathrm{SL}_2(\Z/N\Z)$, in order to obtain a model over the rationals. In the well-known case where $G=B_0(N)$ is the group of upper-triangular matrices, all modular forms fixed by $G$ have rational coefficients, so it suffices to compute $S_2(\Gamma(N),\Q)^{\Gamma_0(N)}$. In general, however, the cusp forms fixed by $G$ do \emph{not} have rational coefficients, and it is necessary to study the action of the full group $G$ in order to find a model over $\Q$. 
%\end{remark}

Our strategy will be to compute a basis for $S_{2}(\Gamma(N),\Q(\zeta_N))^G$ and derive equations for $X_G$ by finding equations between these cusp forms, following Galbraith \cite{galbraith}.
 \subsection{The conjugation trick}
 
 %By Proposition \ref{canonicalmodelprop}, we find that
%\[
%S_{2k}(\Gamma(N),\Q(\zeta_N))^{G,\mathcal{W}}\simeq H^0(X_G/\mathcal{W},\Omega^{\otimes k}).
%\]
%The objective is thus to find a (canonical) model for $X_G/\mathcal{W}$ by computing $S_2(\Gamma(N),\Q(\zeta_N))^{G,\mathcal{W}}$.
%
It will be useful to split the level $N$ into two parts $N=MK$, where $\mathrm{gcd}(M,K)=1$, such that the image of $G$ in $\mathrm{GL}_2(\Z/M\Z)$ is $B_0(M)$. Then
\begin{align}\label{anotheriso}
S_{2}(\Gamma(N),\Q(\zeta_N))^{G}=S_{2}(\Gamma_0(M)\cap\Gamma(K),\Q(\zeta_K))^{G},
\end{align}
by definition of $B_0(M)$. %Note that $S_{2k}(\Gamma_0(M)\cap \Gamma(K),\Q(\zeta_K))\simeq H^0(\widetilde{X}(\Gamma_0(M)\cap \Gamma(K)),\Omega^{\otimes k})$. 

From now on, we think of $\Gamma_G=\Gamma_0(M)\cap \Gamma_{G_K}$ as being a ``level $K$ congruence subgroup of $\Gamma_0(M)$'', rather than a level $N$ congruence subgroup of $\mathrm{SL}_2(\Z)$. When $M>1$, the benefit of this is twofold: we will be able to consider more automorphisms on $X_G$, and computations are faster. 

%Our strategy is to first compute $S_2(\Gamma_0(M)\cap \Gamma(K),\Q(\zeta_K))^{\Gamma_{G_K},\mathcal{W}}$ and then study the action of the matrices in $G$ with determinant unequal to 1. 
%
A problem with computing fixed spaces of modular forms as above, is that no algorithm for computing spaces of the form $S_{2}(\Gamma_0(M)\cap \Gamma(K),\Q(\zeta_K))$ is currently implemented in a computer algebra system. We fix this by conjugating with $\gamma_K$. The trick to studying modular forms on $\Gamma(K)$, as used  by Banwait and Cremona \cite{banwait} and later by Zywina \cite{zywina}, is to notice that
$
\gamma_K^{-1}\Gamma(K)\gamma_K = \Gamma_0(K^2)\cap \Gamma_1(K).
$
In fact, we already saw in Example \ref{isoex} that $\gamma_K$ induces an isomorphism
\begin{align}\label{gammaKiso}
\widetilde{X}(\Gamma_0(M)\cap \Gamma(K))\simeq \widetilde{X}(\Gamma_0(MK^2)\cap \Gamma_1(K)).
\end{align}
Efficient algorithms for computing spaces of cusp forms for $\Gamma_0(MK^2)\cap \Gamma_1(K)$ using modular symbols \emph{have} been implemented in \texttt{Magma} and \texttt{Sage} thanks to the work of Cremona \cite{cremona} and Stein \cite{stein}, amongst others. 
\subsection{Normalisers and statement of the main theorem}
\label{normalisersec}
In this section, we explain for which elements  $A\in \mathcal{N}_{\Gamma_G}$ we can determine its action on $X_G$ explicitly. We would like to be able to act with such $A$ on $S_2(\Gamma_0(MK^2)\cap \Gamma_1(K),\overline{\Q})$, as is the case for $G$. 

In their famous paper, Conway and Norton \cite[Section 3]{conway} mention the ``curious fact'' that the divisors $h$ of 24 are exactly those positive integers satisfying that $xy\equiv 1 \mod h$ implies $x\equiv y \mod h$. Equivalently, they are the integers $h$ such that
\[
\widetilde{T}_h:=\begin{pmatrix} 1 & 1/h \\ 0 & 1\end{pmatrix} 
\]
normalises $\Gamma_0(h^2)$.
\begin{lemma}
\label{normaliserlemma}
The normaliser of $\Gamma_0(M)$ in $\mathrm{PSL}_2(\R)$ is generated by $\Gamma_0(M)$ itself, the Atkin--Lehner matrices $W_m(x,y,z,w)$ for $m\mid M$ with $\mathrm{gcd}(M/m,m)=1$, and $\widetilde{T}_h$, where $h$ is the largest divisor of 24 such that $h^2\mid M$.
\end{lemma}
 This lemma is originally due to Atkin and Lehner \cite{atkinlehner}; see \cite[Section 3]{conway} for an elegant proof. While the Atkin--Lehner matrices act on $X_0(M)$ over $\Q$, the matrix $\widetilde{T}_h$ determines an automorphism on $X_0(M)$ defined over $\Q(\zeta_h)$. When $h$ divides $M$ and $H\subset \mathrm{GL}_2(\Z/M\Z)$ is a subgroup, define
 \[
 H^h:=\{M\in H \mid \mathrm{det}(M)\equiv 1 \mod h\}. 
 \]
 By Example \ref{tildeisoex}, we find that $X_0(M)_{\Q(\zeta_h)}=X_{B_0^h(M)}$. Similarly, we have $(X_G)_{\Q(\zeta_h)}=X_{G^h}$. On $X_{G^h}$, the elements in $\mathcal{N}_{G^h}$ act by automorphisms. We restrict to those automorphisms which also act on $X(\Gamma_0(M)\cap\Gamma(K))$. 
\noindent\begin{corollary}\label{normcor}
Let $h$ be the largest divisor of 24 such that $h^2\mid M$. The normaliser $\mathcal{N}_{\Gamma_0(M)}$ acts by automorphisms on $X(\Gamma_0(M)\cap \Gamma(K))_{\Q(\zeta_{Kh})}$.
\end{corollary}
\begin{proof}
In view of Example \ref{tildeisoex}, it suffices to show that $\mathcal{N}_{\Gamma_0(M)}$ acts by automorphisms on $X_H$, where $H\subset \mathrm{GL}_2(\Z/MK\Z)$ is the intersection of the inverse images of $\widetilde{G}(K)=\{I\}\subset \mathrm{GL}_2(\Z/K\Z)$ and $B_0(M)^h\subset \mathrm{GL}_2(\Z/M\Z)$. Note that $X_H$ is defined over $\Q(\zeta_{Kh})$. We can now verify explicitly that Atkin--Lehner matrices, $h\widetilde{T}_h$ and elements of $\Gamma_0(M)$ satisfy condition (\ref{morphismcondition}) to define a morphism $X_H\to X_H$. Here we note that any matrix in $\mathrm{GL}_2(\Z/K\Z)$ normalises $\{I\}=\widetilde{G}(K)$. To see that $h\widetilde{T}_h$ satisfies (\ref{morphismcondition}), we crucially use that $xy\equiv 1 \mod h$ implies $x\equiv y\mod h$. 
%We saw in Example \ref{atkinlehnerex} that Atkin--Lehner matrices act as automorphisms over $\Q$ on $X_0(M)$.  Moreover, each $\gamma\in \Gamma_0(M)$ normalises $B_0(M)$ and $\{I\}\subset \mathrm{GL}_2(\Z/K\Z)$, hence acts on $\widetilde{X}(\Gamma_0(M)\cap \Gamma(K))$ by Example \ref{ex27}. Now apply Example \ref{tildeisoex}. Finally, define $\widetilde{B}_0^0(h)\subset \mathrm{GL}_2(\Z/h\Z)$ to be the subgroup of elements $\begin{pmatrix} a & 0 \\ 0 & a^{-1}\end{pmatrix}$ for $a\in (\Z/h\Z)^{\times}$. We note that $\widetilde{T}_1$ normalises $B_0^0(h)$ and defines an automorphism on $X(\widetilde{B}_0^0(h))$ (over $\Q(\zeta_h)$) by Example \ref{ex27}. Now the matrix $\gamma_h$ determines, again by Proposition \ref{normalisermoduliprop}, an isomorphism $X(\widetilde{B}_0^0(h))\simeq X_0(h^2)_{\Q(\zeta_h)}$, and $\gamma_h^{-1}\widetilde{T}_1\gamma_h=\widetilde{T}_h$. We conclude that $\widetilde{T}_h$ determines a $\Q(\zeta_h)$-automorphism on $X(\Gamma_0(h^2))$. For the same reason, $\widetilde{T}_h$ acts on $X(\Gamma_0(MK^2)\cap \Gamma_1(K))_{\Q(\zeta_{Kh})}$.
\end{proof}
 Consider a subgroup $\mathcal{A}\subset (\mathcal{N}_{\Gamma_0(M)}\cap \mathcal{N}_{G^h})/\mathrm{P}\Gamma_G$. By definition, this acts by automorphisms on $X_{G^h}=(X_G)_{\Q(\zeta_h)}$. Moreover, by intersecting with $\mathcal{N}_{\Gamma_0(M)}$, we have ensured that $\mathcal{A}$ acts on $X(\Gamma_0(M)\cap \Gamma(K))$ over $\Q(\zeta_{Kh})$, just like $G$ does, and we can treat $\mathcal{A}$ and $G$ in a similar way. Taking $\mathcal{A}$-invariants in Proposition \ref{canonicalmodelprop} and applying (\ref{gammaKiso}) and (\ref{anotheriso}), we obtain
\begin{align}\label{1formsiso}
H^0((X_G/\mathcal{A})_{\Q(\zeta_h)},\Omega)=S_{2}(\Gamma_0(MK^2)\cap \Gamma_1(K),\Q(\zeta_{Kh}))^{G,\mathcal{A}},
\end{align}
%Suppose that the image of $\Gamma'$ in $\mathrm{SL}_2(\Z/K\Z)$ is generated by $\overline{\gamma}_1,\ldots,\overline{\gamma}_n$, and let $\gamma_1,\ldots,\gamma_n \in \Gamma'\cap \Gamma_0(M)=\Gamma$ be their respective lifts.  Note that we can choose them to lie in $\Gamma_0(M)$ because $\mathrm{gcd}(M,K)=1$. Because $\Gamma_0(M)\cap \Gamma(K)$ is a normal subgroup of $\Gamma$, each $\gamma_i$ for $i\in \{1,\ldots,n\}$ acts on $S_2(\Gamma_0(M)\cap \Gamma(K),\CC)$.  The Atkin--Lehner involutions $\gamma_{n+1},\ldots,\gamma_m$ normalise both $\Gamma_0(M)$ and $\Gamma(K)$, hence also their intersection. We thus find that 
%We conclude that
%\[
%H^0(X_G,\Q(\zeta_K))=\left(S_{2}(\Gamma_0(M\cdot K^2)\cap \Gamma_1(K),\Q(\zeta_K))^{G_K}\right)^{\mathcal{A}},
%\]
where $\al\in \mathcal{A}$ and $\al\in \mathrm{SL}_2(\Z/MK\Z)$ act on cusp forms $f$ by $f\mapsto f[\gamma_K^{-1}\al\gamma_K]$, and matrices $\gamma_a=\begin{pmatrix} a & 0 \\ 0 & 1\end{pmatrix}\in \mathrm{GL}_2(\Z/MK\Z)$  act on cusp forms by Galois conjugating Fourier coefficients by $\sigma_a$ (because $\gamma_a$ and $\gamma_K$ commute), c.f. Section \ref{sec31}.
%In general, the matrices in $\mathcal{A}$ act on the space fixed by $G$, but not on the full space of modular forms $S_2(\Gamma_0(MK^2)\cap \Gamma_1(K),\Q(\zeta_K))$. This is problematic from a computational perspective. We therefore restrict to groups $\mathcal{A}$ such that each $a\in \mathcal{A}$ also acts as a morphism on $X(\Gamma_0(M)\cap \Gamma(Kh))$. Denote by $\mathrm{GL}_2(\Z_{(K)})$ the matrices with determinant coprime to $K$. In practice, c.f. Lemma \ref{coprimelemma}, we thus consider groups $\mathcal{A}$ generated by elements of $\mathcal{N}_{B_0(M)}\cap \mathrm{GL}_2(\Z_{(K)})$. 
\begin{theorem}
 Suppose that $M$ and $K$ are coprime integers, $G\subset \mathrm{GL}_2(\Z/MK\Z)$ satisfies $-I\in G$, $\mathrm{det}(G)=(\Z/MK\Z)^{\times}$, $JGJ=G$, and $G_M=B_0(M)$. Let $h$ be the largest divisor of 24 such that $h^2\mid M$. Consider a finite subgroup $\mathcal{A}\subset (\mathcal{N}_{\Gamma_0(M)}\cap \mathcal{N}_{G^h})/\mathrm{P}\Gamma_G$, normalised by $J$ and determining a $\mathrm{Gal}(\overline{\Q}/\Q)$-invariant subgroup of $\mathrm{Aut}_{\overline{\Q}}(X_G)$. Then Algorithm \ref{algorithm2} is a terminating algorithm for computing a model for the curve $X_G/\mathcal{A}$ over $\Q$, provided it has genus $g\geq 2$. 
 \end{theorem}
 \begin{remark} \label{allowedsets}
 While we have restricted the allowed automorphisms by intersecting with $\mathcal{N}_{\Gamma_0(M)}$, we have not excluded the two natural subsets.  
 Firstly, recall that any $\al\in \mathcal{N}_{\Gamma_0(M)}$ determines a morphism on $X_0(M)_{\Q(\zeta_h)}$. Because $G_K$ has surjective determinant and $\al$ has determinant coprime to $K$, we can by Lemma \ref{coprimelemma}  extend this to a morphism on $X_G$ determined by a matrix in $\mathcal{N}_{\Gamma_0(M)}\cap \mathcal{N}_{G^h}$. The allowed automorphisms thus contain the Atkin--Lehner involutions and the morphisms determined by $\widetilde{T}_h$. We will see that $J$ indeed normalises Atkin--Lehner operators. Secondly, any $\al\in \mathrm{SL}_2(\Z/K\Z)$ normalising $G_K$ determines an automorphism on $X_{G_K}$, after lifting $\al$ to $\mathrm{SL}_2(\Z)$. By lifting $\al$ to $\Gamma_0(M)$, we obtain a morphism on $X_G$ determined by a matrix in $\mathcal{N}_{\Gamma_0(M)}\cap \mathcal{N}_{G^h}$. 
 \end{remark}
 We shall first find a way (Algorithm \ref{algorithm1}) to determine the fixed space \newline$S_2(\Gamma_0(MK^2)\cap \Gamma_1(K),\Q(\zeta_{Kh}))^{\langle \Gamma_G,\mathcal{A}\rangle}$, after which we consider the action of matrices in $G_K$ with determinant unequal to 1 in Section \ref{rationalstrsec}. 
\begin{remark}We have decided in  the remainder of this article to work only with weight 2 cusp forms, but all computational arguments generalise in a straightforward way to arbitrary even weight.
\end{remark}

\subsection{The homology group}
We work in the homology group, following Cremona's strategy \cite{cremona}. More specifically, in this section, we follow unpublished work of Cremona (used in  \cite{annotatedsagecode}) to describe a $\CC$-bilinear pairing between cusp forms and homology. 

We consider any group $G\subset \mathrm{GL}_2(\Z/N\Z)$ corresponding to a congruence subgroup $\Gamma:=\Gamma_G$ of level $N$. We assume again that $G$ is normalised by
 $
 J=\begin{pmatrix} -1 & 0 \\ 0 & 1\end{pmatrix}.
 $
For the geometric curve $(X_G)_{\CC}$, we then have the pairing between 1-forms and homology:
\[
S_2(\Gamma,\CC)\times H_1(X_G,\Z) \to \CC,\;\; \langle f,\gamma\rangle :=2\pi i \int_{\gamma} f(z)\mathrm{d} z.
\]
Extending this to $H_1(X_\Gamma,\R)$, the pairing becomes $\R$-bilinear. As in Section \ref{sec31}, the action of $J$ on $\mathcal{H}\subset \CC$ composed with complex conjugation determines an involution of $\R$-algebras $J^*:\;\CC(X_G)\to \CC(X_G)$, or equivalently an involution on the Weil restriction $\mathrm{Res}_{\CC/\R}(X_G)_{\CC}$.  Here $z\in \mathcal{H}$ is mapped to $z^*:=-\overline{z}=\overline{J(z)}$, and $f\in S_2(\Gamma,\CC)$ to $J^*(f)=\overline{f(z^*)}$.  We note that $f\mapsto J^*(f)$ acts as complex conjugation on the Fourier coefficients of $f$, c.f. Section \ref{sec31}. We similarly obtain an involution on $H_1(X_{\Gamma},\R)$. Denote by $H_1(X_{\Gamma},\R)^+$ its $+1$-eigenspace. Then our pairing restricts to an exact duality of real vector spaces
\[
S_2(\Gamma,\R)\times H_1(X_\Gamma,\R)^+\to \R.
\]
Define 
\[
H_{\CC}(\Gamma):=H_1(X_\Gamma,\R)^+\otimes \CC.
\]
Finally, we extend the pairing $\CC$-linearly \emph{on both sides} to obtain an exact $\CC$-bilinear pairing 
\begin{align}
\label{pairing}
S_2(\Gamma,\CC)\times H_{\CC}(\Gamma)\to \CC,\; (f,\mu)\mapsto \langle f,\mu\rangle. 
\end{align}
The pairing (\ref{pairing}) identifies $H_{\CC}(\Gamma)$ with the dual of $S_2(\Gamma,\CC)$. Subsequently, the Petersson inner product identifies $S_2(\Gamma,\CC)$ with its own dual. We thus obtain an isomorphism 
\begin{align}\label{pairingiso}S_2(\Gamma,\CC)\to H_{\CC}(\Gamma),\; f\mapsto \gamma_f
\end{align}
(due to Cremona) defined by
\[
\langle g,\gamma_f\rangle = \langle g,f\rangle \text { for all }g\in S_2(\Gamma,\CC),
\]
where the latter pair of brackets denotes the Petersson inner product. Our strategy for studying $S_2(\Gamma,\CC)$ is to study $H_{\CC}(\Gamma)$ instead, and transform the results under this isomorphism. We note that the Petersson inner product is sesqui-linear in its second argument, whereas (\ref{pairing}) is $\CC$-linear. This means that 
\[
\gamma_{\al f}=\overline{\al}\gamma_f \text{  for each }\al\in \CC. 
\]

 The great benefit of studying the homology group is that the action of  $\mathcal{N}_{\Gamma}$ on $H_1(X(\Gamma),\Z)$ can be computed explicitly when $\Gamma=\Gamma_0(K^2M)\cap \Gamma_1(K)$, i.e. for each $\gamma\in \mathcal{N}_{\Gamma}$ one can find a matrix for this linear action on $H_1(X_{\Gamma},\CC)$ in terms of an explicit basis. This uses a description of the homology group as \emph{modular symbols}. For a detailed overview of these algorithms, we refer the reader to \cite{cremona} and \cite{stein}. 

 Thanks to these algorithms, when $\Gamma\supset \Gamma_0(MK^2)\cap \Gamma_1(K)$, one can compute $H_1(X_{\Gamma},\Q)$ as the subspace of $H_1(X_{\Gamma_0(MK^2)\cap \Gamma_1(K)},\Q)$ fixed by $\Gamma$ with ease. The challenge is to relate this subspace $H_1(X_{\Gamma},\CC)$ to $q$-expansions of modular forms under the isomorphism (\ref{pairingiso}), as explained in the next Section \ref{qexps}.

In the remainder of this subsection, we consider any congruence subgroup $\Gamma$ normalised by $J$, but one may think of this as being $\Gamma_0(MK^2)\cap \Gamma_1(K)$. 

Recall that every $U=\begin{pmatrix} a & b \\ c &d\end{pmatrix}\in \mathrm{GL}_2^+(\Q)$ acts on the complex upper half-plane by fractional linear transformations, giving rise to an action on meromorphic functions on $\mathcal{H}$
\[
[U]: \;f(\tau)\mapsto (f|U)(\tau):=(c\tau+d)^{-2}f\circ U(\tau)
\]
and on paths $\gamma:[0,1]\to \mathcal{H}$:
\begin{align}\label{paths}
[U]:\;  \gamma \mapsto U(\gamma):=U\circ \gamma.
\end{align}

\begin{definition}
We define the \emph{Hecke algebra} $\mathbb{T}_{\Gamma}$ to be the $\CC$-algebra of $\CC$-valued functions on $\mathrm{P}\Gamma\backslash \mathrm{P}\mathrm{GL}_2^+(\Q)/\mathrm{P}\Gamma$, where the $\mathrm{P}$ denotes the image in $\mathrm{PGL}_2(\Q)$.
\end{definition}

Each $T\in \mathbb{T}_{\Gamma}$ can be represented as a $\CC$-linear combination of double cosets $\mathrm{P}\Gamma \al \mathrm{P}\Gamma$. Such a coset has a left-$\mathrm{P}\Gamma$ action, and splits as a finite disjoint union $\mathrm{P}\Gamma \al \mathrm{P}\Gamma =\sqcup \mathrm{P}\Gamma \al_i$. This way, we obtain a \emph{Hecke operator} $T_{\al}=\sum_i [\al_i]$, acting on $(X_{\Gamma})_{\CC}$ as a correspondence, and consequently  on $S_2(\Gamma,\CC)$ and $H_1(X_{\Gamma},\CC)$.

The following lemma, due to Cremona and used in  \cite{annotatedsagecode}, describes how these two actions interact with the pairing (\ref{pairing}).
\begin{lemma}
\label{transformationlemma}
Consider $T\in \T_{\Gamma}$, and $\gamma\in H_{\CC}(\Gamma)$ such that $JT(\gamma)=T(\gamma)$ (i.e. $T(\gamma)$ remains in the +1-eigenspace). This is always the case when $TJ=JT$. Then for all $f\in S_2(\Gamma,\CC)$ we have
\begin{itemize}
    \item[(a)] $\langle f|T,\gamma \rangle = \langle f,T(\gamma)\rangle$ 
    \item[(b)] $\gamma_{f|T^*}=T(\gamma_f)$,
\end{itemize}
where $T^*$ is the adjoint of $T$ with respect to the Petersson inner product. 
\end{lemma}
\begin{proof}
\label{pairingtransformationlemma}
By definition of the action, $(f|U)(\tau)\dd \tau = f(U(\tau))\dd U(\tau)$ for each matrix $U$. Integrating this relation for each $U$ in $T$ gives us part (a), as long as $T(\gamma)\in H_{\CC}(\Gamma)$. Part (b) follows by definition of the Petersson inner product.
\end{proof}
We recall (see e.g. \cite[Proposition 5.5.2]{diamondshurman}) that for $U\in \mathrm{GL}_2^+(\Q)$, we have
$
[U]^*=\mathrm{det}(U)[U^{-1}].
$
\subsection{$q$-expansions}
\label{qexps}
A downside of using the homology group is that the pairing (\ref{pairing}) is not defined explicitly in terms of $q$-expansions. It is therefore a priori not obvious what the $q$-expansion is of the cusp form mapped to a given element of $H_{\CC}(\Gamma)$ under (\ref{pairingiso}). 

When $\Gamma=\Gamma_0(MK^2)\cap \Gamma_1(K)$, the solution is to use Hecke operators and their common eigenvectors. Let $f_1,\ldots,f_n\in S_2(\Gamma_0(MK^2)\cap \Gamma_1(K),\CC)$ be the Hecke eigenforms.  Their $q$-expansions (up to scaling) are determined by Hecke eigenvalues, and these Hecke eigenvalues can be computed for the corresponding Hecke eigenvectors $\mu_1,\ldots,\mu_n \in  H_{\CC}(\Gamma_0(MK^2)\cap \Gamma_1(K))$ instead, by Lemma \ref{transformationlemma}. This yields the $q$-expansions for a basis of $S_2(\Gamma_0(MK^2)\cap \Gamma_1(K),\CC)$.

This approach does not work for $S_2(\Gamma,\CC)$ when $\Gamma\supset \Gamma_0(MK^2\cap \Gamma_1(K)$ and the subspace $S_2(\Gamma,\CC)\subset S_2(\Gamma_0(MK^2)\cap \Gamma_1(K),\CC)$ is not a direct sum of Hecke eigenspaces, the problem being that each $\mu_i$ only corresponds to $f_i$ under (\ref{pairingiso}) \emph{up to scaling}. The scaling issue makes it hard to translate linear combinations of eigenforms under (\ref{pairingiso}). 
% 
 %Now consider again $\Gamma\supset \Gamma_0(MK^2)\cap \Gamma_1(K))$ and take $\Gamma$-invariants. In general, $S_2(\Gamma,\CC)\subset S_2(\Gamma_0(MK^2)\cap \Gamma_1(K),\CC)$ is not a direct sum of eigenspaces for the Hecke operators $T_p$. For example, consider Hecke eigenvectors $\gamma_1,\gamma_2\in H_{\CC}(X_{\Gamma_0(MK^2)\cap \Gamma_1(K)})$. We find Hecke eigenforms $f_1,f_2\in S_2(\Gamma_0(MK^2)\cap \Gamma_1(K),\CC)$ such that there exist $\al_1,\al_2\in \CC^{\times}$ with $\gamma_1=\al_1\gamma_{f_1}$ and $\gamma_2=\al_2\gamma_{f_2}$. Suppose that $\gamma:=a_1 \gamma_1+a_2\gamma_2$ is fixed by $\Gamma$, i.e. $\gamma\in H_{\CC}(\Gamma)$, for some $a_1,a_2\in \CC$. The corresponding modular form $f$ such that $\gamma=\gamma_f$ must be a linear combination $b_1f_1+b_2f_2$ of $f_1$ and $f_2$, but we do not know $b_1$ or $b_2$ because we do not know $\al_1$ and $\al_2$
 
 The aim for the remainder of this article is to present a solution to this problem. The crucial idea -- due to Cremona (see \cite{annotatedsagecode}) -- is that we \emph{can} find a $q$-expansion for the cusp form corresponding to a linear combination $\al_1\mu_1+\al_2\mu_2$ when $\mu_2$ is a twist of $\mu_1$. While two Hecke eigenvectors need not always be twists, we define in Section \ref{twistorbitsection} the \emph{twist orbit space} of $\mu_1$, and show that the action of $\Gamma$ preserves this space.

\subsection{Operators on modular forms and modular symbols}In this subsection, we study the congruence subgroup
$
\Gamma_0(MK^2)\cap \Gamma_1(K),
$
where again $N,M\in \Z_{\geq 1}$ (not necessarily coprime).  We define $N:=K^2M$. Let $\mathcal{D}_K$ be the group of Dirichlet characters on $(\Z/K\Z)^{\times}$. Then
\[
S_2(\Gamma_0(MK^2)\cap \Gamma_1(K),\CC)=\bigoplus_{\ep\in \mathcal{D}_K}S_2(N,\ep_N),
\]
where $\ep_N$ denotes the composition of $(\Z/N\Z)^{\times}\to (\Z/K\Z)^{\times}$ and $\ep$, and $S_2(N,\ep)$ is the common eigenspace in $S_2(\Gamma_1(N),\CC)$ for the diamond operators $\langle d \rangle$ for $\mathrm{gcd}(d,N)=1$ with eigenvalues $\ep(d)$ respectively. The diamond operators also act on modular symbols, and Stein \cite{stein} defines their eigenspaces $H_1(N,\ep)\subset H_1(X_{\Gamma_1(N)},\CC)$. We obtain a similar decomposition 
\[
H_1(X_{\Gamma_0(MK^2)\cap \Gamma_1(K)},\CC)=\bigoplus_{\ep\in\mathcal{D}_K}H_1(N,\ep).
\]
The diamond operators commute with $J$, and $H_1(N,\ep)^+$ is identified with $S_2(N,\overline{\ep})$ under the isomorphism (\ref{pairingiso}). Note here the complex conjugation of $\ep$ which occurs due to Lemma \ref{transformationlemma} (b) since $\langle d \rangle^*=\langle d \rangle ^{-1}$ for $d$ coprime to $N$. Similarly, for primes $p\nmid N$, we have (see e.g. \cite{diamondshurman})
\[
T_p^*=\langle p\rangle ^{-1}T_p. 
\]
\begin{definition}
Let $V$ be a representation of the Hecke algebra. A \emph{Hecke eigenspace} of $V$ is a simultaneous eigenspace for the Hecke operators $T_p$ and $\langle p \rangle$, for all but finitely many primes $p$. 
\end{definition}
Under the isomorphism (\ref{pairingiso}), we conclude by Lemma \ref{transformationlemma} that Hecke eigenspaces in $S_2(\Gamma_0(MK^2)\cap \Gamma_1(K),\CC)$ are identified with Hecke eigenspaces in $H_{\CC}(\Gamma_0(MK^2)\cap \Gamma_1(K))$, but the eigenvalues get complex conjugated.

%When $V=S_2(\Gamma_0(MK^2)\cap \Gamma_1(K))$, each newform generates a 1-dimensional Hecke eigenspace. The Hecke eigenspaces containing old forms have higher dimension due to the multiple embeddings of the old subspace into $V$.
%
%Denote by $T_p$ the Hecke operator at the prime $p$, and by $w_N$ the Atkin--Lehner operator at level $N$. Since $T_p^*=w_NT_pw_N^{-1}$, we find that the simultaneous $T_p$-eigenvectors in $H_{\CC}(\Gamma)$ correspond to the eigenforms. 
\begin{definition}
Let $L$ be a positive integer, and define, as before,
$
\widetilde{T}_L:=\begin{pmatrix} 1 & 1/L \\ 0 & 1\end{pmatrix}.
$
 For $\chi\in \mathcal{D}_L$ we define the \emph{twist operator}
\[
R_{\chi}(L):=\sum_{u\in \Z/L\Z} \overline{\chi}(u)[\widetilde{T}_L^u].
\]
When $L=\mathrm{cond}(\chi)$, we write $R_{\chi}(L)=R_{\chi}$. 
\end{definition}
It is important to note that, even if $f$ is of level $N$ and $L\mid N$, then $\chi$ is considered as a mod $L$ character in the definition of $f|R_{\chi}(L)$; not as a mod $N$ character. 

We saw in Corollary \ref{normcor} that $\widetilde{T}_1$ acts on $\widetilde{X}(\Gamma_0(M)\cap \Gamma(K))$ (and hence normalises $\Gamma_0(M)\cap \Gamma(K))$. By (\ref{gammaKiso}), conjugation by $\gamma_K$ shows that $\widetilde{T}_K=\gamma_K^{-1}\widetilde{T}_1\gamma_K$ normalises $\Gamma_0(MK^2)\cap \Gamma_1(K)$ and acts on its spaces of cusp forms (with coefficients in $\Q(\zeta_K)$) and homology; hence  so does $R_{\chi}(K)$ for $\chi\in \mathcal{D}_K$. 
\begin{definition}
Consider $L\mid N$, $d\mid N/L$ and $\ep\in \mathcal{D}_N$ of conductor dividing $L$. Let $\ep': (\Z/L\Z)^{\times}\to \CC^{\times}$ be the corresponding mod $L$ character. Let $D_d:=\begin{pmatrix} 1&0\\ 0&d\end{pmatrix}$ and suppose that  $\eta_1,\ldots,\eta_s$ are the defined so that $D_d\eta_1,\ldots,D_d\eta_s$ are a full set of coset representatives for the left-action of $\Gamma_0(N)$ on $\begin{pmatrix}1&0 \\ 0 & d\end{pmatrix}\Gamma_0(L)$. Then we have maps
\[
B_d^{N/L}=B_d: S_2(L,\ep')\to S_2(N,\ep), \quad f\mapsto f|\begin{pmatrix} d& 0 \\ 0&1\end{pmatrix}
\]
and 
\[
\mathrm{Tr}_d^{N/L}: S_2(N,\ep)\to S_2(L,\ep'),\quad f \mapsto \sum_{i=1}^s\ep'(\gamma_i)^{-1}f|D_d\gamma_i
\]
satisfying that $\mathrm{Tr}^{N/L}_d\circ B_d^{N/L}$ is multiplication by $[\Gamma_0(N):\Gamma_0(L)]$. These operators similarly act on homology as in (\ref{paths}).
\end{definition}
% When the spaces between which $B_d^{N/L}$ acts are assumed to be known to the reader, we will simply write $B_d$ instead. 
\begin{lemma}
\label{operatortranslation}
We have $JR_{\chi}(L)J=\chi(-1)R_{\chi}(L)$, $JB_dJ=B_d$ and $J\mathrm{Tr}_d^{N/L}J=\mathrm{Tr}_d^{N/L}$ for all $\chi\in \mathcal{D}_L$ and all positive integers $d\mid N/L$. 
\end{lemma}
\begin{proof}
For $B_d$ one simply multiplies matrices. For $\mathrm{Tr}_d^{N/L}$, it suffices to note that $\mathrm{Tr}_d^{N/L}$ is independent of the choice of coset representatives, and that $J$ commutes with $D_d$ and normalises $\Gamma_0(N)$ and $\Gamma_0(L)$. For $R_{\chi}(L)$, we use its definition and note that $J\widetilde{T}_K^aJ=\widetilde{T}_K^{-a}$. 
\end{proof}

%\begin{lemma}
%\label{Rchitranslationlemma}
%Suppose that $\chi(-1)=1$ and that $\widetilde{T}_K$ normalises $\Gamma$. Then $R_{\chi}$ acts on $H_{\CC}(\Gamma)$ and  $R_{\chi}(\gamma_f)=g(\overline{\chi})\gamma_{f\otimes \overline{\chi}}$, where $g(\overline{\chi})$ is the Gauss sum of $\overline{\chi}$. 
%\end{lemma} 
%\begin{proof}
%We first note that $\widetilde{T}_K^*=\widetilde{T}_K^{-1}$. Hence we find that $\langle g,R_{\chi} (\gamma_f)\rangle = \langle g|R_{\chi},\gamma_f\rangle =\langle g|R_{\chi},f\rangle = \langle g,f|R_{\overline{\chi}}\rangle$, as desired.
%\end{proof}

\begin{lemma}
\label{Bdtranslationlemma}
Consider $f\in S_2(N,\ep)$ and $g\in S_2(L,\ep')$, where $L\mid N$. Let $d\mid N/L$, and let $\chi\in \mathcal{D}_L$ with $\chi(-1)=1$. Then 
\begin{align*}
\gamma_{f|R_{\chi}}&=R_{\overline{\chi}}(\gamma_f),\\
[\Gamma_0(N):\Gamma_0(L)]\cdot \gamma_{g|B_d^{N/L}}&=\mathrm{Tr}_d^{N/L}(\gamma_g) \text{ and}\\
\gamma_{f|\mathrm{Tr}_d^{N/L}}&=[\Gamma_0(N):\Gamma_0(L)]\cdot \mathrm{B}_d^{N/L}(\gamma_f).
\end{align*}
\end{lemma}
\begin{proof}
For the first equality, we use that $\widetilde{T}_L^*=\widetilde{T}_L^{-1}$ and recall that the isomorphism (\ref{pairingiso}) is $\CC$-linear in both arguments, to find that
\[
\langle h,\gamma_{f|R_{\chi}}\rangle = \langle h,f|R_{\chi}\rangle = \langle h|R_{\overline{\chi}},f\rangle =\langle h|R_{\overline{\chi}},\gamma_f\rangle = \langle h,R_{\overline{\chi}}(\gamma_f)\rangle,
\]
where we used Lemma \ref{transformationlemma} (a) in the last equality. 

Next, recall that
\[
\langle f|\mathrm{Tr}_d^{N/L},g\rangle = [\Gamma_0(N):\Gamma_0(L)]\langle f,g|B_d\rangle .
\]
Define $I:=[\Gamma_0(N):\Gamma_0(L)]$. We thus find that 
\[
I\cdot \langle g,\gamma_{f|B_d}\rangle = I\cdot \langle g,f|B_d\rangle = \langle g|\mathrm{Tr}_d^{N/L},f\rangle = \langle g|\mathrm{Tr}_d^{N/L},\gamma_f\rangle= \langle g,\mathrm{Tr}_d^{N/L}(\gamma_f)\rangle,
\]
as desired. Now repeat this with $I\cdot B_d$ and $\mathrm{Tr}_d^{N/L}$ reversed.
\end{proof}
We define one final operator.
\begin{definition}
Suppose that $L\mid N$ and $\ep\in \mathcal{D}_N$ of conductor dividing $L$. Then we denote by
\[
\mathrm{pr}_{N/L}: \; S_2(N,\ep)\to B_1^{N/L}(S_2(L,\ep'))
\]
the orthogonal projection map onto the subspace, where ``orthogonal'' refers to the Petersson inner product, and $\ep': (\Z/L\Z)^{\times}\to \CC^{\times}$ is the character $\ep$ factors through.

Similarly, we denote by
\[
\pi_{N/L}:\; H_1(N,\ep)\to \mathrm{Tr}_1^{N/L}(H_1(L,\ep'))
\]
the orthogonal projection onto the image of $\mathrm{Tr}_1^{N/L}$.
\end{definition}
By Lemma \ref{Bdtranslationlemma}, we find that these operators are indeed dual to each other, i.e.
\begin{align}
\label{prequivalence}
\pi_{N/L}(\gamma_f) = \gamma_{\mathrm{pr}_{N/L}(f)} \text{ for all } f\in S_2(N,\ep).
\end{align}

Next, we study the $R_{\chi}(L)$-operators in more detail. We define
$
U_q:=\sum_{u=0}^{q-1}\left[\begin{pmatrix} 1 & u \\ 0 & q\end{pmatrix}\right]
$
and recall a list of properties of the $R_{\chi}(L)$ operators. 
\begin{lemma}
\label{longlemma2}
Consider a character $\chi$ of conductor dividing $L$, where $L\mid N$.
\begin{itemize}
\item[(a)] Suppose $f\in S_2(N,\ep)$ and $\mu \in H_1(N,\ep)$, where $\ep$ has conductor $N'$, and set $\widetilde{N}:=\mathrm{lcm}(N,N'L,L^2)$. Then $f|R_{\chi}(L)\in S_2(\widetilde{N},\ep\chi^2)$ and $R_{\chi}(L)(\mu)\in H_1(\widetilde{N},\ep\overline{\chi}^2)$.
\item[(b)] If $\mathrm{gcd}(L_1,L_2)=1$, $\chi\in \mathcal{D}_{L_1}$ and $\psi\in \mathcal{D}_{L_2}$, then $R_{\chi}(L_1)R_{\psi}(L_2)=\chi(L_2)\psi(L_1)R_{\chi\psi}(L_1L_2)$.
    \item[(c)] We have $f|R_{\chi}=g(\overline{\chi})f\otimes \chi$, where $g(\overline{\chi})$ is the Gauss sum of $\overline{\chi}$.
    \item[(d)] When $p\nmid LN$, we have $R_{\chi}(L)T_p=\chi(p)T_pR_{\chi}(L)$.
    %\item[(d)] When $b\geq a\geq 1$, $p$ is a prime number and $\mathrm{cond}(\chi)=p^a$, we have $R_{\chi}(p^b)=p^{b-a}U_{p^{b-a}}R_{\chi}B_{p^{b-a}}$,
    %\item[(e)] Let $\chi_0$ be the trivial character and $b\geq 1$. Then $R_{\chi_0}(p^b)=p^bU_{q^b}B_{q^b}-q^{b-1}U_{q^{b-1}}B_{q^{b-1}}$
    \item[(e)] For $a\in \Z$ with $\mathrm{gcd}(a,L)=1$, we have $\phi(L)\widetilde{T}_L^a=\sum_{\ep\in\mathcal{D}_L}\ep(a)R_{\ep}(L)$.
    %\item[(g)] If $q^2\mid N$, $\ep$ is a character mod $N/q$ and $f\in S_2(N,\ep)$, then $f|U_q\in S_2(N/q,\ep)$. 
    %\item[(h)] If $p\nmid qN$ then $T_pU_q=U_qT_p$.
    %\item[(i)] Let $Q$ be a power of the prime $q$,  and $f\in S_2(N,\ep_Q\ep_{N/Q})$ be a newform. Then $f|w_{Q} =\lambda f'$ for some $\lambda\in \overline{\Q}$ and a newform $f'\in S_2(N,\overline{\ep_Q}\ep_{N/Q})$. Moreover, $f|R_{\overline{\ep_Q}}=f'-f'|U_q|B_q$. 
    %\item[(j)] When $d^2\mid N$, the operator $U_d$ maps $S_2(N,\ep)$ into $S_2(N/d,\ep)$. In particular, $f|U_d=0$ if $f$ is a newform (by (g)).
    \item[(f)] Suppose that $q$ is prime and $N=q^sL$ for some $s>0$. Suppose further that $f\in S_2(N,\ep)$ is a newform with $a_1(f)=1$ and $\psi$ is a character whose conductor is a power of $q$. Then $f\otimes \psi=g-g|U_Q|B_Q$, where $Q$ is a power of $q$ and $g\in S_2(QL,\ep\psi^2)$ is the newform with $a_1(g)=1$.
\end{itemize}
\end{lemma}
\begin{proof}
These are standard properties first proved by Atkin and Li \cite{atkinli}. The proof of (a) for modular forms is a matrix computation (see \cite[Proposition 3.1]{atkinli}) and thus similarly true for modular symbols, with a complex conjugation bar to account for the change from left-action to right-action. 
\end{proof}
Parts (a) and (d) tell us exactly how $R_{\chi}(L)$ acts on the set of Hecke eigenspaces. We extend part (f) to non-prime power conductors.

%\begin{lemma}
%Suppose that $\chi$ is a character of conductor dividing $K$, and $f$ is a newform in an old subspace of $S_2(\Gamma_0(MK^2)\cap \Gamma_1(K))$ such that $f|B_q\in S_2(\Gamma_0(MK^2)\cap \Gamma_1(K))$. Then
%\[
%f|B_q|R_{\chi}\in \mathrm{Im}B_q.
%\]
%\end{lemma}
%\begin{proof}
%We note that
%\[
%B_qR_{\chi}=\left(\sum_{u\in \Z/\mathrm{cond}(\chi)\Z}\chi(u)\widetilde{T}_K^{qu}\right)B_q.
%\]
%\end{proof}
\begin{corollary}
\label{projcor}
Suppose that $f\in S_2(\Gamma_0(MK^2)\cap \Gamma_1(K),\CC)$ is a newform at a level dividing $N$ with $a_1(f)=1$, and $\chi\in \mathcal{D}_K$. Then $f\otimes \chi=k-k'$, where $k$ is the newform at a level $L$ dividing $N$ with $a_1(k)=1$, and 
\[
k'\in \bigoplus_{\substack{q\mid K \\ q>1 \text{ prime}}}\mathrm{Im}B_q^{N/(N/q)}.
\]
In particular, $\mathrm{pr}_{N/L}(f|R_{\chi})=g(\overline{\chi})k$.
\end{corollary}
\begin{proof}
First, we note that 
\[
B_qR_{\chi}=\left(\sum_{u\in \Z/\mathrm{cond}(\ep)\Z}\chi(u)\widetilde{T}_{\mathrm{cond}(\chi)}^{qu}\right)B_q.
\]
We now write $\chi=\prod_{p\mid K}\chi_p$ as a product of characters of prime-power conductor. Then $R_{\chi}=\lambda \prod_{p\mid K}R_{\chi_p}$ for some $\lambda\in \overline{\Q}$ by Lemma \ref{longlemma2} (b).  We now repeatedly apply Lemma \ref{longlemma2} (f) to each $R_{\chi_p}$, keeping in mind the displayed equation above.
\end{proof}

Given $f$ and $\chi$, we note that a $q$-expansion for $k$ can be computed explicitly from its Hecke eigenvalues using Lemma \ref{longlemma2} (d). For $R_{\chi}$ and $B_d$ it is well-known how they act on $q$-expansions: $f|B_d(q)=f(q^d)$ and $f|R_{\chi}=g(\overline{\chi})f\otimes \chi$, where $g(\overline{\chi})$ is the Gauss sum of $\overline{\chi}$. The above corollary allows us to determine how $\mathrm{pr}_{N/L}$ acts on the $q$-expansions of twists of newforms.

For $Q\mid N$ with $\mathrm{gcd}(Q,N/Q)=1$, we defined the Atkin--Lehner matrices $W_Q(x,y,z,w)$ in Example \ref{atkinlehnerex}. Define the corresponding operators on cusp forms and homology by $w_Q(x,y,z,w)$. For simplicity of exposition, denote by $w_Q$ the Atkin--Lehner operator $w_Q(x,y,z,w)$, where $x\equiv 1 \mod N/Q$ and $y\equiv 1 \mod Q$.  We study how these interact with the Hecke operators, using
classical properties due to Atkin and Li \cite{atkinli}. 
\begin{lemma}
\label{longlemma1}
Consider an integer $N$, a divisor $Q\mid N$ such that $\mathrm{gcd}(Q,N/Q)=1$, and $\ep\in \mathcal{D}_N$. Write $\ep=\ep_Q\ep_{N/Q}$, where $\ep_Q\in \mathcal{D}_Q$ and $\ep_{N/Q}\in \mathcal{D}_{N/Q}$. We consider $\mu\in H_1(N,\ep)$ and $f\in S_2(N,\ep)$. Then
\begin{itemize}
    \item[(a)] Consider $Q'$ such that $Q'\mid N$ and $\mathrm{gcd}(Q',NQ)=1$, and write $\ep_{N/Q}=\ep_{Q'}\ep_{N/(QQ')}$. Then $w_Qw_{Q'}(\mu)=\ep_{Q'}(Q)w_{QQ'}(\mu)$.
    \item[(b)] $w_Q(x,y,z,w)(\mu)=\ep_Q(y)\ep_{N/Q}(x)w_Q(\mu)$,
    %\item[(c)] $w_Q(w_Q(\mu)) = \ep_Q(-1)\overline{\ep}_{N/Q}(Q)\mu$,
    \item[(c)]  $w_Q(x,y,z,w)$ maps $H_1(N,\ep_Q\ep_{N/Q})$ to $H_1(N,\overline{\ep_Q}\ep_{N/Q})$,
    \item[(d)] $T_pw_Q(\mu)=\overline{\ep}_Q(p)w_QT_p(\mu)$ when $\mathrm{gcd}(p,Q)=1$,
    
    %\item[(f)] If $m\mid N$ and $g\in \widetilde{T}_K(m,\ep_{\mathrm{gcd}(Q,m)}\ep_{m/\mathrm{gcd}(Q,M)})$,  define $e=\frac{d}{\mathrm{gcd}(d,Q)}\frac{Q}{\mathrm{gcd}(Q,m)\mathrm{gcd}(d,Q)}$. Then $g|B_d|w_Q= g|w'_{\mathrm{gcd}(Q,m)}|B_e$ for some Atkin--Lehner involution $w'_{\mathrm{gcd}(Q,m)}$ on level $m$. 
     %\item[(i)] Let $Q$ be a power of the prime $q$,  and $f\in S_2(N,\ep_Q\ep_{N/Q})$ be a newform. Then $f|w_{Q} =\lambda f'$ for some $\lambda\in \overline{\Q}$ and a newform $f'\in S_2(N,\overline{\ep_Q}\ep_{N/Q})$. Moreover, $f|R_{\overline{\ep_Q}}=f'-f'|U_q|B_q$. 
\end{itemize}
\end{lemma}
\begin{proof}
These are analogues of standard results of Atkin and Li \cite{atkinli} for the corresponding action on modular forms. These are all based on matrix identities, and thus hold true on modular symbols for the same reason (and with appropriate complex conjugation bars to account for the change from left-action to right-action). 
\end{proof}
Parts (b), (c) and (d) determine how the Atkin--Lehner operators act on the set of Hecke eigenspaces of $H_1(X_{\Gamma_0(MK^2)\cap \Gamma_1(K)},\CC)$.

\section{Algorithms}
\subsection{Twist orbit spaces} \label{twistorbitsection}
We consider again $N=KM$, $h\mid 24$ such that $h^2\mid M$, and the congruence subgroup $\Gamma=\Gamma_G$, where $G$ has level $N$ and $G_M=B_0(M)$.% Note that $J$ commutes with $\gamma_K$ and normalises $\Gamma_0(M)$. Also note that $J$ normalises $\Gamma_0(MK^2)\cap \Gamma_1(K)$.

A crucial step in determining $q$-expansions for a basis of modular forms for $\Gamma_0(N)$ is to determine first the Hecke eigenforms. When $N$ is prime, the Hecke algebra $\mathbb{T}_{\Gamma_0(N)}$ is generated by the Atkin-Lehner operator $w_N$ and the Hecke operators $T_p$ for primes $p$. The Hecke eigenforms thus generate 1-dimensional modules of the full Hecke algebra. 

For $\Gamma_0(MK^2)\cap \Gamma_1(K)$, however, this is not the case. Recall that $\widetilde{T}_h$ normalises $\Gamma_0(M)\cap \Gamma(K)$. Hence $\widetilde{T}_{Kh}=\gamma_K^{-1}\widetilde{T}_h\gamma_K$ normalises $\Gamma_0(MK^2)\cap \Gamma_1(K)$, making the twist operators $R_{\chi}(Kh)$, for $\chi\in \mathcal{D}_{Kh}$, are all elements of the Hecke algebra. These twist operators do not act on the 1-dimensional spaces generated by the Hecke eigenforms, however. In this section, we define the notion of \emph{twist orbit spaces}, and we show that these are modules for the following subalgebra.
\begin{definition}
We define the \emph{explicit Hecke algebra} $\T_{M,K}$ to be the sub-algebra of $\T_{\Gamma_0(MK^2)\cap \Gamma_1(K)}$ generated by the $T_p$-operators for $p\nmid MK$ prime, the diamond operators $\langle d\rangle$ for $d\in (\Z/K\Z)^{\times}$, the Atkin--Lehner operators $w_m(x,y,z,w)$ for $m\mid MK^2$ such that $\mathrm{gcd}(MK^2/m,m)=1$, and the twist operators $R_{\chi}(Kh)$ for $\chi\in \mathcal{D}_{Kh}$. 
\end{definition}

\begin{remark}
It is not unreasonable to expect that $\T_{M,K}$ is the full Hecke algebra, as we have added to the Hecke operators all matrices that visibly normalise $\Gamma_0(MK^2)\cap \Gamma_1(K)$, c.f. Lemma \ref{normaliserlemma} and Corollary \ref{gencor}.
\end{remark}

In the situation of our interest, some of the Atkin--Lehner operators coincide.
\begin{lemma}
\label{atkinlehnercommutes}
If $Q=mK^2$ for some $m\mid M$ with $\mathrm{gcd}(m,M/m)=1$, then 
\[
-W_Q(x,-y,-z,w)\cdot W_Q(x,y,z,w)^{-1}\in \Gamma_0(MK^2)\cap \Gamma_1(K). 
\]
In other words,  $w_Q(x,y,z,w)\in \T_{M,K}$ satisfies $Jw_Q(x,y,z,w)J=w_Q(x,y,z,w)$. %$S_2(\Gamma_0(MK^2)\cap \Gamma_1(K),\CC)$ and $H_1(X_{\Gamma_0(MK^2)\cap \Gamma_1(K)},\CC)$. 
\end{lemma}
\begin{proof} Simply multiply the matrices.\end{proof}

Define $T:=\begin{pmatrix} 1 & 1 \\ 0 & 1\end{pmatrix}$ and $S:=\begin{pmatrix} 0 & 1 \\ -1& 0\end{pmatrix}$. In their studies of modular forms, Banwait and Cremona \cite{banwait} and Zywina \cite{zywina} made crucial use of the fact that $\mathrm{SL}_2(\Z)$ is generated by $T$ and $S$, which satisfy $\gamma_K^{-1}T\gamma_K=\widetilde{T}_K$ and $\gamma_K^{-1}S\gamma_K=W_{K^2}$. A similar statement can be obtained for our ``parent group'' $\Gamma_0(M)$. 
\begin{lemma}
Suppose that $K>1$ satisfies $\mathrm{gcd}(K,M)=1$. Then $\Gamma_0(M)$ can be generated by $T$ and a finite set of matrices of the form 
\[
\begin{pmatrix} K a & b \\  M c & K d\end{pmatrix} \text{ with } a,b,c,d\in \Z.
\]
\end{lemma}
\begin{proof}
Let $T,\gamma_1,\ldots,\gamma_{n}$ be a set of generators for $\Gamma_0(M)$. We add $\gamma_0:=\begin{pmatrix} 1 & 0 \\ M & 1\end{pmatrix}$ to this set. We note that 
\[
\gamma_0^k\begin{pmatrix} a & b \\ cM & d\end{pmatrix} = \begin{pmatrix} a & * \\ (ka+c)M&*\end{pmatrix}.
\]
Consider $\gamma=\begin{pmatrix}a & b \\ cM & d \end{pmatrix}\in \Gamma_0(M)$. Then $\mathrm{gcd}(\det(\gamma),K)=1$ implies that $\mathrm{gcd}(a,cM,K)=1$. Therefore, given the finite set $S$ of primes dividing $K$, we can choose $k\in \Z$ such that $(ka+c)M$ is not divisible by any prime in $S$. So after replacing each $\gamma_i$ for $i>0$ by $\gamma_0^{k_i}\gamma_i$ for some $k_i\in \Z$, we may assume that the bottom left entry of each $\gamma_i$ ($i\in \{0,\ldots,n\}$) is coprime to $K$.

%Let $p$ be a prime dividing $K$. We first show that we can replace these $\gamma_i$ with generators that have their bottom left coefficient coprime to $p$. Not all of the generators can have their bottom left coefficient divisible by $p$; else the group they generate would be contained in $\Gamma_0(Mp)$. Suppose that $\gamma_1=\begin{pmatrix} a & b \\ M\cdot c & d\end{pmatrix}$ satisfies $p\nmid c$. Now suppose that $\gamma_2=\begin{pmatrix} a' & b' \\ p\cdot M \cdot c' & d'\end{pmatrix}$. Then we multiply:
%\[
%\gamma_2\gamma_1 = \begin{pmatrix} * & * \\ pMc'a + d'Mc & * \end{pmatrix}.
%\]
%Note that $p\nmid d'$ because $\mathrm{det}(\gamma_2)=1$, hence $p\nmid (pMc'a+d'MC)$. So replacing $\gamma_2$ by $\gamma_2\gamma_1$, and repeating this for the other generators, we find that $\gamma_1,\ldots,\gamma_n$ all have bottom left coefficient not divisible by $p$.

Next,  let $\gamma=\begin{pmatrix} a & b \\ c M & d\end{pmatrix}$ be one of the $\gamma_i$. We note that
\[
T^k\gamma T^{k'}=\begin{pmatrix} a+kcM & * \\ c M & d+k'cM\end{pmatrix}.
\]
As $cM$ is invertible mod $K$, we can find $k,k'\in \Z$ such that $kcM\equiv -a \mod K$ and $k'cM\equiv -d \mod K$. In particular, $T^k\gamma T^{k'}$ has its top left and bottom right coefficient divisible by $K$. We thus replace $\gamma$ by $T^k\gamma T^{k'}$. Doing this for each $\gamma_i$, we obtain a suitable set of generators..
\end{proof}
\begin{corollary}
\label{gencor}
The group $\Gamma_0(M)$ can be generated by $T$ and elements $\gamma_1,\ldots,\gamma_n$ such that $K\gamma_K^{-1}\gamma_i\gamma_K$ is an Atkin--Lehner matrix $W_{K^2}(x,y,z,w)$. Moreover, if $m$ is a prime power dividing $M$ maximally, and $W_m(a,b,c,d)$ is a corresponding Atkin--Lehner matrix at level $M$, then there exist $\mu,\nu \in \Gamma_0(M)$ such that $K\gamma_K^{-1}\mu W_{m}(a,b,c,d)\nu\gamma_K = W_{mK^2}(x,y,z,w)$ for some choice of $x,y,z,w$.  
\end{corollary}
\begin{proof}
For the first statement, let $\gamma=\begin{pmatrix} K a & b \\ M c & K d\end{pmatrix}$ be one of the generators $\gamma_i$ from the previous lemma. Then $K\gamma_K^{-1}\gamma\gamma_K=W_{K^2}(a,b,c,d)$ at level $K^2M$. 
%\[
%K\gamma_K^{-1}\gamma\gamma_K= \begin{pmatrix} K^2a & b \\ K^2\cdot M \cdot c & K^2\cdot d \end{pmatrix}.
%\]
%Moreover, this matrix has determinant $K^2$, so it is equal to $w_{K^2}(a,b,c,d)$ at level $N=K^2M$. 
Next, we consider $W_m(a,b,c,d)=\begin{pmatrix} m a & b \\ M   c & m d\end{pmatrix}$. As $\det(W_m(a,b,c,d))$ is coprime to $K$, the proof of the previous lemma allows us to find $\mu,\nu\in \Gamma_0(M)$ such that $\mu W_m(a,b,c,d)\nu $ is of the form $\begin{pmatrix} K m x & y \\  M z & K m w\end{pmatrix}$. (When choosing the powers $k,k'$ of $T$ to multiply with, these need to be chosen to be multiples of $m$, which is possible due to the Chinese Remainder Theorem.) Now conjugating with $\gamma_K$ and multiplying with $K$ gives us $W_{mK^2}(x,y,z,w)$, as desired. 
\end{proof}
%Conversely, we note that $\frac{1}{K}\gamma_Kw_{mK^2}(x,y,z,w)\gamma_K^{-1}=w_m(x,Ky,Kz,w)$.
%Also, $\gamma_K^{-1}T\gamma_K=\widetilde{T}_K$.
\begin{corollary}
\label{generatorcor}
Suppose that $\gamma\in \mathrm{GL}_2^+(\Q)$ normalises $\Gamma_0(M)$. Then $[\gamma_{K}^{-1}\gamma\gamma_K]\in \T_{M,K}$. 
%Consider again $A=\CC[\Gamma_{MK}^{\mathcal{W}}]$. Then $A$ is generated as a $\CC$-algebra by all $w_{mK^2}(x,y,z,w)\in \mathcal{W}$ and all $R_{\chi}(K)$ for $\chi\in \mathcal{D}_K$. 
\end{corollary}
\begin{proof}
By Lemma \ref{normaliserlemma}, we may suppose $\gamma$ is a generator of $\Gamma_0(M)$, an Atkin--Lehner matrix, or $\widetilde{T}_{h}$. The result thus follows from the previous corollary and Lemma \ref{longlemma2} (e). 
\end{proof}
%We have already seen that each Atkin--Lehner operator $w_{mK^2}(x,y,z,w)$ commutes with $J$ (Lemma \ref{atkinlehnercommutes}), and that $JR_{\chi}(K)J=\chi(-1)R_{\chi}(K)$.
%
%It follows that $X=X^+\sqcup X^-$, where $X^+$ is the subgroup of elements commuting with $J$, and $X^-$ the subset of elements anti-commuting with $J$. Moreover, $[X:X^+]=2$ and $X^+$ is generated by the Atkin--Lehner operators $w_{mK^2}(x,y,z,w)$ and the twisting operators $R_{\chi}(K)$ for $\chi$ with $\chi(-1)=1$. 
%\begin{corollary}
%$A^+=\CC[X^+]$
%\end{corollary}
The next step is to describe a decomposition of $H_1(X_{\Gamma_0(MK^2)\cap \Gamma_1(K)},\CC)$ into sub-$\T_{M,K}$-modules. 
%\begin{definition}
%Consider a Hecke eigenvector $\mu \in H_{\CC}(\Gamma_0(MK^2)\cap \Gamma_1(K))$. The \emph{twist orbit space} of $\mu$ is the $X$-module $O(\mu):=A\mu$. The \emph{real twist orbit space} of $\mu$ is the $A^+$-module $O(\mu)^+:=\CC[X^+]\mu$. 
%\end{definition}
%We note that $O(\mu)^+$ is the space of $J$-invariants in $O(\mu)$. 

%Before describing twist orbit spaces in more detail, we determine the action of $\mathrm{pr}_{N/L}$ on certain eigenforms.
Each Hecke eigenspace $H\subset H_1(X_{\Gamma_0(MK^2)\cap \Gamma_1(K)},\CC)$ splits as a direct sum $H=H^+\oplus H^-$, where $H^+$ is the $+1$-subspace for $J$, and $H^-$ is the $-1$-subspace. We call $H^+$ and $H^-$ \emph{Hecke-$J$-eigenspaces}. Given a Hecke-$J$-eigenvector $\mu\in H_1(X_{\Gamma_0(MK^2)\cap \Gamma_1(K)},\CC)$, denote by $H(\mu)$ its Hecke-$J$-eigenspace. This space is 1-dimensional when $\mu$ is new, but not when $\mu$ is old. 

\begin{definition}
Suppose that $\mu\in H_1(X_{\Gamma_0(MK^2)\cap \Gamma_1(K)},\CC)$ is a Hecke-$J$-eigenvector. We define the \emph{twist orbit space of $\mu$} to be
\[
O(\mu):=\bigplus_{\chi\in \mathcal{D}_{Kh}}H(R_{\chi}(\mu)).
\]
Note that $R_{\chi}(\mu)$ is indeed still a Hecke-$J$-eigenvector by Lemmas \ref{longlemma2} (d) and \ref{operatortranslation}. The subspace $O(\mu)^+$ on which $J$ acts with $+1$-eigenvalues is called the \emph{real twist orbit space of $\mu$}. 
\end{definition}
By Lemma \ref{operatortranslation}, if $J\mu=\mu$, we have
\[
O(\mu)^+=\bigplus_{\chi \in \mathcal{D}_{Kh} \text{ even}}H(R_{\chi}(\mu)).
\]
\begin{proposition}
\label{twistorbitprop}
 Suppose that $\mu\in H_1(X_{\Gamma_0(MK^2)\cap \Gamma_1(K)},\CC)$ is a Hecke-$J$-eigenvector. Then $O(\mu)$ is a representation of $\T_{M,K}$. %The twist orbit space $O(\mu)$ is contained in the span $V(\mu)$ of the modular symbols
%\[
%\mathrm{Tr}_e^{N/L}\pi_{N/L}R_{\chi}(\mu),
%\]
%where $\chi \in \mathcal{D}_K$, $R_{\chi}(\mu)$ has new level $L\mid N$, and  $e$ divides $N/L$. The real twist orbit space $O(\mu)^+$ is contained in the span $V^+(\mu)$ of the subset of these where $\chi(-1)=1$. 
\end{proposition}
\begin{proof}
%We first note that $V(\mu)$ is a direct sum of Hecke-$J$-eigenspaces: Hecke eigenspaces intersected with one of the two  $\pm1$-eigenspaces for $J$. Indeed, the operators $\mathrm{Tr}_e^{N/L}\pi_{N/L}$ are designed to ensure the entire Hecke-$J$-eigenspace containing $R_{\chi}(\mu)$ is covered. 

Since $O(\mu)$ is a direct sum of Hecke-$J$-eigenspaces, it suffices to determine how the twist and Atkin--Lehner operators act on the set of Hecke-$J$-eigenspaces. For the Atkin--Lehner operators $w_{Q}(x,y,z,w)$ we use Lemmas \ref{longlemma1} (c), (d) and \ref{longlemma2} (d) to see that  the Hecke eigenspace of $\mu$ is mapped to the Hecke eigenspace of $R_{\ep}(\mu)$, where $\ep$ is the Nebentypus character of $\mu$. Furthermore, Atkin--Lehner operators commute with $J$ by Lemma \ref{atkinlehnercommutes}, whereas $JR_{\ep}J=\ep(-1)R_{\ep}$. Note, however, that $\ep(-1)=1$ unless $\mu=0$, so indeed $H(\mu)$ is mapped to $H(R_{\ep}(\mu))$. 

Finally, for $\chi\in \mathcal{D}_{Kh}$, the operator $R_{\chi}(Kh)$ maps a Hecke-$J$-eigenspace to the same Hecke-$J$-eigenspace as $R_{\chi}$ does, and hence preserves $O(\mu)$ by definition. %act the same way on Hecke-$J$-eigenspaces by Lemma \ref{longlemma1} (d) and Lemma \ref{transformationlemma}.  %whereas $R_{\chi}(K)$ maps a Hecke-$J$-eigenspace to the same Hecke-$J$-eigenspace as $R_{\chi}$ does. 
%
%Now, by definition $O(\mu)$ is spanned by the images $R_{\chi}(K)(\mu)$ and $w_{mK^2}(x,y,z,w)(\mu)$, ranging over all $\chi$ with $\mathrm{cond}(\chi)\mid K$ and all choices of $x,y,z,w$. Suppose that $\mu$ has Nebentypus character $\ep$. First recall that $\ep(-1)=1$ (unless $\mu=0$). By Lemma \ref{longlemma2} (i), we find that $w_{mK^2}(x,y,z,w)(\mu)$ and $R_{\overline{\ep}}(\mu)$ are in the same Hecke-$J$-eigenspace, which suffices.
%
%Next, we consider $R_{\chi}(K)(\mu)$ when $\mathrm{cond}(\chi)\neq K$ and $\chi$ is not the trivial character. Then Lemma \ref{longlemma2} parts (d) and (b) show that again $R_{\chi}(\mu)$ and $R_{\chi}(K)(\mu)$ are in the same Hecke-$J$-eigenspace. Finally, Lemma \ref{longlemma2} (e) shows that $R_{\chi_0}(K)$ preserves Hecke-$J$-eigenspaces.
%
%This finishes the proof for $O(\mu)$. For the real subspaces, it now suffices to recall that $\ep$ always satisfies $\ep(-1)=1$, unless $\mu=0$. 
%
%NOT COMPLETE $\{R_{\chi}(\mu) \mid \mathrm{cond}(\chi)\mid K\}$ if all of these twists $R_{\chi}(\mu)$ are new. Otherwise, whenever $R_{\chi}(\mu)$ is old, replace in this set $R_{\chi}(\mu)$ by a basis for the Hecke eigenspace corresponding to $R_{\chi}(\mu)$. One can obtain a basis for this Hecke eigenspace by applying combinations of the operators $B_d\circ \mathrm{Tr}_e^{N/L}$ to $R_{\chi}(\mu)$. 
\end{proof}
\begin{proposition}
\label{basisprop}
Suppose that $f\in S_2(\Gamma_0(MK^2)\cap \Gamma_1(K),\CC)$ is a newform at a level dividing $N$ and $\mu=\gamma_f$. Then a basis for $O(\mu)^+$ is given by the elements
\[
\mathrm{Tr}_e^{N/L}\pi_{N/L}R_{\chi}(\mu),
\]
where $\chi \in \mathcal{D}_{Kh}$ is even, $R_{\chi}(\mu)$ has new level $L\mid N$, and  $e$ divides $N/L$. 
\end{proposition}
\begin{proof}
Translate this statement to modular forms using (\ref{prequivalence}) and Lemma \ref{Bdtranslationlemma}, then apply Corollary \ref{projcor} and the standard theory of old and new subspaces.
\end{proof}
We note that each real twist orbit space contains an element $\mu=\gamma_f$, where $f$ is a newform at a level dividing $N$, and thus has a basis of this form. In this case, denote by $O^+(f)$ the analogue of $O(\mu)^+$ under the isomorphism $(\ref{pairingiso})$.  We call this the \emph{real twist orbit space} of $f$. We note that each eigenform $f$ has Fourier coefficients defined over some number field. Denote by 
\[
V(f)^+\subset S_2(\Gamma_0(MK^2)\cap \Gamma_1(K),\overline{\Q})
\]
the subspace  spanned by the real twist orbit spaces $O(f^{\sigma})^+$, where $\sigma\in \mathrm{Gal}(\overline{\Q}/\Q)$.

\subsection{Computing $q$-Expansions}
We consider again the set-up from the previous section. Now assume that $JGJ=G$, $-I\in G$ and $\mathrm{det}(G)=(\Z/N\Z)^{\times}$. Consider also a group $\mathcal{A}\subset (\mathcal{N}_{\Gamma_0(M)}\cap \mathcal{N}_{G^h})/\mathrm{P}\Gamma$ such that $J\mathcal{A}J=\mathcal{A}$. Here $G^h$ is the subgroup of $G$ with determinant $1\mod h$, and $h$ is the largest divisor of 24 such that $h^2\mid M$. Then the group $F:=\langle \mathrm{P}\Gamma,\mathcal{A}\rangle$ satisfies $JFJ=F$ and $F\subset \mathcal{N}_{\Gamma_0(M)}$. (Technically, we should replace $\mathcal{A}$ by a set of lifts to $\mathcal{N}_{\Gamma_0(M)}$, but note that $F$ is independent of the choices of lifts.) Also define $V:=H_1(X_{\Gamma_0(MK^2)\cap \Gamma_1(K)},\CC)$ and decompose $V=\oplus_i O_i$  into distinct twist orbit spaces.

 By Corollary \ref{generatorcor} and Proposition \ref{twistorbitprop},  $F$ acts on each $O_i$ separately, and
$
V^{F}=\bigoplus_i O_i^{F}.
$
Now $J$ acts on each $O_i^F$ separately, and we obtain 
\[
V^{F,+}=\bigoplus_i O_i^{F,+}, \text{ where } O_i^{F,+} \subset O_i^+
\]
and the superscript $+$ denotes the $J$-invariant subspace.
This means that $V^{F,+}$ has a basis  of elements of real twist orbit spaces. In order to determine $q$-expansions for a basis of $S_2(\Gamma,\CC)^{\mathcal{A}}$, it thus suffices to be able to compute $q$-expansions for cusp forms in a fixed real twist orbit space.

  To this end, recall that we have a sesqui-linear isomorphism
\[
S_2(\Gamma_0(MK^2)\cap \Gamma_1(K),\CC)\to H_{\CC}(\Gamma_0(MK^2)\cap \Gamma_1(K)),\;\; f \mapsto \gamma_f,
\]
under which the operators correspond as follows:
\begin{align*}
T_p & \longleftrightarrow T_p^*,\;\;
R_{\chi}  \longleftrightarrow R_{\overline{\chi}} \text{ if } \chi(-1)=1,\;\;
\mathrm{pr}_{N/L}\longleftrightarrow \pi_{N/L},\;\;
[\Gamma_0(N):\Gamma_0(L)]B_d^{N/L}\longleftrightarrow \mathrm{Tr}_d^{N/L}.
\end{align*}
This means the following. 
\begin{proposition}
\label{translationprop}
Suppose that $f\in S_2(\Gamma_0(MK^2)\cap\Gamma_1(K),\CC)$ and $\mu=\gamma_f\in H_{\CC}(\Gamma_0(MK^2)\cap \Gamma_1(K))$. For each $i\in \{1,\ldots,n\}$, consider even $\chi_i\in \mathcal{D}_{Kh}$ and $a_i\in \CC$. Let $L_i\mid N$ be the new level of $f|R_{\overline{\chi}_i}$, and consider $e_i\mid N/L_i$. Then
\[
\sum_{i=1}^na_i\mathrm{Tr}_{e_i}^{N/L_i}\pi_{N/L_i}R_{\chi_i}(\mu)=\gamma_g,
\]
where
\[
g=\sum_{i=1}^n[\Gamma_0(N):\Gamma_0(L_i)]\cdot\overline{a_i}\cdot f|R_{\overline{\chi}_i}\mathrm{pr}_{N/L_i}B^{N/L_i}_{e_i}. 
\]
\end{proposition}
From the action of $T_p$ on newforms we can determine $q$-expansions of newforms. Moreover, we know how $B_e$ and $R_{\chi}$ act on $q$-expansions, and by Corollary \ref{projcor}, we can also determine what $\mathrm{pr}_{N/L}$ does to $q$-expansions. So, given a $q$-expansion for $f$ in the above proposition, we can compute a $q$-expansion for $g$. This leads to the following algorithm.\\

%We know how $R_{\chi}$ and $B_e$ act on $q$-expansions. We would like to know how $\mathrm{pr}_L$ acts.  Now let $f\in S_2(\Gamma_0(MK^2)\cap \Gamma_1(K))$ be the Hecke eigenform corresponding to $\mu$, i.e. with the same Hecke eigenvalues and (if it is an old form) a newform embedded in the same way. Scale $f$ so that $a_1(f)=1$. Then there is $\al\in \CC$ such that $\mu=\al \gamma_f$.
%
%By Lemmas \ref{Rchitranslationlemma} and \ref{Bdtranslationlemma}, we then find that $\nu=\al \gamma_g$, where
%\[
%g=\sum_{i=1}^n\overline{a_i}\cdot f|R_{\overline{\chi}_i}\mathrm{pr}_{L_i}B_{e_i}. 
%\]
%From its Hecke eigenvalues, we find the $q$-expansion for $f$. We thus know the $q$-expansions for $f|R_{\overline{\chi}_i}=g(\overline{chi}_i)f\otimes \overline{\chi}_i$. We also know how each $B_{e_i}$ acts on $q$-expansions. If $f|R_{\overline{\chi}_i}$ is an old form, we also know how $\mathrm{pr}_{L_i}$ acts, by the following Lemma.
%So, if $\eta = \pi_L R_{\overline{\chi}}(\mu)$ then $\eta = \al \gamma_h$. We can thus compute a $q$-expansion for $g$. \\

\begin{algorithm}[For computing $q$-expansions for a basis of $S_2(\Gamma,\Q(\zeta_{Kh})^+)^{\mathcal{A}}$]
\label{algorithm1}
\emph{Input}: Coprime integers $K$ and $M$, a finite subset $S$ of $\mathrm{GL}_2^+(\Q)$ generating a subgroup $F\subset \mathcal{N}_{\Gamma_0(M)}$ normalised by $J$, and  a positive integer \emph{\texttt{prec}}. Let $h$ be the largest divisor of 24 such that $h^2\mid M$. 

\emph{Steps}:
\begin{itemize}
%\item[(1)] Lift each element of $S$ to $\Gamma_0(M)$. The lifts generate a congruence subgroup $\Gamma$ of level $KM$. Together with $\mathcal{W}$, they make  up a finite set $F$.
    \item[(1)] Find a basis for $H_1(\Gamma_0(MK^2)\cap \Gamma_1(K),\Q)$ and determine the fixed subspace \newline $H_1(\Gamma_0(MK^2)\cap \Gamma_1(K),\Q(\zeta_{Kh}))^{F}$. 
    \item[(2)] Choose a divisor $L=L_1L_2^2\mid K^2M$ with $L_1\mid M$ and $L_2\mid K$, and a newform $f \in S_2(\Gamma_0(L)\cap \Gamma_1(L_2),\overline{\Q})$. Compute the $q$-expansion for $f$ up to $q^{\texttt{\emph{prec}}+1}$.
    \item[(3)] Using Corollary \ref{projcor}, compute the $q$-expansions (up to $q^{\texttt{\emph{prec}}+1}$) for all \newline$f|R_{\overline{\chi}}\mathrm{pr}_{N/L}B^{N/L}_{e}$, where $\chi\in \mathcal{D}_{Kh}$ is even, $L\mid N$ is the new level of $f|R_{\overline{\chi}}$ and $e\mid N/L$. 
    \item[(4)] Using (finitely many) $T_p$-operators and their eigenvalues, compute\newline $\mu\in H_{\CC}(\Gamma_0(MK^2)\cap \Gamma_1(K))$ such that $\mu=\al \gamma_f$ for some $\al\in \overline{\Q}^{\times}$.
    \item[(5)] Compute the modular symbols $\mathrm{Tr}^{N/L}_{e}\pi_{N/L}R_{\chi}(\mu)$ corresponding to the modular forms computed in (3). The space $O(\mu)^+$ they generate is a real twist orbit space.
    \item[(6)] Determine a basis for the intersection $O(\mu)^+\cap H_1(\Gamma_0(MK^2)\cap \Gamma_1(K),\overline{\Q})^{F}$, as linear combinations of the modular symbols computed in Step (5).
    \item[(7)] Using Proposition \ref{translationprop} and the result from Step (3), determine the $q$-expansions of the modular forms corresponding to the basis computed in Step (6).
    \item[(8)] Return to Step (2), unless the spaces $O(\mu)^+$ considered so far span\newline $H_{\CC}(\Gamma_0(MK^2)\cap \Gamma_1(K))^F$. 
    \item[(9)] The $q$-expansions for the computed basis are defined over a number field containing $\Q(\zeta_{Kh})^+$. Take linear combinations to reduce to a basis over $\Q(\zeta_{Kh})^+$.
\end{itemize}
\emph{Output:} A finite set of $q$-expansions $a_1q+a_2q^2+\ldots+a_{\texttt{\emph{prec}}+1}q^{\texttt{\emph{prec}}+1}$ with each $a_i\in \Q(\zeta_{Kh})^+$, corresponding to a basis for $S_2(\Gamma_0(MK^2)\cap \Gamma_1(K),\Q(\zeta_{Kh})^+)^F$. When $\mathcal{A}$, $\Gamma$ and $F$ are defined as at the start of this section, this space equals $S_2(\Gamma,\Q(\zeta_{Kh})^+)^{\mathcal{A}}$. 
\end{algorithm}
\begin{proof}
The discussion at the start of this section shows that we indeed obtain a basis for the space $S_2(\Gamma_0(MK^2)\cap \Gamma_1(K),\overline{\Q})^F$. In (\ref{1formsiso}) in Section \ref{normalisersec}, we saw that a basis of $q$-expansions with coefficients in $\Q(\zeta_{Kh})$ exists, and, as $J$ normalises $F$, we can reduce further to $\Q(\zeta_{Kh})^+$.
\end{proof}

\subsection{The $\Q$-rational structure}\label{rationalstrsec}
We continue the notation from the previous  section. Assume also that $\mathcal{A}$ determines a $\mathrm{Gal}(\overline{\Q}/\Q)$-invariant automorphism \emph{group}. Then we know that a model for $X_G/\mathcal{A}$ over $\Q$ must exist. Given the modular forms in $S_2(\Gamma,\Q(\zeta_{Kh})^+)^{\mathcal{A}}$ computed using Algorithm \ref{algorithm1}, we can compute \emph{some} model for $(X_G/\mathcal{A})_{\Q(\zeta_{Kh})^+}$, but this model tends to be defined over $\Q(\zeta_{Kh})^+$ rather than over $\Q$. 

We want to compute the fixed space for the action of the remaining matrices in $G\subset \mathrm{GL}_2(\Z/N\Z)$ of determinant unequal to 1, as defined in Section \ref{sec31}. This action on cusp forms is only $\Q$-linear, as opposed to $\CC$-linear, which complicates matters: the pairing (\ref{pairing}) between modular forms and modular symbols, being defined by integration, does not descend to a pairing between two $\Q$-vector spaces.  

We first prove a Sturm bound. For a curve $X$ and a divisor $D$ on $X$, denote by $\Omega$ the sheaf of regular differential 1-forms, and by $\O_X(D)$ the sheaf such that for each open $U\subset X$, the set $\O_X(D)(U)$ consists of those functions $f$ in the function field of $X$ satisfying $\mathrm{div}(f|_U)\geq -D\cap U$. 
\begin{lemma}\label{sturmbound}
Suppose that $X$ is a curve, and $\omega \in H^0(X,\Omega^{\otimes k})$. We consider a point $x\in X$ and a uniformiser $q\in \widehat{\mathcal{O}}_x$ in the completed local ring. Write the image of $\om$ in $(\Omega^{\otimes k})_x$ as $\om_x = (a_1+a_2q+a_3q^2+\ldots)\left(\dd q\right)^{\otimes k}$. If $a_i=0$ for all $i\in \{1,\ldots,k(2g-2)+1\}$ then $\om=0$.
\end{lemma}
\begin{proof}
This is a standard argument. Let $K$ be the canonical divisor. By assumption, $\om$ yields a global section of the sheaf $\O_X(kK-(k(2g-2)+1)x)$. This sheaf has degree $k\mathrm{deg}(K)-k(2g-2)-1<0$, and therefore has no non-zero global sections. 
\end{proof}

\begin{algorithm}[For computing a model for $X_G/\mathcal{A}$ over $\Q$]
\label{algorithm2}
\emph{Input}: A finite set of matrices generating a subgroup $G\subset \mathrm{GL}_2(\Z/N\Z)$ with $\mathrm{det}(G)=(\Z/N\Z)^{\times}$, $-I\in G$ and $JGJ=G$. An integer $M$ such that $N=MK$ with $\mathrm{gcd}(M,K)=1$ and $G$ surjects onto $B_0(M)\subset \mathrm{GL}_2(\Z/M\Z)$. Let $h$ be the largest divisor of 24 such that $h^2\mid M$. A finite subset of $\mathrm{GL}_2^+(\Q)$ generating a subgroup  $\mathcal{A}\subset (\mathcal{N}_{\Gamma_0(M)}\cap \mathcal{N}_{G^h})/\mathrm{P}\Gamma_G$ normalised by $J$, and defining a $\mathrm{Gal}(\overline{\Q}/\Q)$-invariant automorphism group on $X_G$ such that $X_G/\mathcal{A}$ has genus $g\geq 2$. A positive integer \emph{\texttt{prec}}. 

\emph{Steps}:
\begin{itemize}
    \item[(1)] Define $F:=\langle \Gamma_G,\mathcal{A}\rangle$ and run Algorithm \ref{algorithm1} for $F,K,M,h$, \emph{\texttt{prec}} to determine $q$-expansions for a basis of $S_2(\Gamma_G,\Q(\zeta_{Kh})^+)^{\mathcal{A}}$, and reduce to a basis of $S_2(\Gamma_G,\Q(\zeta_{K}^+))^{\mathcal{A}}$. The dimension of this space is the genus $g$ of $X_G/\mathcal{A}$.
    \item[(2)] Find generators $A_1,\ldots,A_n,A_{n+1},\ldots,A_{n'}$ for the image $G_{K}$ of $G$ in $\mathrm{GL}_2(\Z/K\Z)$ such that $A_{n+1},\ldots,A_{n'}$ generate $G_{K}\cap \mathrm{SL}_2(\Z/K\Z)$ and $\mathrm{det}(A_i)\neq 1$ for all $i\in \{1,\ldots,n\}$.
    \item[(3)] For each $A\in \{A_1,\ldots,A_n\}$, run Steps (4)-(6). 
    \item[(4)] Write $A=B\cdot \gamma_d$, where $d=\mathrm{det}(A)$ and $\gamma_d=\begin{pmatrix} d & 0 \\ 0 & 1\end{pmatrix}$. Determine a lift $\widetilde{B}$ of $B$ to $\Gamma_0(M)\subset \mathrm{SL}_2(\Z)$. For each real twist orbit space $O(\mu)^+$ computed in Algorithm \ref{algorithm1},  $\widetilde{B}$ maps $O(\mu)^{\Gamma_G,+}$ into $O(\mu)^+$. Run steps (5)-(6) for each such real twist orbit space.
    \item[(5)] As in Algorithm \ref{algorithm1}, compute the action of $\widetilde{B}$ on  $O(\mu)^{\Gamma_G,+}$, and translate this to a $\Q(\zeta_{K})$-linear map $\widetilde{B}:\;V(f)^+\cap S_2(\Gamma_G,\Q(\zeta_{K}))\to V(f)^+$ in terms of the basis computed in Step (1).
    \item[(6)] Using a basis for $\Q(\zeta_{K})/\Q$, consider the linear map computed in (5) as a $\Q$-linear map, and postcompose with the action of $\gamma_d:\; V(f)^+\to V(f)^+$, which Galois conjugates coefficients in a $q$-expansion with $\sigma_d: \zeta_{K}\to \zeta_{K}^d$. This composition yields a map $V(f)^+\cap S_2(\Gamma_G,\Q(\zeta_{K}))\to V(f)^+\cap S_2(\Gamma_G,\Q(\zeta_{K}))$ corresponding to the action of $A$ on $q$-expansions.
    \item[(7)] For each $V(f)^+$, determine $q$-expansions of a basis for $(V(f)^+\cap S_2(\Gamma_G,\Q(\zeta_{K})))^{\mathcal{A},A_1,\ldots,A_n}$, up to the precision needed for the next step.
    \item[(8)] Use the $q$-expansions to determine equations for $X_G/\mathcal{A}$, as done by Galbraith \cite{galbraith}. Unless $X_G/\mathcal{A}$ is hyperelliptic, this is done by computing homogeneous $\Q$-rational polynomial equations satsified by the $q$-expansions, up to $q^\emph{\texttt{prec}}$. To show an equation of degree $d$ holds provably, we need $\emph{\texttt{prec}}>d(2g-2)-(d-1)$, where $g=\mathrm{dim}(S_2(\Gamma_G,\Q(\zeta_{K}))^{\mathcal{A}})$. For the details when $X_G/\mathcal{A}$ is hyperelliptic, we refer to \cite{galbraith}. 
\end{itemize}
Output: A finite number of homogeneous polynomials over $\Q$ in $g+1$ variables if $X_G/\mathcal{A}$ is non-hyperelliptic, a single polynomial in 2 variables if $X_G/\mathcal{A}$ is hyperelliptic.
\end{algorithm}
\begin{remark}\label{degreermk}
By Petri's theorem (see \cite{petri}), the canonical image of a non-hyperelliptic curve of genus $g\geq 4$ is an intersection of quadrics unless the curve is trigonal or (isomorphic to) a plane quintic. It thus often suffices to take $\texttt{prec}>4g-5$. In the trigonal and plane quintic cases, the canonical image is defined by equations of degree 4, while for genus 3 curves the canonical image is cut out by equations of degrees up to 3 (see \cite[Lemma 7.1]{zywina}). In every case, therefore, $\texttt{prec}>8g-11$ suffices. 
\end{remark}
\begin{proof}
This algorithm computes the right object because $S_2(\Gamma_G,\Q(\zeta_{K}))^{\mathcal{A},A_1,\ldots,A_n}$ is equal to $H^0(X_G/\mathcal{A},\Omega^1)$, as shown in Section \ref{normalisersec}. 

In Step (1), a basis over $\Q(\zeta_K^+)$ exists because $\mathcal{A}$ is Galois invariant. The crucial claim in Step (4) that $\widetilde{B}$ maps $O(\mu)^{\Gamma_G,+}$ into $O(\mu)^+$ requires proof. Since $A\in G_{K}$ and $JAJ\in G_{K}$, we find that $JAJA^{-1}\in G_{K}\cap \mathrm{SL}_2(\Z/K\Z)$. So $JAJ$ and $A$, though different operators in general, have identical restriction to\[ (H_1(X(\Gamma_0(MK^2)\cap \Gamma_1(K)),\R)\otimes \CC)^{\Gamma_G}.
\]
In particular, $A$ acts maps the +-eigenspace $H_{\CC}(\Gamma_0(MK^2)\cap \Gamma_1(K))^{\Gamma_G}$ into the space $H_{\CC}(\Gamma_0(MK^2)\cap \Gamma_1(K))$. Since $\gamma_d$ commutes with $J$ (even as matrices), we conclude that $\widetilde{B}$ maps  $H_{\CC}(\Gamma_0(MK^2)\cap \Gamma_1(K))^{\Gamma_G}$ into $H_{\CC}(\Gamma_0(MK^2)\cap \Gamma_1(K))$. Because $\widetilde{B}\in \Gamma_0(M)$ moreover acts on twist orbit spaces, this proves the claim. Furthermore, $\gamma_d$ acts on $V(f)^+$, since for $\sigma:\zeta_{K}\to \zeta_{K}^d$, we have
\[
(f|R_{\chi}\mathrm{pr}_LB_d)^{\sigma}=g(\overline{\chi})^{\sigma}(f^{\sigma}\otimes \chi^{\sigma})|\mathrm{pr}_LB_d.
\]
Step (5) again relies on Propositions \ref{twistorbitprop} and \ref{translationprop}. Steps (6) and (7) are linear algebra computations, and for the details of Step (8) we refer to \cite{galbraith}. We recall that $S_{2k}(\Gamma_G,\Q(\zeta_{K}))^{\mathcal{A}}=H^0(X/\mathcal{A},\Omega^{\otimes k})$. The lower bound on \texttt{prec} follows from Lemma \ref{sturmbound} by noting that the coefficient for $q^{\ell}$ in the expansion of a product of $d$ cusp forms is determined by the first $\ell-(d-1)$ coefficients of each cusp form. %If $g\leq 1$, Galbraith's approach does not work and one needs to instead determine  $S_{2k}(\Gamma_G,\Q(\zeta_{K}))^{\mathcal{A}}$ for sufficiently many $k\geq 1$, see \cite{muic}.  
\end{proof}
\begin{remark}
If $\al\in \mathcal{N}_{G}/\mathrm{P}\Gamma_G$ commutes with each element of $\mathcal{A}$, then $\al$ determines an automorphism of $X_G/\mathcal{A}$. If $\al$ also commutes with $J$, the action of $\al$ on each real twist orbit space can be determined explicitly, which can be translated into a matrix for the action of $\al$ on the basis of cusp forms found in Step (7). From this, we obtain the automorphism on $X_G/\mathcal{A}$ determined by $\al$ explicitly.

Moreover, having determined the $q$-expansions for generators of the function field of $X_G/\mathcal{A}$, we can determine maps to other modular curves, such as the $j$-map.
\end{remark}
We note that the $q$-expansions computed in Step (7) still do not have $\Q$-rational coefficients in general, but do satisfy rational equations. Ultimately, this is because the image of the infinity cusp on (a rational model for) $X_G/\mathcal{A}$ is -- in general -- not a $\Q$-rational point, so the expansion of a regular differential form at this cusp also tends to have non-rational coefficients.

%
%Suppose that $A\in G$. Then $JAJ\in G$ so that $JAJA^{-1}\in G$ with determinant 1. In particular, $JAJA^{-1}$ acts trivially on $S_2(\Gamma,\Q(\zeta_N))$, and $JAJ$ and $A$, though different operators in general, have the same restriction to $S_2(\Gamma,\Q(\zeta_N))$. In particular, $A$ acts on $H_1(X_{\Gamma},\Q(\zeta_N))^+$. Now write $A=B\cdot P_d$. Since $P_d$ commutes with $J$ -- even as matrices --, we find that $B$ maps $H_1(X_{\Gamma},\Q(\zeta_N))^+$ into the +1-eigenspace of $J$. This means that we can translate the action of $B$ on $H_1(X_{\Gamma},\Q(\zeta_N))^+$ to a corresponding action on modular forms using Lemma \ref{transformationlemma}. Then on modular forms we know how $P_d$ acts, and we can compute the action of $A$ and its fixed subspace. This yields an explicit basis of modular forms, not necessarily with rational Fourier coefficients, which satisfy rational equations.
\section{Examples}
\label{examplesection}
\subsection{The three curves}Consider the subgroup
\begin{align}
\label{ge7}
G(\mathrm{e}7)=\left\langle \begin{pmatrix} 0 & 5 \\ 3 & 0 \end{pmatrix},\begin{pmatrix} 5 & 0 \\ 3 & 2 \end{pmatrix}\right\rangle \subset \mathrm{GL}_2(\F_7)
\end{align}
defined in \cite[Proposition 9.1]{freitas}. This is an index 2 subgroup of the normaliser $G(\mathrm{ns}7^+)$ of a non-split Cartan subgroup of $\mathrm{GL}_2(\F_7)$. Also consider  $B_0(5)\subset \mathrm{GL}_2(\F_5)$, and let $G(\mathrm{b}5,\mathrm{e}7)$ be the intersection of the inverse images of $G(\mathrm{e}7)$ and $B_0(5)$ in $\mathrm{GL}_2(\Z/35\Z)$.  We obtain a degree 2 map
\[
X(\mathrm{b}5,\mathrm{e}7):=X_{G(\mathrm{b}5,\mathrm{e}7)}\to X(\mathrm{b}5,\mathrm{ns}7):=X_{G(\mathrm{b}5,\mathrm{ns}7+)}.
\]
Every degree 2 map of curves is the quotient by an involution. Let us call this involution $\phi_7: X(\mathrm{b}5,\mathrm{e}7)\to X(\mathrm{b}5,\mathrm{e}7)$. This corresponds to the action of a matrix in $G(\mathrm{ns}7^+)\setminus G(\mathrm{e}7)$ (c.f. Example \ref{ex27}) and commutes with the Atkin--Lehner involution $w_5$ (defined in Definition \ref{aldef}, using any $\eta\in G(e7)$ of determinant $5\mod 7$). We thus obtain the following diagram of degree 2 maps between modular curves
\[
\begin{tikzcd}
                                         & {X(\mathrm{b}5,\mathrm{e}7)} \arrow[d] \arrow[ld] \arrow[rd] &                                             \\
{X(\mathrm{b}5,\mathrm{ns}7)} \arrow[rd] & {X(\mathrm{b}5,\mathrm{e}7)/\phi_7w_5} \arrow[d]             & {X(\mathrm{b}5,\mathrm{e}7)/w_5} \arrow[ld] \\
                                         & {X(\mathrm{b}5,\mathrm{ns}7)/w_5}                            &                                            
\end{tikzcd}.
\]
We note that $\phi_7$ and $w_5$ represent the two subsets of allowed automorphisms described in Remark \ref{allowedsets}. We use Algorithm \ref{algorithm2} to compute canonical models for the three modular curves in the middle row, as well as their maps to $X(\mathrm{b}5,\mathrm{ns}7)/w_5$, and a basis for their global differential forms by means of modular forms. This has been used by the author in \cite{box2} as part of the proof that all elliptic curves over quartic fields not containing $\sqrt{5}$ are modular. 

We note that the genera of $X(\mathrm{b}5,\mathrm{ns}7)$, $X(\mathrm{b}5,\mathrm{e}7)/w_5$ and $X(\mathrm{b}5,\mathrm{e}7)/\phi_7w_5$ are 6, 5 and 8 respectively, while $X(\mathrm{b}5,\mathrm{e}7)$ has genus 15 and $X(\mathrm{b}5,\mathrm{ns}7)/w_5$ is hyperelliptic of genus 2. A model for $X(\mathrm{b}5,\mathrm{ns}7)$ was first computed by Le Hung \cite{lehung} using its description as a fibred product of $X(\mathrm{b}5)$ and $X(\mathrm{e}7)$, and this was transformed into a planar model by Derickx, Najman and Siksek in \cite{derickx}. This known model serves as an extra verification for the correctness of our computations. We also compute the map $X(\mathrm{b}5,\mathrm{ns}7)\to X(\mathrm{ns}7)$. Composition with Chen's $j$-map \cite[pp.69]{chen} on $X(\mathrm{ns}7)=\P^1$ 
\begin{align}\label{chenj}
j_{\mathrm{ns}7}(x)=\frac{\left((4x^2+5x+2)(x^2+3x+4)(x^2+10x+4)(3x+1)\right)^3}{(x^3+x^2-2x-1)^7}
\end{align}
thus yields the $j$-map on $X(\mathrm{b}5,\mathrm{ns}7)$.

\begin{theorem}
\label{modelthm}
A canonical model for $X(\mathrm{b}5,\mathrm{ns}7)$ in $\P_{X_0,\ldots,X_5}^5$ is given by:
\begin{align*}X(\mathrm{b}5,\mathrm{ns}7):\; & 14X_0^2 + 12X_2X_3 - 16X_3^2 - 14X_2X_4 + 30X_3X_4 - 11X_4^2\\&\;\;\;\;\; + 28X_2X_5 - 58X_3X_5 + 40X_4X_5 - 28X_5^2=0,\\& 7X_0X_1 - 2X_2X_4 - 4X_3X_4 + 2X_4^2 + 12X_3X_5 - 7X_4X_5 + 10X_5^2=0,\\& 14X_1^2 - 4X_2X_3 + 16X_3^2 + 10X_2X_4 + 14X_3X_4 - 21X_4^2\\&\;\;\;\;\; + 4X_2X_5 - 58X_3X_5 + 64X_4X_5 - 66X_5^2=0,\\& 2X_0X_2 - 2X_0X_3 + 2X_1X_3 - 5X_0X_4 - 6X_1X_4 + 8X_0X_5 + 4X_1X_5=0,\\& 4X_1X_2 - 2X_0X_3 - 6X_1X_3 - X_0X_4 + 3X_1X_4 + 3X_0X_5 - 2X_1X_5=0,\\& 8X_2^2 - 20X_2X_3 + 16X_3^2 - 14X_2X_4 + 14X_3X_4 - 21X_4^2\\&\;\;\;\;\; + 28X_2X_5 - 42X_3X_5 + 56X_4X_5 - 28X_5^2=0.
\end{align*}  
and the Atkin--Lehner involution acts as
\[
w_5:\; (X_0:X_1:X_2:X_3:X_4:X_5)\mapsto (-X_0:-X_1:X_2:X_3:X_4:X_5).
\]
The map $X(\mathrm{b}5,\mathrm{ns}7)\to X(\mathrm{ns}7)=\P^1$ is given by
\[
(X_0:\ldots:X_5)\mapsto (7X_0 - 2X_2 + 4X_3 - X_4 - 4X_5:
    -14X_0 - 7X_1 + 6X_2 - 12X_4 + 10X_4 - 9X_5).
\]
A canonical model for $X(\mathrm{b}5,\mathrm{e}7)/\phi_7w_5$ in $\P^4_{X_0,\ldots,X_4}$ is given by

\begin{align*} X(\mathrm{b}5,\mathrm{e}7)/w_5:\;& 448X_0^2 - 9X_1^2 + 9X_2^2 + 54X_2X_3 + 9X_3^2 + 112X_0X_4 + 126X_1X_4 + 7X_4^2=0,\\& 16X_0X_1 - 3X_1^2 + 3X_2^2 + 6X_2X_3 + 3X_3^2 + 2X_1X_4 + 21X_4^2=0,\\& 3X_1X_2 + 28X_0X_3 + 12X_1X_3 + 21X_2X_4 + 14X_3X_4=0
\end{align*}
and $\phi_7$ acts as
\[
\phi_7:\; (X_0:X_1:X_2:X_3:X_4)\mapsto (X_0:X_1:-X_2:-X_3:X_4).
\]
Finally, a canonical model for $X(\mathrm{b}5,\mathrm{e}7)/\phi_7w_5$ in $\P^7_{X_0,\ldots,X_7}$ is given by
\begin{align*} & 3528X_0^2 + 177X_4X_5 + 597X_5^2 - 2716X_0X_6 - 423X_1X_6 - 6365X_2X_6 - 1918X_3X_6\\&\;\;\;\;\; - 13454X_6^2 + 2842X_0X_7 + 626X_1X_7 + 9144X_2X_7 + 3010X_3X_7 + 35252X_6X_7 - 22960X_7^2=0,\\& 56X_0X_1 - 135X_4X_5 - 327X_5^2 - 140X_0X_6 - 37X_1X_6 + 2381X_2X_6 + 1890X_3X_6\\&\;\;\;\;\; + 2982X_6^2 + 910X_0X_7 + 16X_1X_7 - 3270X_2X_7 - 2758X_3X_7 - 7476X_6X_7 + 4816X_7^2=0,\\& 56X_1^2 + 1215X_4X_5 + 2835X_5^2 + 11340X_0X_6 - 715X_1X_6 \\&\;\;\;\;\;- 22389X_2X_6 - 12726X_3X_6 - 38290X_6^2 - 19530X_0X_7 + 820X_1X_7\\&\;\;\;\;\; + 30894X_2X_7 + 21546X_3X_7 + 99764X_6X_7 - 64624X_7^2=0,\\& 168X_0X_2 - 15X_4X_5 - 3X_5^2 - 56X_0X_6 - 45X_1X_6 + 89X_2X_6 + 238X_3X_6\\&\;\;\;\;\; - 238X_6^2 + 182X_0X_7 + 52X_1X_7 - 138X_2X_7 - 406X_3X_7 + 532X_6X_7 - 224X_7^2=0,\\& 56X_1X_2 - 171X_4X_5 - 315X_5^2 + 168X_0X_6 - 253X_1X_6 + 2713X_2X_6 + 2562X_3X_6\\&\;\;\;\;\; + 2086X_6^2 + 1050X_0X_7 + 260X_1X_7 - 3910X_2X_7 - 3738X_3X_7 - 5292X_6X_7 + 3584X_7^2=0,\\& 56X_2^2 + 87X_4X_5 + 147X_5^2 - 364X_0X_6 + 141X_1X_6 - 1285X_2X_6 - 1582X_3X_6\\&\;\;\;\;\; - 714X_6^2 - 266X_0X_7 - 148X_1X_7 + 1894X_2X_7 + 2170X_3X_7 + 1764X_6X_7 - 1232X_7^2=0,\\& 1764X_0X_3 - 579X_4X_5 - 1050X_5^2 + 1505X_0X_6 - 774X_1X_6 + 7009X_2X_6 + 6083X_3X_6\\&\;\;\;\;\; + 7021X_6^2 + 1456X_0X_7 + 851X_1X_7 - 9837X_2X_7 - 8582X_3X_7 - 15778X_6X_7 + 9044X_7^2=0,\\& 28X_1X_3 + 9X_4X_5 + 6X_5^2 + 357X_0X_6 - X_1X_6 + 74X_2X_6 + 161X_3X_6\\&\;\;\;\;\; - 63X_6^2 - 364X_0X_7 - 20X_1X_7 - 144X_2X_7 - 168X_3X_7 + 210X_6X_7 - 140X_7^2=0,\\& 84X_2X_3 - 15X_4X_5 - 30X_5^2 - 35X_0X_6 - 27X_1X_6 + 80X_2X_6 + 133X_3X_6\\&\;\;\;\;\; + 329X_6^2 + 140X_0X_7 + 34X_1X_7 - 66X_2X_7 - 196X_3X_7 - 854X_6X_7 + 532X_7^2=0,\\& 252X_3^2 + 12X_4X_5 - 30X_5^2 + 70X_0X_6 + 45X_1X_6 + 443X_2X_6 + 28X_3X_6\\&\;\;\;\;\; + 350X_6^2 - 196X_0X_7 - 59X_1X_7 - 639X_2X_7 + 14X_3X_7 - 728X_6X_7 + 280X_7^2=0,\\& 18X_0X_4 - 19X_0X_5 - 12X_1X_5 - 27X_2X_5 + 17X_3X_5 - 120X_4X_6\\&\;\;\;\;\; - 66X_5X_6 + 114X_4X_7 + 54X_5X_7=0,\\& 2X_1X_4 - 21X_0X_5 + 8X_1X_5 + 7X_2X_5 - 21X_3X_5 + 28X_4X_6 + 56X_5X_6 - 28X_4X_7 - 84X_5X_7=0,\\& 6X_2X_4 + 7X_0X_5 + 3X_2X_5 + 7X_3X_5=0,\\& 9X_3X_4 - 28X_0X_5 - 3X_1X_5 - 9X_2X_5 + 8X_3X_5 - 39X_4X_6 - 30X_5X_6 + 33X_4X_7 + 27X_5X_7=0,\\& 63X_4^2 + 426X_4X_5 + 690X_5^2 + 1120X_0X_6 + 207X_1X_6 - 4405X_2X_6 - 2702X_3X_6\\&\;\;\;\;\; - 7000X_6^2 - 2464X_0X_7 - 263X_1X_7 + 6057X_2X_7 + 4298X_3X_7 + 18172X_6X_7 - 11984X_7^2=0
\end{align*}    
and the remaining involution acts as
\[
w_5=\phi_7:\; (X_0:\ldots:X_7)\mapsto (X_0:X_1:X_2:X_3:-X_4:-X_5:X_6:X_7).
\]
The quotients of each of these three curves by their respective involutions give rise to the same genus 2 hyperelliptic curve
\[
X(\mathrm{b}5,\mathrm{ns}7)/w_5:\;y^2 = x^6 + 2x^5 + 
    7x^4 - 4x^3 + 3x^2 - 10x + 5.
\]
\end{theorem}

Using \texttt{Magma}'s built-in \texttt{Genus6PlaneCurveModel} function, we can transform our model for $X(\mathrm{b}5,\mathrm{ns}7)$ into a planar model. As in \cite{derickx}, this planar model has four singular points, and we can make a translation such that the singular points have coordinates $(i:0:1),\;(-i:0:1),\;(0:1/\sqrt{5}:1)$ and $(0:-1/\sqrt{5}:1)$. The planar equation thus obtained is identical to the one found by  Derickx--Najman--Siksek, showing that our model is indeed isomorphic to theirs.

%$X(\mathrm{b}5,\mathrm{e}7)$, defined as the modular curve $X_{G}$, where $G\subset \mathrm{GL}_2(\Z/35\Z)$ is definedparametrising triples $(E,\phi_5,\phi_7)$, where $E$ is an elliptic curve as defined in \cite{freitas}. This is a degree 2 cover of $X(\mathrm{b}5,\mathrm{ns}7)$, the modular curve parametrising elliptic curves with mod 5 representation inside the Borel subgroup and mod 7 representation inside the normaliser of a non-split Cartan subgroup. Being a degree 2 cover, there must be an involution $\phi_7: X(\mathrm{b}5,\mathrm{e}7)\to X(\mathrm{b}5,\mathrm{e}7)$ such that $X(\mathrm{b}5,\mathrm{ns}7)=X(\mathrm{b}5,\mathrm{e}7)/\phi_7$. On $X(\mathrm{b}5,\mathrm{e}7)$ there is another order 2 automorphism: the Atkin--Lehner involution $w_5$. Acting on the level structure for distinct primes, the two involutions commute. We thus obtain three curves:
%\[
%X(\mathrm{b}5,\mathrm{ns}7), \; X_1:=X(\mathrm{b}5,\mathrm{e}7)/w_5, \text{ and } X_2:=X(\mathrm{b}5,\mathrm{e}7)/\phi_7\circ w_5.
%\]
%These curves have genera 6, 8 and 5 respectively. The three curves have a common quotient $X(\mathrm{b}5,\mathrm{e}7)/\langle w_5,\phi_7\rangle$, which is hyperelliptic of genus 2. We note that Derickx, Najman and Siksek \cite{dns} have already computed a planar model for $X(\mathrm{b}5,\mathrm{ns}7)$ based on Le Hung's model \cite{lehung} as a fibred product.

%In this section we find canonical models for each of these three curves by computing the corresponding modular forms. We first find generators for the congruence subgroups.

\subsection{Finding generators}
In \cite{freitas}, the group $G(\mathrm{e}7)$ was defined as (\ref{ge7}) inside the normaliser of a non-split Cartan subgroup. Recall that these Cartan subgroups are only defined up to conjugacy. We work in the non-split Cartan subgroup
\[
G(\mathrm{ns}7):=\left\{ \begin{pmatrix} a & 5b \\ b & a\end{pmatrix} \in \mathrm{GL}_2(\F_7) \mid a,b\in \F_7\right \}
\]
and its normaliser $G(\mathrm{ns}7^+)$ generated by $G(\mathrm{ns}7)$ and $\begin{pmatrix} 1 & 0 \\ 0 &-1\end{pmatrix}$. Note that 5 is a generator of $\F_7^{\times}$. 

In order to determine which subgroup of $G(\mathrm{ns}7^+)$ corresponds to $G(\mathrm{e}7)$, we first find a choice-independent definition.
\begin{lemma}
\label{ge7def}
Let $G$ be the normaliser of a non-split Cartan subgroup of $\mathrm{GL}_2(\F_7)$. The corresponding subgroup $G(\mathrm{e}7)\subset G$ is the unique index 2 subgroup $H\subset G$ such that
\begin{itemize}
    \item[(i)] $H$ is not cyclic and
    \item[(ii)] $H\cap \mathrm{SL}_2(\F_7)$ is cyclic of order 8.
\end{itemize}
\end{lemma}
\begin{proof}
The group $\mathrm{GL}_2(\F_7)$ is a relatively small finite group, and we do these computations in \texttt{Magma}. For uniqueness, we can use any normaliser of a non-split Cartan $G(\mathrm{ns}7^+)$, and we use the one defined above. To verify that $G(\mathrm{e}7)$ indeed satisfies these properties, we can work with the group defined in (\ref{ge7}).
\end{proof}
We also see from the definition of $G(\mathrm{ns}7^+)$ and Lemma \ref{ge7def} that $G(\mathrm{ns}7+)$ and $G(\mathrm{e}7)$ are normalised by $J$. This means that also $\phi_7$ is normalised by $J$. Since the only element of order 2 in $\mathrm{SL}_2(\F_7)$ is $-I$, we deduce from (ii) that $-I\in G(\mathrm{e}7)$. We also check that $\mathrm{det}(G(\mathrm{e}7))=\F_7^{\times}$.

Next, we use Lemma \ref{ge7def} to find generators, and we lift them to $\Gamma_0(5)$ to obtain the following five matrices.
%Next, we define the index 2 subgroup $G_{\mathrm{e}7}\subset G_{\mathrm{ns}7+}$. Note that the non-split Cartan subgroup is only defined up to conjugacy. In \cite{freitas}, $G_{\mathrm{e}7}$ is defined using explicit generators inside the  normaliser of a different non-split Cartan subgroup. This group has two properties: the intersection $G_{\mathrm{e}7}\cap \mathrm{SL}_2(\F_7)$ is the unique cyclic subgroup $H_{\mathrm{e}7}$ of order 8 inside $G_{\mathrm{ns}7+}\cap \mathrm{SL}_2(\F_7)$. Moreover, $G_{\mathrm{e}7}$ is \emph{not} cyclic. There are two index 2 subgroup of $G_{\mathrm{ns}7+}$ whose subgroup of determinant 1 equals $H_{\mathrm{e}7}$, but only one is not cyclic. This gives us a conjugacy class representative-independent definition of $G_{\mathrm{e}7}$ as the unique non-cyclic index 2 subgroup of $G_{\mathrm{ns}7+}$ which has an 8-cyclic subgroup of determinant 1, allowing us to compute $G_{\mathrm{e}7}$ as a subgroup of our $G_{\mathrm{ns}7+}$.
Define
\[
g_0:=\begin{pmatrix} 61 & -55 \\ 10 & -9\end{pmatrix} \text{ and } \phi_7:=\begin{pmatrix} 3 & 1\\ -10&-3\end{pmatrix} \in \mathrm{SL}_2(\Z).
\]
Then $g_0$ reduces to the identity mod 5 and to a generator of  $G(\mathrm{e}7)\cap \mathrm{SL}_2(\F_7)$ mod 7, while $\phi_7$ is in $\Gamma_0(5)$ and reduces mod 7 to an element in $G(\mathrm{ns}7+)\cap \mathrm{SL}_2(\F_7)\setminus G(\mathrm{e}7)\cap \mathrm{SL}_2(\F_7)$. Next, define 
\[
B:=\begin{pmatrix} 6 & 5 \\ -5 & -4 \end{pmatrix},\; C:=B\cdot \begin{pmatrix} 4 & 0 \\ 0 & 1\end{pmatrix}\text{ and }w_5:=\begin{pmatrix} 2890 & 193 \\
 -8685 & -580 \end{pmatrix} \in \mathrm{GL}_2^+(\Q).
\]
We note that $C$ reduces into $B_0(5)$ mod 5, $\mathrm{det}(w_5)=5$ and $w_5$ reduces into $G(\mathrm{e}7)$ mod 7. So $w_5$ indeed corresponds to the Atkin--Lehner involution.  Moreover, 
\[
G(\mathrm{e}7)=\langle \overline{g}_0,\overline{C},\overline{J}\rangle \text{ and } G(\mathrm{ns}7+)=\langle \overline{g}_0,\overline{C},\overline{J},\overline{\phi}_7\rangle,
\]
where a bar denotes reduction mod 7. 

Using the built-in commands in \texttt{Sage}, we compute $S:=S_2(\Gamma_0(5\cdot 7^2)\cap \Gamma_1(7),\Q(\zeta_7))$. This space has dimension 61. We need to determine the fixed spaces
\[
S^{\langle g_0,\phi_7,C,J\rangle},\; S^{\langle g_0,w_5,C,J\rangle} \text{ and } S^{\langle g_0,w_5\cdot \phi_7,C,J\rangle}. 
\]
To this end, we first compute the fixed spaces $S^{\langle g_0,\phi_7\rangle},S^{\langle g_0,w_5\rangle}$ and $S^{\langle g_0,w_5\cdot \phi_7\rangle}$, after which we determine the $\Q$-rational structure by considering $C$ and $J$. 

\subsection{The twist orbit spaces}
We execute Step (4) of Algorithm \ref{algorithm1}, by computing $q$-expansions for the eigenforms corresponding to the bases of the real twist orbit spaces as described in Proposition \ref{basisprop}. There are two non-zero spaces of cusp forms of lower level: $S_2(\Gamma_0(49)\cap \Gamma_1(7),\Q(\zeta_7))$ (3-dimensional), and $S_2(\Gamma_0(5)\cap \Gamma_1(7),\Q(\zeta_7))$ (7-dimensional).

Define 
\[
\chi: (\Z/7\Z)^{\times}\to \CC^{\times},\;\; \chi(3)=e^{\pi i/3}.
\]
Then $\chi$ has order 6 and $\mathcal{D}_{7}=\langle \chi \rangle$. Also, $\chi(-1)=-1$, so the even characters are generated by $\chi^2$. We note all found equations have degree 2, so it suffices to take $\texttt{prec}=2(2\cdot 8-2)=28$. To save space, we here display modular forms only up to $q^{10}$. 

We consider the following newforms, all of which have trivial Nebentypus character:
\begin{align*}
   f_{49}&:=q+q^2-q^4-3q^8-3q^9+O(q^{11}) \in S_2(\Gamma_0(49)\cap \Gamma_1(7),\overline{\Q}), \\
   f_{35}&:=q+q^3-2q^4-q^5+q^6-2q^8-3q^{10}+O(q^{11})\in S_2(\Gamma_0(5)\cap \Gamma_1(7),\overline{\Q}),\\
   g_{35}&:=q+ \al q^2 -(\al+1)q^3 +(2-\al)q^4+ q^5 -4q^6, -q^7+ (\al - 4)q^8+(\al + 2)q^9\\&\;\;\;\;\;\;\;\;\;\;+ \al q^{10}+O(q^{11}) \in S_2(\Gamma_0(5)\cap \Gamma_1(7),\overline{\Q}), \text{ where }\al=(-1+\sqrt{17})/2,\\
   f_0&:=q -2q^2 -3q^3+ 2q^4 +q^5+ 6q^6 +6q^9 -2q^{10}+O(q^{11})\in S_2(\Gamma_0(5\cdot 7^2)\cap \Gamma_1(7),\overline{\Q}),\\
   f_1&:=q+\sqrt{2}q^2-(\sqrt{2}+1)q^3-q^5\\&\;\;\;\;\;\;\;\;\;\;-(\sqrt{2}+2)q^6-2\sqrt{2}q^8+2\sqrt{2}q^9-\sqrt{2}q^{10}+O(q^{11}) \in S_2(\Gamma_0(5\cdot 7^2)\cap \Gamma_1(7),\overline{\Q}),\text{ and}\\
   f_2&:=q+ (1+\sqrt{2})q^2+ (1-\sqrt{2})q^3+(2\sqrt{2}+1)q^4 +q^5-q^6+ (\sqrt{2}+3)q^8-2\sqrt{2}q^9\\&\;\;\;\;\;\;\;\;\;\;+(1+\sqrt{2})q^{10}+O(q^{11})\in S_2(\Gamma_0(5\cdot 7^2)\cap \Gamma_1(7),\overline{\Q}).
\end{align*}
When $f$ is a cusp form defined over a quadratic field, let $f^c$ be its Galois conjugate cusp form. For brevity, we write $\mathrm{pr}:=\mathrm{pr}_{245/35}$, and we omit any application of the $B_1$-operator, since it acts as the identity on $q$-expansions. This leads to the following real twist orbit spaces of modular forms:
\begin{align*}
    O_1^+&:=\mathrm{Span}\{f_{49},f_{49}\otimes \chi^2,f_{49}\otimes \chi^4,f_{49}|B_5,(f_{49}\otimes \chi^2)|B_5,(f_{49}\otimes \chi^4)|B_5\},\\
    O_2^+&:=\mathrm{Span}\{f_{35},f_{35}\otimes \chi^2,f_{35}\otimes \chi^4,f_{35}|B_7\},\\ O_3^+&:=\mathrm{Span}\{f_{35}\otimes \chi,f_{35}\otimes \chi^3,f_{35}\otimes \chi^5\},\\
    O_4^+&:=\mathrm{Span}\{g_{35},g_{35}\otimes \chi^2,g_{35}\otimes \chi^4,g_{35}|B_7\} \text{ and } O_4^{c,+}:=\mathrm{Span}\{g_{35}^c,g_{35}^c\otimes \chi^2,g_{35}^c\otimes \chi^4,g_{35}^c|B_7\},\\
    O_5^+&:=\mathrm{Span}\{g_{35}\otimes \chi,g_{35}\otimes \chi^3,g_{35}\otimes \chi^5\} \text{ and } O_5^{c,+}:=\mathrm{Span}\{g_{35}^c\otimes \chi,g_{35}^c\otimes \chi^3,g_{35}^c\otimes \chi^5\},\\
    O_6^+&:=\mathrm{Span}\{f_0,f_0\otimes \chi^2,f_0\otimes \chi^4\},\\
    O_7^+&:=\mathrm{Span}\{f_0\otimes\chi,f_0\otimes\chi^3,f_0\otimes\chi^5\},\\
    O_8^+&:=\mathrm{Span}\{f_1,f_1\otimes\chi^2,f_1\otimes\chi^4\} \text{ and }O_8^{c,+}:=\mathrm{Span}\{f_1^c,f_1^c\otimes\chi^2,f_1^c\otimes\chi^4\},\\
    O_9^+&:=\mathrm{Span}\{f_1\otimes \chi,f_1\otimes\chi^3,f_1\otimes\chi^5\} \text{ and } O_9^{c,+}:=\mathrm{Span}\{f_1^c\otimes \chi,f_1^c\otimes\chi^3,f_1^c\otimes\chi^5\},\\
    O_{10}^+&:=\mathrm{Span}\{f_2,(f_2\otimes \chi^2)|\mathrm{pr},(f_2\otimes \chi^4)|\mathrm{pr},(f_2\otimes \chi^2)|\mathrm{pr}|B_7,(f_2\otimes \chi^4)|\mathrm{pr}|B_7\} \text{ and }\\ 
    O_{10}^{c,+}&:=\mathrm{Span}\{f_2^c,(f_2^c\otimes \chi^2)|\mathrm{pr},(f_2^c\otimes \chi^4)|\mathrm{pr},(f_2^c\otimes \chi^2)|\mathrm{pr}|B_7,(f_2^c\otimes \chi^4)|\mathrm{pr}|B_7\},\\
    O_{11}^+&:=\mathrm{Span}\{f_2\otimes \chi,f_2\otimes \chi^3,f_2\otimes \chi^5\}\text { and }O_{11}^{c,+}:=\mathrm{Span}\{f_2^c\otimes \chi,f_2^c\otimes \chi^3,f_2^c\otimes \chi^5\}.
\end{align*}
By counting dimensions, we see that these are all of the real twist orbit spaces. These computations were done in the ``space of cusp forms'' in \texttt{Sage}, but note that the underlying calculations do happen in the space of modular symbols. 

For each $i\in \{1,\ldots,11\}$, denote by $V_i^+$ the sum of $O_i^+$ and, if it exists, $O_i^{c,+}$.

\subsection{Computations with modular symbols}
We ask \texttt{Sage} to compute the 122-dimensional space $H_1(X_{\Gamma_0(5\cdot 7^2)\cap \Gamma_1(7)},\Q)$. Given a basis for this space, we ask for the matrix of the operator induced by $J$, and we  determine the $+1$-eigenspace. Consider an eigenform $f\in S_2(\Gamma_0(5\cdot 7^2)\cap \Gamma_1(7),\overline{\Q})$. As $T_p$ commutes with $J$, we can use Lemma \ref{transformationlemma} to determine the modular symbol $\gamma\in H_1(\Gamma_0(5\cdot 7^2)\cap \Gamma_1(7),\overline{\Q})^+$ such that $\gamma=\gamma_f$ up to a constant multiple. In practice, it sufficed for us to compute eigenvectors in $H_1(\Gamma_0(5\cdot 7^2)\cap \Gamma_1(7),\overline{\Q})^+$ for the $T_2$- and $T_3$-operators only: when the intersection of the kernels of $T_2-\overline{a}_2(f)$ and $T_3-\overline{a}_3(f)$ on $H_1(\Gamma_0(5\cdot 7^2)\cap \Gamma_1(7),\overline{\Q})^+$ is 1-dimensional, its elements must be multiples of $\gamma_{f}$.  We thus obtain explicit eigensymbols
\[
\mu_0:=\al_0\gamma_{f_0},\; \mu_1:=\al_1\gamma_{f_1},\; \mu_2:=\al_2\gamma_{f_2},\; \mu_{35}:=\al_{35}\gamma_{f_{35}},\; \mu_{49}:=\al_{49}\gamma_{f_{49}},\; \nu_{35}:=\beta_{35}\gamma_{g_{35}},
\]
where $\beta_{35}$ and all $\al_i$ are non-zero but unknown. However, within twist orbit spaces these constants are also irrelevant.

Moreover, we recall that $R_{\chi^2}(\gamma_f)=g(\overline{\chi}^2)\gamma_{f\otimes \overline{\chi}^2}$ and $[\Gamma_0(N):\Gamma_0(M)]\gamma_{f|B_d}=\mathrm{Tr}_d^{N/L}(\gamma_f)$, where the action of $R_{\chi^2}$, $\mathrm{Tr}_7$ and $\mathrm{Tr}_5$, as well as that of $\pi_{245/35}$ and $\pi_{235/49}$, on $H_1(\Gamma_0(5\cdot 7^2)\cap \Gamma_1(7),\overline{\Q})$ can be determined explicitly. To compute the coset representatives needed for the $\mathrm{Tr}_d^{N/L}$ operators, we used an algorithm of Stein \cite[Algorithm 2.20]{stein}. We thus explicitly compute the real twist orbit spaces on the modular symbols side. 

The next step is to conjugate $g_0$, $\phi_7$ and $w_5$ by $\gamma_7$ and compute their action on $H_1(\Gamma_0(5\cdot 7^2)\cap \Gamma_1(7),\overline{\Q})$. This can be done in \texttt{Sage} via built-in functions. We intersect the different fixed-spaces with the real twist orbit spaces to obtain the desired basis of fixed modular symbols. For brevity reasons, we do not display all of these here, but we give two examples: we find that
\[
4\left(\mu_{35}-\mathrm{Tr}_7(\mu_{35})\right)+\left(\frac{6\zeta_3}{7}+\frac27\right)R_{\overline{\chi}^2}(\mu_{35})+\left(\frac{36}{343}\zeta_3-\frac{2}{343}\right)R_{\overline{\chi}^4}(\mu_{35})\]
is in $(H_{\CC}(\Gamma_0(5\cdot 7^2)\cap \Gamma_1(7))^+)^{g_0,w_5\cdot \phi_7}$ and that
\begin{align*}
\left(\frac{2}{21}\sqrt{2} + \frac{8}{21}\right)\mu_2 &+ \left(\left(\frac{10}{1029} \zeta_3 - \frac{2}{343}\right)\sqrt{2} + \frac{40}{1029}\cdot\zeta_3 -
\frac{8}{343}\right)R_{\overline{\chi}^2}(\mu_2)\\&+ \left(\left(\frac{16}{7203} \zeta_3 + \frac{10}{7203}\right)\sqrt{2} + \frac{64}{7203} \zeta_3 + \frac{40}{7203}\right)R_{\overline{\chi}^4}(\mu_2)\\&+
\left(\left(-\frac{2}{147} \zeta_3 - \frac{4}{147}\right) \sqrt{2} - \frac{20}{1029} \zeta_3 - \frac{16}{1029}\right)\frac{B_7\mathrm{pr}R_{\overline{\chi}^4}(\mu_2)}{7}\\&+ \left(\left(\frac{46}{1029} \zeta_3 +
\frac{188}{1029}\right)\sqrt{2} - \frac{68}{1029} \zeta_3 + \frac{80}{1029}\right)\frac{B_7\mathrm{pr}R_{\overline{\chi}^2}(\mu_2)}{ 7}
\end{align*}
is in $H_{\CC}(\Gamma_0(5\cdot 7^2)\cap \Gamma_1(7))^{g_0,\phi_7}$. Finally, we determine the $\CC$-linear action of $B$ on these bases of fixed modular symbols, one real twist orbit at a time. We also find a matrix for the $\Q$-linear action of $\sigma_4: \; \zeta_7\mapsto \zeta_7^4$ and $J$ on the corresponding cusp forms. We choose to translate the actions of $\sigma_4$ and $J$ to a $\Q$-linear action on modular symbols (rather than translating the action of $B$ to modular forms), and compose $\sigma_4$ and $J$ with $B$. This yields two $\Q$-linear maps on each fixed subspace of a real twist orbit space, and we determine the $\Q$-linear subspace fixed by these two maps. Translating this back to modular forms, we obtain the desired cusp forms.

\subsection{The cusp forms}
 Let $\beta_7=\zeta_7+\zeta_7^{-1}$. On each of the curves, denote by $q=e^{2\pi i\tau/7}$ a uniformiser at (the image of) $\infty$. The space $H^0(X(\mathrm{b}5,\mathrm{ns}7)/w_5,\Omega^1)$ has a basis of 1-forms with $q$-expansions $h_1(q)\frac{\mathrm{d}q}{q}$ and  $h_2(q)\frac{\mathrm{d}q}{q}$, where 
 \begin{align*}
     h_1:=&\left(10\beta_7^2 + 2\beta_7 - 16\right)q + \left(4\beta_7^2 + 12\beta_7 - 12\right)q^2 + \left(-6\beta_7^2 - 4\beta_7 + 18\right)q^3 + \left(8\beta_7^2 + 10\beta_7 - 10\right)q^5 \\&\;\;\;\;\;+ \left(-8\beta_7^2 + 4\beta_7 + 24\right)q^6 + \left(24\beta_7^2 + 16\beta_7 - 16\right)q^8 + \left(8\beta_7^2 + 24\beta_7 - 24\right)q^9 + O\left(q^{10}\right)\in V_8^+,\\h_2:=&
 \left(-6\beta_7^2 - 4\beta_7 + 4\right)q + \left(-8\beta_7^2 - 10\beta_7 + 10\right)q^2 + \left(-2\beta_7^2 - 6\beta_7 + 6\right)q^3 + \left(-2\beta_7^2 - 6\beta_7 + 6\right)q^5\\&\;\;\;\;\; + \left(2\beta_7^2 + 6\beta_7 + 8\right)q^6 + \left(-20\beta_7^2 - 4\beta_7 + 32\right)q^8 + \left(-16\beta_7^2 - 20\beta_7 + 20\right)q^9 + O\left(q^{10}\right)\in V_8^+.
 \end{align*}The space $H^0(X(\mathrm{b}5,\mathrm{ns}7),\Omega^1)$ has a basis of 1-forms $f(q)\frac{\mathrm{d}q}{q}$, where $f(q)$ is in the following set: 
\begin{align*}
&h_1,\; h_2,\in V_8^+,\\&
 \left(-2\beta_7^2 - 6\beta_7 + 6\right)q + \left(-16\beta_7^2 + 8\beta_7 + 48\right)q^2 + \left(-12\beta_7^2 - 8\beta_7 + 8\right)q^3 + \left(42\beta_7^2 + 28\beta_7 - 28\right)q^4\\&\;\;\;\;\; + \left(-4\beta_7^2 + 2\beta_7 + 12\right)q^5 + \left(2\beta_7^2 + 6\beta_7 - 6\right)q^6 + \left(-28\beta_7^2 + 28\beta_7 + 56\right)q^7\\&\;\;\;\;\; + \left(-12\beta_7^2 - 36\beta_7 + 36\right)q^8 + \left(24\beta_7^2 - 12\beta_7 - 72\right)q^9 + O\left(q^{10}\right)\in V_{10}^+,\\&
 \left(-3\beta_7^2 - 9\beta_7 + 9\right)q + \left(-10\beta_7^2 + 5\beta_7 + 30\right)q^2 + \left(3\beta_7^2 + 2\beta_7 - 2\right)q^3 + \left(21\beta_7^2 + 14\beta_7 - 14\right)q^4\\&\;\;\;\;\; + \left(-6\beta_7^2 + 3\beta_7 + 18\right)q^5 + \left(3\beta_7^2 + 9\beta_7 - 9\right)q^6 + \left(21\beta_7 + 7\right)q^7\\&\;\;\;\;\; + \left(-11\beta_7^2 - 33\beta_7 + 33\right)q^8 + \left(8\beta_7^2 - 4\beta_7 - 24\right)q^9 + O\left(q^{10}\right)\in V_{10}^+,\\&
 \left(10\beta_7^2 + 2\beta_7 - 16\right)q + \left(-4\beta_7^2 - 12\beta_7 + 12\right)q^2 + \left(-10\beta_7^2 + 12\beta_7 + 30\right)q^3 + \left(-8\beta_7^2 - 10\beta_7 + 10\right)q^5\\&\;\;\;\;\; + \left(32\beta_7^2 + 12\beta_7 - 40\right)q^6 + \left(-24\beta_7^2 - 16\beta_7 + 16\right)q^8 + \left(-8\beta_7^2 - 24\beta_7 + 24\right)q^9 + O\left(q^{10}\right)\in V_{9}^+,\\&
 \left(6\beta_7^2 + 4\beta_7 - 4\right)q + \left(-8\beta_7^2 - 10\beta_7 + 10\right)q^2 + \left(-6\beta_7^2 + 10\beta_7 + 18\right)q^3 + \left(-2\beta_7^2 - 6\beta_7 + 6\right)q^5\\&\;\;\;\;\; + \left(22\beta_7^2 + 10\beta_7 - 24\right)q^6 + \left(-20\beta_7^2 - 4\beta_7 + 32\right)q^8 + \left(-16\beta_7^2 - 20\beta_7 + 20\right)q^9 + O\left(q^{10}\right)\in V_{9}^+.
\end{align*}
The space $H^0(X(\mathrm{b}5,\mathrm{e}7)/w_5\phi_7),\Omega^1)$ has a basis of 1-forms $f(q)\frac{\mathrm{d}q}{q}$, where $f(q)$ is one of:
\begin{align*}
&\left(\beta_7^2 - \beta_7 - 2\right)q + \left(-2\beta_7^2 - \beta_7 + 3\right)q^2 + \left(-\beta_7^2 - 2\beta_7 + 1\right)q^4 + \left(10\beta_7^2 + 5\beta_7 - 15\right)q^5\\&\;\;\;\;\; + \left(-3\beta_7^2 + 3\beta_7 + 6\right)q^8 + \left(6\beta_7^2 + 3\beta_7 - 9\right)q^9 + O\left(q^{10}\right)\in V_1^+,\\&
 \left(-14\beta_7^2 - 14\beta_7 + 28\right)q + \left(-7\beta_7 - 7\right)q^2 + \left(-7\beta_7^2 + 7\right)q^3\\&\;\;\;\;\; + \left(35\beta_7^2 - 35\right)q^4 + \left(14\beta_7 + 14\right)q^5 + \left(56\beta_7^2 + 56\beta_7 - 112\right)q^6\\&\;\;\;\;\; + 56q^7 + \left(63\beta_7^2 + 63\beta_7 - 126\right)q^8 + \left(21\beta_7 + 21\right)q^9 + O\left(q^{10}\right)\in V_4^+,\\&
 \left(7\beta_7 + 7\right)q^2 + \left(-7\beta_7^2 + 7\right)q^3 + \left(-7\beta_7^2 + 7\right)q^4\\&\;\;\;\;\; + \left(-7\beta_7^2 - 7\beta_7 + 14\right)q^8 + \left(7\beta_7 + 7\right)q^9 + O\left(q^{10}\right)\in V_4^+,\\&
 \left(-\beta_7^2 + \beta_7 + 2\right)q + \left(-4\beta_7^2 - 2\beta_7 + 6\right)q^2 + \left(3\beta_7^2 + 6\beta_7 - 3\right)q^3\\&\;\;\;\;\; + \left(-2\beta_7^2 - 4\beta_7 + 2\right)q^4 + \left(2\beta_7^2 + \beta_7 - 3\right)q^5 + \left(-6\beta_7^2 + 6\beta_7 + 12\right)q^6\\&\;\;\;\;\; + \left(12\beta_7^2 + 6\beta_7 - 18\right)q^9 + O\left(q^{10}\right)\in V_6^+,\\&h_1,\; h_2\in V_8^+,\\&
 \left(2\beta_7^2 + 2\beta_7 + 2\right)q + \left(-4\beta_7 + 8\right)q^2 + \left(-6\beta_7^2 + 24\right)q^4\\&\;\;\;\;\; + \left(-2\beta_7 + 4\right)q^5 + \left(-2\beta_7^2 - 2\beta_7 - 2\right)q^6 + \left(12\beta_7^2 - 20\right)q^7\\&\;\;\;\;\; + \left(8\beta_7^2 + 8\beta_7 + 8\right)q^8 + \left(4\beta_7 - 8\right)q^9 + O\left(q^{10}\right)\in V_{10}^+,\\&
 \left(\beta_7^2 + \beta_7 + 1\right)q + \left(-3\beta_7 + 6\right)q^2 + \left(\beta_7^2 - 4\right)q^3\\&\;\;\;\;\; + \left(-5\beta_7^2 + 20\right)q^4 + \left(-\beta_7 + 2\right)q^5 + \left(-\beta_7^2 - \beta_7 - 1\right)q^6\\&\;\;\;\;\; + \left(12\beta_7^2 + 3\beta_7 - 19\right)q^7 + \left(5\beta_7^2 + 5\beta_7 + 5\right)q^8 + \left(4\beta_7 - 8\right)q^9 + O\left(q^{10}\right)\in V_{10}^+.
\end{align*}

Finally, the space $H^0(X(\mathrm{b}5,\mathrm{e}7)/w_5),\Omega^1)$ has a basis of 1-forms $f(q)\frac{\mathrm{d}q}{q}$, where $f(q)$ is in the set

\begin{align*}
&\left(-\beta_7^2 + \beta_7 + 2\right)q + \left(2\beta_7^2 + \beta_7 - 3\right)q^2 + \left(\beta_7^2 + 2\beta_7 - 1\right)q^4 + \left(10\beta_7^2 + 5\beta_7 - 15\right)q^5\\&\;\;\;\;\; + \left(3\beta_7^2 - 3\beta_7 - 6\right)q^8 + \left(-6\beta_7^2 - 3\beta_7 + 9\right)q^9 + O\left(q^{10}\right)\in V_1^+,\\&
 \left(-7\beta_7^2 - 7\beta_7 + 14\right)q + \left(7\beta_7^2 - 7\right)q^3 + \left(-14\beta_7^2 + 14\right)q^4\\&\;\;\;\;\; + \left(-7\beta_7 - 7\right)q^5 - 28q^7 + \left(-14\beta_7 - 14\right)q^9 + O\left(q^{10}\right)\in V_2^+,\\& h_1,\; h_2\in V_8^+,\\&
 \left(-\beta_7^2 + \beta_7 + 2\right)q + \left(-4\beta_7^2 - 2\beta_7 + 6\right)q^2 + \left(-3\beta_7^2 - 6\beta_7 + 3\right)q^3\\&\;\;\;\;\; + \left(-2\beta_7^2 - 4\beta_7 + 2\right)q^4 + \left(-2\beta_7^2 - \beta_7 + 3\right)q^5 + \left(6\beta_7^2 - 6\beta_7 - 12\right)q^6\\&\;\;\;\;\; + \left(12\beta_7^2 + 6\beta_7 - 18\right)q^9 + O\left(q^{10}\right)\in V_7^+.
\end{align*}
For each of the final three curves, these cusp forms satisfy the equations as displayed in Theorem \ref{modelthm}, where the variable $X_i$ corresponds to the $(i+1)$th displayed cusp form. 

\subsection{The map $X(\mathrm{b}5,\mathrm{ns}7)\to X(\mathrm{ns}7)$}
To determine $X(\mathrm{b}5,\mathrm{ns}7)\to X(\mathrm{ns}7)$, we first use Chen's $j$-map (\ref{chenj}) to find the $q$-expansion of a Hauptmodul $n_7$ on $X(\mathrm{ns}7)$, by solving the equation
\[
j(q^7)=j_{\mathrm{ns}7}(n_7(q)).
\]
(Note the 7th power of $q$ because we defined $q$ as $e^{2\pi i \tau/7}$.) This yields
\begin{align*}
\eta_7(q)&=-\beta_7^2 - \beta_7 + 1+
    (-4\beta_7^2 - \beta_7 + 11)q+
    (-18\beta_7^2 - 11\beta_7 + 38)q^2\\&+
    (-53\beta_7^2 - 26\beta_7 + 124)q^3+
    (-171\beta_7^2 - 102\beta_7 + 370)q^4+O(q^5).
\end{align*}. Denote by $f_0,\ldots,f_5$ the computed basis for $H^0(X(\mathrm{b}5,\mathrm{ns}7),\Omega^1)$, in the same order as displayed above. It is a linear algebra computation to find homogeneous polynomials $p,r\in \Q[X_0,\ldots,X_5]$ such that 
\[
r(f_0(q),\ldots,f_5(q))n_7(q)=p(f_0(q),\ldots,f_5(q),
\]
at least up to $O(q^{200})$. 
The map $X(\mathrm{b}5,\mathrm{ns}7)\to X(\mathrm{ns}7)$ is then given by $\pi:(X_0:\ldots:X_5)\mapsto (p(X_0,\ldots,X_5):r(X_0,\ldots,X_5))$.

We verify correctness as follows. We compute that $\pi$ has degree 6. The correct map $\pi'$ to $X(\mathrm{ns}7)$ also has degree 6. Viewing $\pi,\pi'$ as elements of the function field of $X(\mathrm{b}5,\mathrm{ns}7)$, their difference $\pi-\pi'$ thus has polar degree at most 12. Hence a precision of 13 suffices. 
\bibliographystyle{amsplain}
\bibliography{refs}

\providecommand{\bysame}{\leavevmode\hbox to3em{\hrulefill}\thinspace}
\providecommand{\MR}{\relax\ifhmode\unskip\space\fi MR }
\providecommand{\MRhref}[2]{%
  \href{http://www.ams.org/mathscinet-getitem?mr=#1}{#2}
}
\providecommand{\href}[2]{#2}
\begin{thebibliography}{10}

\bibitem{assaf}
Eran Assaf, \emph{Computing classical modular forms for arbitrary congruence
  subgroups}, arXiv:2002.07212v1 (2020).

\bibitem{atkinlehner}
A.~O.~L. Atkin and J.~Lehner, \emph{Hecke operators on {$\Gamma _{0}(m)$}},
  Math. Ann. \textbf{185} (1970), 134--160. \MR{268123}

\bibitem{atkinli}
A.~O.~L. Atkin and Wen Ch'ing~Winnie Li, \emph{Twists of newforms and
  pseudo-eigenvalues of {$W$}-operators}, Invent. Math. \textbf{48} (1978),
  no.~3, 221--243. \MR{508986}

\bibitem{annotatedsagecode}
Barinder~S. Banwait and John~E. Cremona, \emph{Computing the modular curves
  ${X}_s(13)$, ${X}_{ns}(13)$ and ${X}_{A_4}(13)$ using modular symbols in
  {S}age (annotated {S}age code)}, ancillary file to arXiv:1306.6818, see also
  \href{https://github.com/JohnCremona/X13/blob/main/X13.ipynb}{https://github.com/JohnCremona/X13/blob/main/X13.ipynb}.

\bibitem{banwait}
\bysame, \emph{Tetrahedral elliptic curves and the local-global principle for
  isogenies}, Algebra Number Theory \textbf{8} (2014), no.~5, 1201--1229.
  \MR{3263141}

\bibitem{baran}
Burcu Baran, \emph{Normalizers of non-split {C}artan subgroups, modular curves,
  and the class number one problem}, J. Number Theory \textbf{130} (2010),
  no.~12, 2753--2772. \MR{2684496}

\bibitem{baran2}
\bysame, \emph{An exceptional isomorphism between modular curves of level 13},
  J. Number Theory \textbf{145} (2014), 273--300. \MR{3253304}

\bibitem{box2}
Josha Box, \emph{Elliptic curves over quartic fields not containing $\sqrt{5}$
  are modular}, in preparation.

\bibitem{bruin}
Peter Bruin and Filip Najman, \emph{Hyperelliptic modular curves {$X_0(n)$} and
  isogenies of elliptic curves over quadratic fields}, LMS J. Comput. Math.
  \textbf{18} (2015), no.~1, 578--602. \MR{3389884}

\bibitem{brunault}
Fran\c{c}ois Brunault and Michael Neururer, \emph{Fourier expansions at cusps},
  The Ramanujan Journal (2019).

\bibitem{chen}
Imin Chen, \emph{The jacobian of modular curves associated to cartan
  subgroups}, 1996, Thesis (Ph.D.)--Oxford University.

\bibitem{conrad}
Brian Conrad, Fred Diamond, and Richard Taylor, \emph{Modularity of certain
  potentially {B}arsotti-{T}ate {G}alois representations}, J. Amer. Math. Soc.
  \textbf{12} (1999), no.~2, 521--567. \MR{1639612}

\bibitem{conway}
J.~H. Conway and S.~P. Norton, \emph{Monstrous moonshine}, Bull. London Math.
  Soc. \textbf{11} (1979), no.~3, 308--339. \MR{554399}

\bibitem{cremona}
J.~E. Cremona, \emph{Algorithms for modular elliptic curves}, second ed.,
  Cambridge University Press, Cambridge, 1997. \MR{1628193}

\bibitem{deligne}
P.~Deligne and M.~Rapoport, \emph{Les sch\'{e}mas de modules de courbes
  elliptiques},  (1973), 143--316. Lecture Notes in Math., Vol. 349.
  \MR{0337993}

\bibitem{derickx}
Maarten Derickx, Filip Najman, and Samir Siksek, \emph{Elliptic curves over
  totally real cubic fields are modular}, Algebra Number Theory \textbf{14}
  (2020), no.~7, 1791--1800. \MR{4150250}

\bibitem{diamondshurman}
Fred Diamond and Jerry Shurman, \emph{A first course in modular forms},
  Graduate Texts in Mathematics, vol. 228, Springer-Verlag, New York, 2005.
  \MR{2112196}

\bibitem{freitas}
Nuno Freitas, Bao~V. Le~Hung, and Samir Siksek, \emph{Elliptic curves over real
  quadratic fields are modular}, Invent. Math. \textbf{201} (2015), no.~1,
  159--206. \MR{3359051}

\bibitem{galbraith}
Steven~D. Galbraith, \emph{Rational points on {$X^+_0(N)$} and quadratic {$\Bbb
  Q$}-curves}, J. Th\'{e}or. Nombres Bordeaux \textbf{14} (2002), no.~1,
  205--219. \MR{1925998}

\bibitem{katzmazur}
Nicholas~M. Katz and Barry Mazur, \emph{Arithmetic moduli of elliptic curves},
  Annals of Mathematics Studies, vol. 108, Princeton University Press,
  Princeton, NJ, 1985. \MR{772569}

\bibitem{lehung}
Bao~Viet Le~Hung, \emph{Modularity of some elliptic curves over totally real
  fields}, ProQuest LLC, Ann Arbor, MI, 2014, Thesis (Ph.D.)--Harvard
  University. \MR{3251352}

\bibitem{serre}
Jean-Pierre Serre, \emph{Propri\'{e}t\'{e}s galoisiennes des points d'ordre
  fini des courbes elliptiques}, Invent. Math. \textbf{15} (1972), no.~4,
  259--331. \MR{387283}

\bibitem{shimura}
Goro Shimura, \emph{Introduction to the arithmetic theory of automorphic
  functions}, Publications of the Mathematical Society of Japan, vol.~11,
  Princeton University Press, Princeton, NJ, 1994, Reprint of the 1971
  original, Kan\^{o} Memorial Lectures, 1. \MR{1291394}

\bibitem{stein}
William~Arthur Stein, \emph{Explicit approaches to modular abelian varieties},
  ProQuest LLC, Ann Arbor, MI, 2000, Thesis (Ph.D.)--University of California,
  Berkeley. \MR{2701042}

\bibitem{petri}
Karl-Otto St\"{o}hr and Paulo Viana, \emph{A variant of {P}etri's analysis of
  the canonical ideal of an algebraic curve}, Manuscripta Math. \textbf{61}
  (1988), no.~2, 223--248. \MR{943539}

\bibitem{tingley}
D.J. Tingley, \emph{Elliptic curves uniformized by modular functions}, DPhil
  thesis, University of Oxford, 1975.

\bibitem{zywina}
David Zywina, \emph{Computing actions on cusp forms}, arXiv:2001.07270.

\end{thebibliography}

\centering 
\vspace{.3cm}
\textsc{Author e-mail address: joshabox@msn.com}
\end{document}